\documentclass[11pt]{article}
\usepackage[utf8]{inputenc}
\usepackage[margin=1in]{geometry}
\usepackage{amssymb, amsmath, amsthm}
\usepackage{mathtools}
\usepackage{graphicx}
\usepackage{enumitem}
\usepackage{tikz}
\usetikzlibrary{arrows.meta,calc,decorations.markings}
\usepackage{mathrsfs}
\usepackage{comment}
\usepackage[T1]{fontenc}
\usepackage{hyperref}
\usepackage{stmaryrd}

\theoremstyle{plain}
\newtheorem{lemma}{Lemma}[section]

\newtheorem{proposition-definition}[lemma]{Proposition-Definition}

\newtheorem{theorem}[lemma]{Theorem}
\newtheorem{proposition}[lemma]{Proposition}

\newtheorem{definition}[lemma]{Definition}

\newtheorem{theorem-definition}[lemma]{Theorem-Definition}
\newtheorem{remark}[lemma]{Remark}
\newtheorem{example}[lemma]{Example}

\renewcommand{\P}{\mathcal P}

\newcommand{\D}{\mathcal{ D}}

\newcommand{\Z}{\mathbb Z}

\newcommand{\R}{\mathbb R}
\newcommand{\C}{\mathbb C}
\newcommand{\N}{\mathbb N}
\newcommand{\E}{\mathbb E}
\newcommand{\Y}{\mathbb Y}

\newcommand{\vecx}{\vec{x}}
\newcommand{\MM}{\mathbf M}
\newcommand{\pmu}{\pmb{\mu}}

\newcommand{\M}{\mathcal M}

\newcommand{\la}{\lambda}
\renewcommand{\d}{\textrm d}
\DeclareMathOperator{\Cov}{Cov}
\newcommand{\WBPolyBlockText}[4]{%
    \draw[thick,fill=white] #1;
    \node at (#2,#3) {$#4$};
}

\newcommand{\WBxA}{1.20}
\newcommand{\WBxB}{2.55}
\newcommand{\WBxM}{3.85}
\newcommand{\WBxTOne}{4.45}
\newcommand{\WBxTTwo}{5.15}
\newcommand{\WBxTau}{5.70}
\newcommand{\WBxAfterTau}{5.98}
\newcommand{\WBxPost}{6.38}
\newcommand{\WBxEnd}{7.25}

\newcommand{\WBC}{2.08}
\newcommand{\WBLow}{1.25}
\newcommand{\WBComp}{0.83}
\newcommand{\WBMidTop}{2.65}
\newcommand{\WBPreTop}{2.60}

\newcommand{\WBPanelBack}{%
    \draw[thick] (0,0) -- (\WBxEnd,0);
}

\newcommand{\WBPanelFront}{%
    \draw[thick]
        (0.00,0.00)
        .. controls (0.24,0.46) and (0.82,2.20) .. (1.20,2.76)
        .. controls (1.58,3.22) and (2.10,2.95)
            .. (\WBxB,\WBPreTop)
        .. controls (3.00,2.42) and (3.43,2.10)
            .. (\WBxM,\WBC)
        .. controls (4.13,2.18) and (4.25,2.69)
            .. (\WBxTOne,2.78)
        .. controls (4.72,2.88) and (4.98,2.73)
            .. (\WBxTTwo,\WBMidTop)
        .. controls (5.38,2.52) and (5.52,2.20)
            .. (\WBxTau,\WBC)
        .. controls (5.91,1.82) and (6.17,1.48)
            .. (\WBxPost,\WBLow)
        .. controls (6.62,1.00) and (7.03,0.24)
            .. (\WBxEnd,0.00);

    \draw[thick] (\WBxM,0) -- (\WBxM,4.05);

    \draw[thick] (\WBxTau,0) -- (\WBxTau,0.10);

    \node at (1.10,3.28) {$\mathcal D_1$};
    \node at (6.15,3.18) {$\WBRightDomainLabel$};

    \node[anchor=north] at (\WBxM,-0.04) {$m$};
    \node[anchor=north] at (\WBxTau,-0.04) {$\tau$};
}

\newcommand{\WBRect}[4]{%
    \draw[thick,fill=white] (#1,#3) rectangle (#2,#4);
}

\newcommand{\WBBlock}[5]{%
    \WBRect{#1}{#2}{#3}{#4}%
    \node at ($(#1,#3)!0.5!(#2,#4)$) {$#5$};
}

\newcommand{\WBBlockXiRight}[4]{%
    \WBRect{#1}{#2}{#3}{#4}%
    \node[anchor=west] at ($(\WBxM,#3)!0.5!(\WBxM,#4)+(0.14,0)$) {$x_i$};
}

\newcommand{\WBPanelTag}[1]{%
    \node[anchor=north] at ({\WBxEnd/2},-0.58) {#1};
}

\long\def\WBPair#1#2#3#4{%
\begin{center}
\noindent\resizebox{0.6\textwidth}{!}{%
\begin{tikzpicture}[
    scale=0.95,
    line cap=round,
    line join=round,
    every node/.style={font=\large}
]

\begin{scope}[xshift=0cm]
    \def\WBRightDomainLabel{\mathcal D_2}
    \WBPanelBack
    #1
    \WBPanelFront
    \WBPanelTag{#3}
\end{scope}

\draw[thick,<->,>=Stealth]
    (7.78,1.55) -- (9.03,1.55);

\begin{scope}[xshift=9.35cm]
    \def\WBRightDomainLabel{\mathcal D_2}
    \WBPanelBack
    #2
    \WBPanelFront
    \WBPanelTag{#4}
\end{scope}

\end{tikzpicture}%
}
\end{center}
}

\newcommand{\WBxTail}{6.855}
\newcommand{\WBSmall}{0.57}

\newcommand{\WBMark}[2]{%
    \draw[thick] (#1,0) -- (#1,0.10);
    \node[anchor=north,font=\normalsize] at (#1,-0.20) {$#2$};
}

\newcommand{\WBTriplePanelTag}[1]{%
    \node[anchor=north] at ({\WBxEnd/2},-1.02) {#1};
}

\long\def\WBTriple#1#2#3#4#5#6{%
\begin{center}
\noindent\resizebox{0.8\textwidth}{!}{%
\begin{tikzpicture}[
    scale=0.95,
    line cap=round,
    line join=round,
    every node/.style={font=\large}
]

\begin{scope}[xshift=0cm]
    \def\WBRightDomainLabel{\mathcal D_2}
    \WBPanelBack
    #1
    \WBPanelFront
    \WBTriplePanelTag{#4}
\end{scope}

\draw[thick,<->,>=Stealth]
    (7.52,1.55) -- (8.17,1.55);

\begin{scope}[xshift=8.45cm]
    \def\WBRightDomainLabel{\mathcal D_2}
    \WBPanelBack
    #2
    \WBPanelFront
    \WBTriplePanelTag{#5}
\end{scope}

\draw[thick,<->,>=Stealth]
    (15.97,1.55) -- (16.62,1.55);

\begin{scope}[xshift=16.90cm]
    \def\WBRightDomainLabel{\mathcal D_2}
    \WBPanelBack
    #3
    \WBPanelFront
    \WBTriplePanelTag{#6}
\end{scope}

\end{tikzpicture}%
}
\end{center}
}

\title{Central Limit Theorem for continuous $N$-particle systems in fixed- and high-temperature: single level}

\date{}

\author{Jiaming Xu\thanks{Department of Mathematics, The Ohio State University,
Columbus, OH 43210, USA. Email:
\href{mailto:jxu0800@gmail.com}{\texttt{jxu0800@gmail.com}}}}

\newtheoremstyle{named}{}{}{\itshape}{}{\bfseries}{.}{.5em}{\thmnote{#3}#1}
\theoremstyle{named}

\begin{document}

\maketitle

\begin{abstract}
We characterize central limit theorems for polynomial linear statistics of
continuous $N$-particle systems on $\mathbb R$, in both the fixed-temperature
regime $\theta=\beta/2>0$ and the high-temperature regime
$N\theta\to\gamma>0$. In each regime, we give necessary and sufficient
conditions in terms of the coefficients of the logarithmic Bessel generating
function, together with an explicit combinatorial formula for the limiting
covariance. At fixed temperature, this formula recovers the classical
second-order free moment--cumulant relation, up to the factor $1/\theta$;
at high temperature, it yields new second-order free $\gamma$-cumulants
and their covariance relations. The proof combines moment extraction through Dunkl operators with a topological expansion of the logarithmic
multivariate Bessel function indexed by ribbon graphs on  surfaces. This expansion gives a topological interpretation
of the free-cumulants: only genus zero contributes at fixed temperature,
whereas arbitrarily high genera contribute at high temperature. As applications, we prove that $\beta$-addition
preserves the CLT in both regimes and determine its limiting covariance
from the input free cumulants, and establish a suitably rescaled
transition from the high-temperature covariance to its fixed-temperature
counterpart as $\gamma\to\infty$.
\end{abstract}

\tableofcontents
\section{Introduction}

\subsection{Overview}

Consider a basic problem for a system of $N$ particles
on $\mathbb R$: describing its global limit and the fluctuations around
that limit as $N\to\infty$. The law of large numbers (LLN) characterizes
the limiting empirical distribution of the particles, while the central
limit theorem (CLT) describes the Gaussian fluctuations of global
observables around their means. Such a particle system can be viewed as
the spectrum of a Hermitian random matrix, so that the same problem asks
for the limiting eigenvalue distribution and the fluctuations of traces
of test functions of the matrix.
Throughout this paper, we restrict ourselves to polynomial test functions, so that the LLN and CLTs are formulated in terms of joint moments.

This question has two natural extensions. The first is to vary the strength
of the interaction between particles. The logarithmic repulsion in the
classical invariant matrix ensembles leads to the notion of a
\emph{$\beta$-ensemble}, or log-gas, with arbitrary inverse temperature
$\beta>0$, extending the classical values $\beta=1,2,4$. Its joint
density on $\mathbb R^N$ takes the form
\[
\frac{1}{Z_{N,\beta,V}}
\exp\!\left(-\frac{\beta N}{2}\sum_{i=1}^N V(a_i)\right)
\prod_{1\le i<j\le N}|a_i-a_j|^\beta,
\]
where $V$ is a confining potential and $Z_{N,\beta,V}$ is the normalizing
constant; see Forrester \cite{For10} for a systematic account.
Starting with Johansson
\cite{Joh98}, a substantial theory of global fluctuations has developed;
see, among many others, \cite{DE06,BG13,Shc13,BLS18,LLW19}.
Fluctuations at mesoscopic scales have also been studied extensively,
including for general and circular $\beta$-ensembles \cite{BL18,Lam21};
see also \cite{BMP22} for optimal local laws and logarithmically correlated
Gaussian fluctuations.

The second extension is to consider addition: how do the global limit and
the fluctuations behave when two independent random matrices are added?
{The classical result of Voiculescu \cite{Voi91} identifies free additive
convolution as the operation governing the limiting spectra of broad
classes of independent matrices, establishing a fundamental connection
between random matrices and free probability. First-order free cumulants
linearize this operation and provide the free-probabilistic analogue of
classical cumulants; see \cite{NS} for a systematic account.} At the fluctuation level, Mingo and Speicher developed
second-order freeness in \cite{MS06}, with Gaussian and Wishart matrices as fundamental
examples. To characterize the  covariances of the Gaussian fluctuation
under the addition, Collins, Mingo, \'{S}niady, and Speicher
\cite{CMSS07} introduced the second-order free cumulants and gave the explicit
covariance--cumulant relations; see also \cite{MP13,MS}. In this way, second-order free
probability extends the role of free convolution from limiting spectral
measures to their fluctuations, and makes it possible to study the latter
through quantities that are additive under matrix addition.

For $\beta$-ensembles, it is also natural to let the interaction weaken
as the number of particles grows, and to ask how the fixed-temperature
theory changes as one approaches the  case $\beta=0$, where particles are independent.
This leads to the \emph{high-temperature regime}, where $\theta=\beta/2$
tends to zero with $N\theta\to\gamma>0$. Varying $\gamma$ gives an
interpolation between these two situations, after suitable rescaling;
already for the Gaussian ensemble, the limiting empirical measures
interpolate between the Gaussian and semicircle laws \cite{ABG12}.
At the level of fluctuations, the change is equally substantial:
normalized global statistics of the classical ensembles fluctuate on the
$N^{-1/2}$ scale rather than the fixed-temperature $N^{-1}$ scale.
High-temperature CLTs have been established for Gaussian ensembles and
more general log-gases, see e.g \cite{NT18,DGM24,MM25}. From the additive
viewpoint, Benaych-Georges, Cuenca, and Gorin \cite{BCG} introduced
$\gamma$-cumulants and a $\gamma$-convolution interpolating between
classical and free convolution. Their LLN theory therefore raises the
question of whether a similar interpolation exists for the CLT, and what
should replace the classical second-order free cumulants in this regime.

The present paper brings together these perspectives by giving a
characterization of the CLT for general continuous $N$-particle systems
in both the fixed- and high-temperature regimes. Under the analytic
assumptions stated below, we show that the CLT is equivalent to suitable
asymptotic conditions on the coefficients of the logarithmic Bessel
generating function, and express the limiting covariance explicitly in
terms of the corresponding free cumulants. A principal application is
\emph{$\beta$-addition}, a binary operation on $N$-particle systems that
extends the eigenvalue law of the sum of independent invariant matrices
beyond the classical values of $\beta$, see \cite{GM,BCG}. We prove that if
both inputs satisfy a CLT in either regime, then their $\beta$-addition
satisfies a CLT in the same regime, and its limiting covariance is
explicitly determined by the first- and second-order free cumulants of
the inputs.

An important feature of this characterization is that it does not assume
a log-gas density, which makes it suitable for studying operations such
as $\beta$-addition whose output need not remain in the usual class of
log-gases. Much of the fluctuation theory for general $\beta$-ensembles
has been developed through analytic approaches, including loop
equations, master operators, and Stein's method, see \cite{Joh98,BG13,BLS18,LLW19,DGM24}. Here, by working with
the Bessel generating function, we obtain a criterion that applies beyond
prescribed log-gas densities and is compatible with addition, while also
giving a combinatorial description of the covariance through cumulants
and weighted lattice paths.

From the free-probabilistic viewpoint, our starting point is symmetric
functions and Lukasiewicz paths, rather than an operator-algebraic model
or annular non-crossing partitions, used in the classical literature. The resulting combinatorics gives a
different description of the same second-order structure at fixed
temperature: we show that our covariance formula agrees, up to the
factor $1/\theta$, with the classical formula of \cite{CMSS07}.
Its usefulness becomes more apparent at high temperature, where it
provides new second-order free $\gamma$-cumulants and their covariance
relations, extending the first-order theory of \cite{BCG}. While a
high-temperature analogue is not immediate from the classical
descriptions, it arises naturally from
the same symmetric-function and lattice-path formulation.

These results also connect free cumulants with the idea of \emph{topological expansions}.  \cite{BGF20,BCGFLS21} interpreted the
first- and second-order free moment--cumulant relations through ordinary
and fully simple maps, relating their generating series by monotone
Hurwitz numbers, which count the number of {branched coverings of orientable closed surfaces}. In a different direction, \cite{GGPN14} obtained an expansion  of the HCIZ integral in terms of {monotone Hurwitz numbers}. This is directly related to our
setting: the multivariate Bessel function we use in this text specializes to the HCIZ integral
at $\theta=1$ and extends it to general $\theta>0$. Our work, together with the
companion LLN paper \cite{CX}, gives a symmetric-function description of
free cumulants through the Bessel topological expansion, derived from
a combinatorial expansion of ribbon graphs on surfaces in \cite{CD}.
It also reveals a distinction between the two temperature regimes:
genus zero governs the fixed-temperature limits, whereas contributions
of arbitrary genus enter at high temperature. We discuss this
perspective and the role of the expansion in Section~\ref{sec:methodology}.

\subsection{Main results}

The main object of this work has a natural representation-theoretic origin.
For $\theta=1/2,1,2$, let $G_\theta(N)$ be the orthogonal group $O(N)$,
the unitary group $U(N)$, and the compact symplectic group $Sp(N)$,
respectively.  For real $N$-tuples $a$ and
$\vec x$, write $A=\operatorname{diag}(a)$ and
$X=\operatorname{diag}(\vec x)$, and consider the orbital integral
\[
    B_a(\vec x;\theta)
    =\int_{G_\theta(N)}
      \exp\!\bigl(\operatorname{tr}(AUXU^*)\bigr)\,\d U,
\]
where $\d U$ is normalized Haar measure.\footnote{In the quaternionic
case, the trace is one half of the trace in the standard complex
representation.}
For $\theta=1$, this is the Harish-Chandra--Itzykson--Zuber (HCIZ)
integral \cite{HC57,IZ80}; the other two cases are its orthogonal and
symplectic analogues.  Replacing $\vec x$ by $\mathrm i\vec x$ gives the
Fourier kernel averaged over a matrix orbit.  Such integrals are central
to harmonic analysis and representation theory, arising as spherical
functions on symmetric spaces of Euclidean type \cite{Hel00}, and play an
important role in matrix models of mathematical physics \cite{IZ80}.

The integral is symmetric separately in $a$ and $\vec x$, and is unchanged
when the two tuples are exchanged.  It admits an expansion in symmetric
polynomials parameterized by $\theta$ \cite{Ok-Ol}, the so-called Jack polynomial, which naturally
extends it to every $\theta>0$.  The resulting
\emph{multivariate Bessel function} $B_a(\vec x;\theta)$ is entire in
$(a,\vec x)\in\C^{2N}$ and normalized by $B_a(0^N;\theta)=1$.
For more details of its analytic properties, see the surveys
\cite{Ros03,Ank17} and Section~\ref{sec:mbf} below.

These functions provide a powerful characterization of a class of continuous particle systems in integrable probability,
notably for the $\beta$-corners processes \cite{Cue21} and $\beta$-addition
\cite{BCG}. The Bessel-function approach has also been used to establish global
scaling limits \cite{BCG,Yao21,X,CX} and local edge limits
\cite{KX24,GXZ24}.
Their large-$N$ asymptotics also enter asymptotic
representation theory and the study of branching-graph boundaries:
examples include the classification of invariant measures on infinite
Hermitian matrices \cite{OV96} and its extension to the orbital beta
process \cite{AN21}.

We consider the following class of $N$-particle objects. For each $N$, let $\mu_N$ be either an exponentially decaying probability measure on ordered real $N$-tuples, or a compactly supported real symmetric generalized function on $\mathbb R^N$ with
\[
    \langle\mu_N,1\rangle=N!.
\]
In the first case, $\E_{\mu_N}$ is the usual expectation; in the second, it denotes the normalized pairing
\[
    \E_{\mu_N}[f]:=\frac{1}{N!}\langle\mu_N,f\rangle.
\]
Thus $\E_{\mu_N}[1]=1$ in either case. For a given $\mu_N$, we define its \emph{Bessel generating function} as
\[
    G_{N,\theta}(\vec x)
    :=G_\theta(\vec x;\mu_N)
    :=\E_{\mu_N}[B_a(\vec x;\theta)].
\]
The BGF is a natural analog of the characteristic function/Laplace transform for the $N$-particle systems on $\R^N$, and captures all information about the distribution $\mu_N$. While the moments of a random variable can be extracted by derivatives, due to the fact that the Fourier kernel $\exp(tx)$ diagonalizes $\frac{d}{dt}$, the Bessel function is also eigenfunction of certain difference-differential operators, see Section \ref{sec:dunkl}.
Moreover, there is a collection of additive quantities of a random variable, namely the (classical) cumulants, which are bijective to the moments and are encoded by the Taylor coefficients of
the logarithm of its Laplace transform. This motivates our study of the logarithmic BGF.
Since $G_{N,\theta}$ is holomorphic near $0^N$ and
$G_{N,\theta}(0^N)=1$, its logarithm is well defined near the origin and
vanishes at $0^N$.
\begin{definition}\label{def:clambda}
For any partition $\lambda$ such that $|\lambda|\le N$, let $c_\lambda=c_\lambda(N,\theta)$ denote the
coefficient of $p_\lambda(\vec x)$ in the Taylor expansion of
$\log G_{N,\theta}(\vec x)$ at $\vec x=0^N$; equivalently,
\[
    \log G_{N,\theta}(\vec x)
    =\sum_{|\lambda|\le N}c_\lambda p_\lambda(\vec x)+O(\|x\|^{N+1}).
\]
\end{definition}

For the coordinate tuple $a=(a_1,\ldots,a_N)$, consider the polynomial observables
\begin{equation*}
    p_k^N:=\frac1N\sum_{i=1}^N a_i^k,
    \qquad
    \widetilde p_k^N:=\frac1N\sum_{i=1}^N
       \left(\frac{a_i}{N}\right)^k.
\end{equation*}
The second normalization is the natural one when the particles themselves
have size $N$, as in the fixed-temperature framework considered here.

All expectations below mean normalized moment pairings with $\mu_N$. Covariances and classical joint cumulants are defined algebraically from these moments; in particular,
\[
    \Cov_{\mu_N}(f,g)=\E_{\mu_N}[fg]-\E_{\mu_N}[f]\E_{\mu_N}[g].
\]
No positivity is imposed on this bilinear form.

\medskip
\noindent\textbf{Fixed temperature.}

\begin{definition}[Fixed-temperature CLT-satisfaction]\label{def:fixedtempclt}
Fix $\theta>0$. We say that $\{\mu_N\}_{N\ge1}$ \emph{satisfies a CLT in the fixed-temperature regime} if there exist real numbers $\{m_k\}_{k\ge1}$ and symmetric real numbers $\{\alpha_{k,l}\}_{k,l\ge1}$ such that:
\begin{itemize}
    \item For every $k\ge1$,
    \[
        \lim_{N\to\infty}\E_{\mu_N}[\widetilde p_k^N]=m_k.
    \]
    \item For every $s\ge1$ and $k_1,\ldots,k_s\ge1$,
    \[
        \lim_{N\to\infty}N^s\E_{\mu_N}\!\left[
            \prod_{j=1}^s\bigl(\widetilde p_{k_j}^N
                -\E_{\mu_N}[\widetilde p_{k_j}^N]\bigr)
        \right]
        =\sum_{\pi}\prod_{\{i,j\}\in\pi}\alpha_{k_i,k_j}.
    \]
    Here the sum is over all pair partitions $\pi$ of $\{1,\ldots,s\}$ and is zero when $s$ is odd.
\end{itemize}
\end{definition}

The two directions give the concise result below.  The covariance formula
appears later in Theorem \ref{thm:fixedtempclt}.

\begin{theorem}[Fixed-temperature characterization]\label{thm:intro-fixed-iff}
Under the standing assumptions above, fix $\theta>0$. The sequence $\{\mu_N\}$ satisfies a fixed-temperature CLT in the sense of Definition~\ref{def:fixedtempclt}
if and only if there exist constants $\kappa_d$ and symmetric constants
$\kappa_{r,s}=\kappa_{s,r}$ such that, for every $d,r,s\ge1$ and every
partition $\lambda$ with $\ell(\lambda)\ge3$,
\[
    \frac{d}{N}c_{(d)}\longrightarrow\theta^{1-d}\kappa_d,
    \qquad
    (1+\mathbf1_{r=s}) rs\,c_{(r,s)}\longrightarrow
       \theta^{1-r-s}\kappa_{r,s},
    \qquad
    c_\lambda\longrightarrow0.
\]
The limiting means are determined by the first-order free cumulants
$\{\kappa_d\}$, and the limiting covariance is determined explicitly by
$\{\kappa_d\}$ and the second-order free cumulants
$\{\kappa_{r,s}\}$ through Eqn.~\eqref{eq_fixedtempcov}.
\end{theorem}

The first limit is the fixed-temperature LLN criterion of \cite{CX}.  The new
content at second order is that the two-row BGF coefficients are exactly the
second-order free cumulants, while the decay of all longer coefficients is
equivalent to the vanishing of higher classical cumulants.  In particular,
the condition is intrinsic to the BGF and does not require the particles to be
eigenvalues of an actual random matrix.

\medskip
\noindent\textbf{High temperature.}

\begin{definition}[High-temperature CLT-satisfaction]\label{Definition_CLT_sat}
Let $\theta=\theta_N\to0$ and $N\theta_N\to\gamma>0$. We say that $\{\mu_N\}_{N\ge1}$ \emph{satisfies a CLT in the high-temperature regime} if there exist real numbers $\{m_k\}_{k\ge1}$ and symmetric real numbers $\{\alpha_{k,l}\}_{k,l\ge1}$ such that:
\begin{itemize}
    \item For every $k\ge1$,
    \[
        \lim_{N\to\infty}\E_{\mu_N}[p_k^N]=m_k.
    \]
    \item For every $s\ge1$ and $k_1,\ldots,k_s\ge1$,
    \[
        \lim_{N\to\infty}N^{s/2}\E_{\mu_N}\!\left[
            \prod_{j=1}^s\bigl(p_{k_j}^N
                -\E_{\mu_N}[p_{k_j}^N]\bigr)
        \right]
        =\sum_{\pi}\prod_{\{i,j\}\in\pi}\alpha_{k_i,k_j},
    \]
    with the same pair-partition convention as in Definition~\ref{def:fixedtempclt}.
\end{itemize}
\end{definition}

The different fluctuation scale is reflected exactly in the scaling of the BGF
coefficients.

\begin{theorem}[High-temperature characterization]\label{thm:intro-high-iff}
Under the standing assumptions above, let $\theta=\theta_N\to0$ and $N\theta_N\to\gamma>0$. The sequence $\{\mu_N\}$ satisfies a high-temperature CLT in the sense of Definition~\ref{Definition_CLT_sat} if and only if there exist
constants $\kappa_d^\gamma$ and symmetric constants
$\kappa_{r,s}^\gamma=\kappa_{s,r}^\gamma$ such that, for every
$d,r,s\ge1$ and every partition $\lambda$ with $\ell(\lambda)\ge3$,
\[
    d\,c_{(d)}\longrightarrow\kappa_d^\gamma,
    \qquad
    N(1+\mathbf{1}_{r=s})rs\,c_{(r,s)}\longrightarrow\kappa_{r,s}^\gamma,
    \qquad
    N^{\ell(\lambda)/2}c_\lambda\longrightarrow0.
\]
The limiting means are determined by the first-order free $\gamma$-cumulants
$\{\kappa_d^\gamma\}$, and the limiting covariance is determined explicitly by
$\{\kappa_d^\gamma\}$ and the second-order free $\gamma$-cumulants
$\{\kappa_{r,s}^\gamma\}$ through Eqn.~\eqref{eq_hightempcov}.
\end{theorem}

\begin{remark}\label{rem:moment-clt-positivity}
Definitions~\ref{def:fixedtempclt} and~\ref{Definition_CLT_sat} are entirely in terms of moments and do not require $(\alpha_{k,l})_{k,l\ge1}$ to be positive semidefinite. By the moment--cumulant relations, their centered-moment conditions are equivalent to convergence of the second joint cumulants to $\alpha_{k,l}$ and vanishing of all joint cumulants of order at least three, under the corresponding scaling.

If each normalized $N$-particle generalized function is a probability measure, so that $\E_{\mu_N}$ is integration against that measure, every finite principal submatrix of $(\alpha_{k,l})$ is a limit of covariance matrices of the centered and scaled power sums, and is therefore positive semidefinite. Hence these second moments define a centered Gaussian family, possibly degenerate, whose joint moments are given by the same Wick formula. In this case, our CLT is precisely convergence in finite-dimensional moments to that Gaussian family.
\end{remark}

This theorem is the fluctuation-level counterpart of the high-temperature LLN
criterion of \cite{BCG}.  The first-order free $\gamma$-cumulants linearize the
$\gamma$-convolution introduced there.  The second-order array in
Theorem~\ref{thm:intro-high-iff} supplies the additional data needed for
the high-temperature CLT, and Section~8 shows that it is additive under the $\beta$-addition.  

Theorems~\ref{thm:intro-fixed-iff} and \ref{thm:intro-high-iff} have the same
three-level architecture, but their scalings are not formal substitutions for
one another.  At fixed temperature the one-row coefficients have order $N$ and
the two-row coefficients have order one.  At high temperature the corresponding
orders are one and $N^{-1}$.  This change is precisely what converts
$N^{-1}$ fluctuations of normalized global statistics into the usual
mean-field scale $N^{-1/2}$.

\subsection{Applications}

Our if-and-only-if characterizations reduce fixed- and high-temperature
CLTs to the corresponding coefficient asymptotics of logarithmic BGFs.
Whenever these transforms can be described explicitly, the criteria give
direct access to fluctuations and their limiting covariance. We give three
applications of such characterization.

First, the classical Gaussian and Laguerre $\beta$-ensembles possess a rich
integrable structure, including explicit BGFs. Their fixed-temperature CLTs
are well known; see, for example, \cite{Joh98,DE06}. Applying our
high-temperature criterion to these explicit transforms gives
CLTs and explicit limiting covariances directly; see
Section~\ref{sec:gaussian-laguerre-high-temp}.

Second, consider $N$-particle sequences $a(N)$ and $b(N)$ whose laws are
compactly supported real symmetric generalized functions, normalized as in
Section~\ref{sec:bgf}, following \cite[Section~2.4]{BCG}. We regard them as the
eigenvalues of virtual random matrices. Their $\beta$-addition
$c(N)=a(N)+_\theta b(N)$, where $\beta=2\theta$, is defined by
\[
    G^c_{N,\theta}(\vec x)
       =G^a_{N,\theta}(\vec x)G^b_{N,\theta}(\vec x).
\]
When the inputs are probability measures and $\theta=1/2,1,2$, this is precisely the eigenvalue law of the sum of
independent orthogonally, unitarily, or symplectically invariant matrices,
respectively, and hence naturally extends classical matrix addition, see e.g
\cite{GM, BCG}.
For general $\theta>0$, the output is defined as a generalized function;
that it is a probability measure whenever the inputs are probability
measures remains conjectural.
The product identity makes logarithmic BGF coefficients additive.
Consequently, if both inputs satisfy a CLT in either temperature regime,
then so does their sum in the same regime.
Its limiting covariance is explicit in the first- and second-order free
cumulants of the inputs; see Theorem~\ref{thm:betaadditionclt}.

Third, the high-temperature theory interpolates between the classical
theory at $\theta=0$ and the free-probability theory at fixed temperature.
At the LLN level, Benaych-Georges, Cuenca, and Gorin established the
$\gamma\to\infty$ transition to free cumulants and free convolution
\cite[Section~8.2]{BCG}. We obtain the corresponding transition at the
fluctuation level: after appropriately rescaling the first- and
second-order free $\gamma$-cumulants and the covariance, the
high-temperature formulas converge to their fixed-temperature counterparts; see
Theorem~\ref{thm:gamma-to-fixed-transition}.

Beyond these applications, one can extend the Gaussian fluctuation to two-dimensional particle systems, and study the limiting Gaussian fields. For a broad class of two-dimensional random surface models, global height fluctuations are expected to be described by the Gaussian free field (GFF). This picture is supported by results for random tilings and their discrete log-gas marginals \cite{Ken08,BGG17,BuG19,BoGG26}, for the six-vertex model in certain parameter regimes \cite{DKLM26}, and for anisotropic $(2+1)$-dimensional KPZ growth \cite{BF14}. On the random matrix side, the single-level fluctuation of log-gases at fixed-temperature can be viewed as a one-dimensional projection of a GFF, and the corners process of Wigner, Wishart matrices and the Jacobi ensembles in two dimension are also shown to be GFF
\cite{Joh98,Bor14,DP18,BG15}. A natural multilevel extension of a
one-dimensional $\beta$-ensemble is its $\beta$-corners process; see
\cite{Cue21}. By \cite[Proposition~4.1]{Cue21}, the BGF formalism is compatible with the interlacing
projections of these processes, suggesting
that our characterization can be extended to fixed- and high-temperature
CLTs for general two-dimensional corners processes. We will pursue this
direction in subsequent work.

\subsection{Methodology and topological expansion}\label{sec:methodology}

{The covariance in the \emph{if} direction is obtained by applying Dunkl
operators to the BGF and expanding their action into weighted sums of
Lukasiewicz paths with associated block functions. Such a moment-method
approach using Dunkl operators was used for LLNs in \cite{BCG,Yao21,X,CX}
and for edge limits in \cite{KX24,GXZ24}.
At the fluctuation scale, the leading moment terms cancel, and the
covariance is determined by contributions that are invisible to the LLN.
We organize the terms into configurations and track their cancellations
to isolate the surviving weights. Unlike the edge calculations, here the
moment degrees remain fixed, and we evaluate these weights directly,
without passing to a Brownian limit of the paths. The surviving
configurations give the explicit covariance formulas in
Eqns.~\eqref{eq_fixedtempcov} and \eqref{eq_hightempcov}. A separate
combinatorial argument in Section~\ref{sec:comb} identifies the
fixed-temperature formula with the second-order free-cumulant formula of
Collins--Mingo--\'{S}niady--Speicher \cite{CMSS07}, up to a factor $1/\theta$.}

The \emph{only if} direction and the vanishing of higher classical
cumulants in the \emph{if} direction use a different tool: a topological
expansion of the multivariate Bessel function. To recover the coefficients of
$\log G_{N,\theta}$ from particle moments, we write
\[
    \log G_{N,\theta}(\vec x)
    =\log\E\!\left[\exp\!\left(\log B_a(\vec x;\theta)\right)\right]
\]
and expand the right-hand side in classical joint cumulants. {The required input is an explicit expansion of $\log B_a$. The expansion
itself was obtained in the companion work \cite{CX} from the Jack-series
identities of Chapuy--Do\l\k{e}ga \cite{CD} and used there to study LLNs;
we recall it in Theorem~\ref{thm:keyexpansion}. Here we use it to relate
BGF coefficients to joint cumulants at the fluctuation scale, in both
directions of the CLT characterization.} It is indexed by connected
\emph{constellations}, a class of ribbon graphs embedded in surfaces,
allowing twisted ribbons and hence non-orientable surfaces. Each term
contains a weight
\[
    b^{\nu_\rho},\qquad b=\theta^{-1}-1=2/\beta-1,
\]
where $\nu_\rho$ is a nonnegative integer measuring non-orientability and
vanishes precisely for orientable constellations. The expansion packages
the joint moment information into a single formula, in which the power of $N$ of each term is controlled by the topology of the corresponding surface.

{Combining this formula with the classical cumulant expansion expresses
the BGF coefficients in terms of joint cumulants of the power sums and
gives the estimates needed for necessity. For sufficiency, we use the
same relation in reverse. At each fixed total degree $d$, after the
appropriate normalization, the coefficients indexed by partitions of
length at least three form a linear system for the higher joint
cumulants of that degree. Its matrix is asymptotically diagonal with
nonzero diagonal limits. The remaining terms are either controlled by
the LLN and covariance limits or are products of cumulants of smaller
total degree. Induction on the total degree then proves that all higher
classical joint cumulants vanish. Together with the covariance
calculation, this gives a necessary and sufficient fluctuation criterion
and a topological interpretation of second-order free cumulants and
their high-temperature counterparts.}

The idea of topological expansion has its root in gauge theory \cite{Hoo74,Cic82}, and has a long tradition in random matrix
theory. For GUE and GOE, the expansion comes directly from the matrix structure:
expanding traces into a polynomial of matrix entries and applying Wick's formula turns
moments into sums over polygon gluings, equivalently \emph{maps} on orientable surfaces;
non-orientable surfaces occur in the real case. At finite $N$, topology of the underlying surface of the map organizes an
exact moment expression, and increasing the genus lowers its order in
$N$. Consequently, only genus zero contributes to the fixed-temperature
LLN and CLT.

The picture changes under the edge scaling, which captures the local limit of the largest eigenvalues: the powers in the moment
method grow with $N$, and the growing number of maps compensates for
the suppression of higher genus. Contributions from all genera must
therefore be retained. See \cite{Oko00} for the edge calculation of GUE, and 
\cite{Sod07,Sod09,Sod10,FS10,Sod14} for GOE and more general real symmetric matrices.

A distinct source of topological expansion is the \emph{group structure} appearing in the corresponding random matrix object. For Haar-distributed compact
groups, Haar integration and Weingarten calculus replace Wick pairing, see  \cite{MPud19,MPud24}. {A symmetric-group counterpart appears in Okounkov's study of random
partitions \cite{Oko00}, where moments under the Plancherel measure are
expressed as weighted sums over branched coverings of surfaces.}

For general $\beta$, the combinatorial expansions relevant here come
instead from \emph{Jack series}. \cite{LC09} related the Gaussian
$\beta$-ensemble to a map series, where each term is weighted by a measure of non-orientability as in Section \ref{sec:topexpansion}, and
 \cite{CD} generalizes the expansion by introducing the more involved notion of constellations. In this setting, the expansion is extracted from the
combinatorics of symmetric functions, without requiring a realization by
concrete matrix or group structure. Recently, 
 such perspective was used in \cite{XuJack26} to study the soft edge limits of the
Jack--Plancherel process and their relation to asymptotic
$\beta$-weighted map enumeration.

{All the above approaches treat} $\beta>0$ as a fixed parameter or implicitly take $\beta=1,2,4$. Returning to the Bessel expansion, the high-temperature regime studied in this text introduces
a different reason for retaining all genera, even though the degrees of
the moments remain fixed. When $\theta\sim\gamma/N$, we have
$b\sim N/\gamma$. The genus-dependent factor $N^{-2g}b^{\nu_\rho}$ then
has order $N^{-2g+\nu_\rho}$, and the bound $0\le\nu_\rho\le2g$ allows
leading contributions with $\nu_\rho=2g$, instead of just $g=0$. We control this
all-genus expansion and {establish} the resulting coefficient estimates for
every $\gamma>0$; see Lemma~\ref{lem:onlyifhightemp}. This contrasts with
the genus-zero description at fixed temperature and gives a topological
interpretation of both the classical second-order free cumulants and
their high-temperature counterparts. In particular, topological
expansions of Jack series carry probabilistic information beyond the
classical matrix models.

\begin{table}[!htbp]
\centering
\small
\setlength{\tabcolsep}{4pt}
\renewcommand{\arraystretch}{1.22}
\begin{tabular}{@{}p{0.22\linewidth}p{0.25\linewidth}p{0.25\linewidth}p{0.22\linewidth}@{}}
\hline
\raggedright Object & \raggedright Limit regime &
\raggedright Combinatorial objects & \raggedright Source\tabularnewline
\hline
\raggedright Real/complex Hermitian matrices & \raggedright Fixed-temperature LLN and CLT &
\raggedright Polygon gluings / maps, genus zero &
\raggedright Matrix  structure\tabularnewline
\raggedright Real/complex Hermitian matrices & \raggedright Spectral edge &
\raggedright Polygon gluings / maps, all genera &
\raggedright Matrix  structure\tabularnewline
\raggedright Haar compact groups & \raggedright Finite $N$ expansion and its large $N$ limit &
\raggedright Surfaces with boundary &
\raggedright Group structure\tabularnewline
\raggedright Plancherel measure & \raggedright Edge of a random partition &
\raggedright Branched coverings &
\raggedright Group structure (symmetric group)\tabularnewline
\raggedright Gaussian $\beta$-ensemble & \raggedright Fixed-temperature LLN and CLT &
\raggedright $b$-weighted maps, genus zero &
\raggedright Jack series\tabularnewline
\raggedright Gaussian $\beta$-ensemble & \raggedright Finite $N$ expansion and edge limit &
\raggedright $b$-weighted maps, all genera &
\raggedright Jack series\tabularnewline
\raggedright Continuous $N$-particles & \raggedright Fixed-temperature LLN and CLT &
\raggedright Constellations, genus zero &
\raggedright Jack series (multivariate Bessel functions)\tabularnewline
\raggedright Continuous $N$-particles & \raggedright High-temperature LLN and CLT &
\raggedright Constellations, all genera &
\raggedright Jack series (multivariate Bessel functions)\tabularnewline
\hline
\end{tabular}
\caption{Sources and limiting regimes of the topological expansions discussed above.
``All genera'' means that the leading asymptotics are not restricted to genus zero.}
\label{tab:topological-expansions}
\end{table}

\begin{samepage}
The paper is organized as follows. Section~2 recalls the backgrounds in symmetric functions and topological expansion. Sections~\ref{sec:luk-cov} and
\ref{sec:dunklaction} state and prove the covariance formulas.
Section~\ref{sec:only-if} proves the \emph{only if} direction, and
Section~\ref{sec:vanishing-higher-cumulants} completes the \emph{if}
direction by proving the vanishing of higher classical cumulants.
Section~\ref{sec:comb} gives the second-order free-probabilistic
identification. Section~8 presents the applications.
\par
\end{samepage}

\medskip
\noindent\textbf{Notation.}
Throughout the paper, $N$ is a positive integer tending to infinity.  We work
with $N$ variables $x_1,\ldots,x_N$ and write $\vec x=(x_1,\ldots,x_N)$.  The
partial derivatives are denoted by
\[
    \partial_i:=\frac{\partial}{\partial x_i},
    \qquad i=1,\ldots,N,
\]
and $\sigma_{i,j}$ denotes the operator that interchanges $x_i$ and $x_j$.
The \emph{fixed-temperature regime} means that $\theta>0$ is fixed while
$N\to\infty$.  The \emph{high-temperature regime} means that
$\theta=\theta_N\to0$ and $N\theta_N\to\gamma$ for some $\gamma>0$.
For a real sequence $u_N$, its \emph{$N$-degree} is the smallest
$s\in\Z\cup\{-\infty\}$ such that
$\limsup_{N\to\infty}N^{-s}|u_N|<\infty$.

\subsection*{Acknowledgements}
We thank Maciej Dołęga, Cesar Cuenca and Vadim Gorin for helpful discussions. Some of the discussions took place during our stay in the Institut Mittag-Leffler in 2024 fall, and we are grateful to the institute and the organizers of the program ``Random Matrices and Scaling Limits".
\subsection*{Statement of AI use}
Most of the math ingredients in this text were figured out before 2025 spring, and the rest were also generated later by the human author. During the preparation of this text, we used ChatGPT 5.5, 5.6 and 6 for literature search, error detection and improving the exposition. In particular, LLM  pointed out the useful reference for the Leonov–Shiryaev
formula in Lemma \ref{lem:leonov}.
\section{Preliminaries}

In this section, we collect the necessary background on multivariate Bessel functions and Bessel generating functions. Sections \ref{sec:dunkl}--\ref{sec:bgf} contain classical results, while Section \ref{sec:topexpansion} recalls a more recent expansion from \cite{CD} and \cite{CX}.

\subsection{Dunkl operators}\label{sec:dunkl}

The Dunkl operators, introduced in \cite{D}, are defined by
\[
\D_i^{N,\theta} := \partial_i + \theta\sum_{j\colon j\ne i} \frac{1}{x_i-x_j}(1-\sigma_{i,j}),\qquad i=1,\dots,N.
\]
They pairwise commute, $\D_i^{N,\theta}\D_j^{N,\theta} = \D_j^{N,\theta}\D_i^{N,\theta}$, for all $i,j$, and their symmetric joint eigenfunctions, discussed in the next subsection, will play an important role in this paper.
We can view $\D_i^{N,\theta}$ as operators acting on holomorphic functions on domains of $\C^N$ that are symmetric with respect to permutations of the coordinates.
We will also use the symmetric operators
\[
\P_k^{N,\theta} := \sum_{i=1}^N{ \Big(\D_i^{N,\theta}\Big)^k },\qquad k=1,2,\cdots.
\]
For simplicity, we drop the superscripts $N,\theta$, so that $\D_i^{N,\theta}, \P_k^{N,\theta}$ will simply be denoted $\D_i, \P_k$.

\subsection{Multivariate Bessel functions}\label{sec:mbf}

\begin{theorem}[\cite{O}]\label{thm:bessel_existence}
    Let $a_1,\dots,a_N\in\C$ be arbitrary. There exists a unique holomorphic function $F(x_1,\dots,x_N)$, normalized by $F(0^N)=1$, that is symmetric with respect to its variables and solves the differential equations
    \[
    \P_k F(x_1,\dots,x_N) = \left( \sum_{i=1}^N{a_i^k} \right)F(x_1,\dots,x_N),\quad\textrm{for all }k\ge 1.
    \]
    If we denote this solution by $B_{(a_1,\dots,a_N)}(x_1,\dots,x_N;\theta)$, then the map $(a,x)\mapsto B_a(x;\theta)$ can be extended to an entire function on $2N$ variables.
\end{theorem}

Indeed, $B_a(\vecx;\theta)$ has a more explicit expansion in terms of Jack polynomials as follows (see~\cite[Sec.~4]{Ok-Ol}):
\begin{equation}\label{eq_BesselJack}
    B_a(\vecx;\theta) = \sum_{\ell(\la)\le N}\frac{\prod_{(i,j)\in\la}\big((\la_i-j)+\theta(\la_j'-i)+\theta\big)}{\prod_{(i,j)\in\la}\big((\la_i-j)+\theta(\la_j'-i) + 1\big) \big(N\theta+j-1-\theta(i-1)\big)} \,J_\la(\vecx;\theta)J_\la(a;\theta).
\end{equation}

\subsection{Bessel generating functions}\label{sec:bgf}

Following \cite{BCG}, let $\M_N$ be the set of Borel probability measures on ordered real $N$-tuples $a_1\le\dots\le a_N$ and, for any $R>0$, let $\M_N^R\subset\M_N$ be the subset of measures $\mu$ such that
\begin{equation}\label{eq:technical}
\int_{a_1\le\dots\le a_N} e^{NR\max_i|a_i|}\,\mu(\d a_1,\dots,\d a_N) < \infty.
\end{equation}

\begin{proposition-definition}[Lemma~2.9 of~\cite{BCG}]\label{def:bgf}
    If $\mu\in\M_N^R$ for some $R>0$, then the integral
    \[
    G_\theta(x_1,\dots,x_N;\mu) := \int_{ a_1\le\dots\le a_N} B_{( a_1,\dots, a_N)}(x_1,\dots,x_N;\theta)\mu(\d a_1,\dots,\d a_N)
    \]
    defines a holomorphic function in the domain
    \[
    \{ (x_1,\dots,x_N)\in\C^N \colon\ |\Re x_i|<R,\ i=1,2,\dots,N \},
    \]
    that is called the \textbf{Bessel generating function} of $\mu$.
\end{proposition-definition}

The measures in $\M_N^R$ will be called \emph{$R$-exponentially decaying}.
We also let $\M_N^{>0} := \bigcup_{R>0}{\M_N^R}$, and call all measures in $\M_N^{>0}$ \emph{exponentially decaying}.
Any $\mu\in\M_N^{>0}$ admits a Bessel generating function $G_\theta(x_1,\dots,x_N;\mu)$ that is a holomorphic function in a neighborhood of the origin.
Following \cite[Section~2.4]{BCG}, we also allow $\mu$ to be a compactly supported real symmetric generalized function on $\mathbb R^N$ with $\langle\mu,1\rangle=N!$. Here a generalized function is a continuous linear functional on compactly supported smooth test functions, and symmetry means invariance under coordinate permutations. Compact support permits pairing with any smooth function. In this case, set
\[
    \E_\mu[f]:=\frac{1}{N!}\langle\mu,f\rangle,
    \qquad
    G_\theta(\vec x;\mu)
    :=\frac{1}{N!}\langle\mu,B_{\bullet}(\vec x;\theta)\rangle.
\]
The BGF is entire by \cite[Proposition~2.13]{BCG}. Together with exponentially decaying probability measures, these are the $N$-particle objects considered in our general limit statements. In both cases, $\E_\mu[1]=1$.

In either case, the normalization and symmetry of the Bessel function give
\[
G_\theta(0^N;\mu)=1,
\]
and $G_\theta(x_1,\dots,x_N;\mu)$ is symmetric with respect to $x_1,\dots,x_N$.

\begin{proposition}[Propositions~2.11 and~2.13 of~\cite{BCG}]\label{prop:action_dunkl}
Let $s\ge1$, $k_1,\ldots,k_s\ge1$, and let $\mu$ be either an exponentially decaying probability measure or a compactly supported symmetric generalized function normalized as above. Then
\[
    \left(\prod_{i=1}^s\P_{k_i}\right)G_\theta(\vec x;\mu)
    \Big|_{\vec x=0}
    =\E_\mu\!\left[\prod_{i=1}^s p_{k_i}(a)\right],
\]
where the right-hand side is the corresponding normalized moment pairing.
\end{proposition}

\subsection{Maps on closed surfaces and a topological expansion}\label{sec:topexpansion}

In this section, we recall a finite-$N$ expansion of $\ln B_a(\vec{x};\theta)$ in terms of maps on closed surfaces, obtained from the generating function of Chapuy--Do\l\k{e}ga \cite{CD} in \cite[Section 5]{CX}. It is used in both fixed- and high-temperature regimes in Section~\ref{sec:only-if}. We give the needed definitions below and reproduce the proof of Theorem~\ref{thm:keyexpansion} in Appendix~\ref{app:topological-expansion}.

A closed surface is a compact surface without boundary; all surfaces here are connected. For a cellularly embedded graph (that is, an embedding whose complementary regions are open disks), the Euler characteristic is
\[
\chi=V-E+F=2-2g.
\]
Here $V,E,F$ count vertices, edges and complementary regions, called faces. Using our convention of genus, on an orientable surface, $g=0,1,2,\ldots$, while on a non-orientable surface with $h\ge1$ crosscaps, $g=h/2$. A surface is determined up to homeomorphism by its orientability and this value of $g$.

\begin{definition}\label{def:map}
    
A \textbf{map} $\MM$ is a cellular embedding of a finite graph, possibly with multiple edges and loops, on a closed (orientable or non-orientable) surface. We identify embeddings related by a homeomorphism preserving the graph and any specified colors and root.
\end{definition}

\begin{definition}\label{def:const}
Let $k\ge 1$ be an integer. A $\mathbf{k}$-\textbf{constellation} $\MM$ is a map, equipped with a coloring of its vertices with colors from the set $\{0, 1,\dots, k\}$, such that the following conditions are satisfied.

\smallskip

    $\bullet$ Each vertex of color 0 is only connected to vertices of color 1.

\smallskip

    $\bullet$ Each vertex of color $k$ is only connected to vertices of color $k-1$.

\smallskip

    $\bullet$ For $1\le i\le k-1$, each corner at a vertex of color $i$ separates vertices of colors $i-1$ and $i+1$.

\smallskip

\noindent Finally, we also color all corners with the same colors as their corresponding vertices.
\end{definition}

The \textbf{size} of $\MM$, denoted by $|\MM|$, is the number of corners of color $0$ in $\MM$.
Note that the number of edges of any $k$-constellation $\MM$ is $k\cdot|\MM|$.
There are exactly $|\MM|$ edges between each successive pair of colors: this is immediate for colors $0,1$ and then follows successively from the alternating condition at intermediate vertices.

We say that $\MM$ is \textbf{connected} if the underlying graph is connected, and $\MM$ is \textbf{orientable/non-orientable} if its underlying surface is orientable/non-orientable. The genus of $\MM$, denoted $g(\MM)$, is the genus of the underlying surface.

An \emph{oriented corner} is a corner together with one of the two local directions around its incident vertex; we may encode this choice by a sign $\epsilon\in\{-1,+1\}$. 
We say that $\MM$ is \emph{rooted} if it has a distinguished oriented corner of color $0$, which is called the \emph{root of $\MM$}.

By definition, in any $k$-constellation $\MM$, each vertex of color $k$ is connected only to vertices of color $k-1$.
Thus, we can add a new vertex of color $k+1$ for each corner of color $k$ and join it with an edge to the corresponding vertex of color $k$ to obtain a $(k+1)$-constellation of the same size.
This size-preserving map from $k$-constellations to $(k+1)$-constellations preserves rooted, connected constellations. Let $C_{d,\bullet}^{(k)} $ be the set of rooted, connected $k$-constellations of size $d$; the previous discussion describes inclusion maps
\[
i_k\colon C_{d,\bullet}^{(k)} \hookrightarrow C_{d,\bullet}^{(k+1)},\quad k\ge 1.
\]

A rooted, connected, infinite constellation of size $d$ is represented by a rooted, connected $k_0$-constellation for some finite $k_0\ge1$, followed by the leaf extensions $i_{k_0},i_{k_0+1},\ldots$. Two representatives describe the same object if their extensions agree at some common level. In other words,
\[
C_{d,\bullet}^{(\infty)}=\varinjlim_k C_{d,\bullet}^{(k)}.
\]
Note that the initial level $k_0$ is arbitrary, and each extension adds $d$ leaves and $d$ edges, without changing the faces, the color-$0$ vertices, the root, or the underlying surface.

\begin{definition}\label{def:dataconstellation}

For $\MM\in C_{d,\bullet}^{(\infty)}$, let $F(\MM)$ be its set of faces, read from a sufficiently large finite representative. We record the following data:

    $\bullet$ For any $f\in F(\MM)$, denote by $\deg(f)$ the number of corners of color $0$ in $f$. Also, denote by $\pmb{\mu_1}(\MM)$ the partition with parts $\deg(f)$, $f\in F(\MM)$.

\smallskip

    $\bullet$ Denote by $V_0(\MM)$ the set of vertices of $\MM$ of color $0$, and $v_0(\MM):=|V_0(\MM)|$.

\smallskip

    $\bullet$ 
For any $v\in V_0(\MM)$, denote by $\deg(v)$ the number of incident edges.\footnote{Note that $\deg(v)$ is not necessarily the number of vertices adjacent to $v$, because multiple edges are allowed.} Also, denote by $\pmb{\mu_2}(\MM)$ the partition with parts $\deg(v)$, $v\in V_0(\MM)$.

\smallskip

    $\bullet$ For any $i\ge 1$, denote by $v_i(\MM)$ the number of vertices of $\MM$ of color $i$.
    Consider the sequence $$\eta(\MM):=\big(|\MM|-v_i(\MM)\big)_{i\ge 1}.$$
We say that $\MM$ is \textbf{normal} if $v_1(\MM)\le v_2(\MM)\le\cdots$.
Since $v_i(\MM)\le|\MM|$, for all $i$, normality is equivalent to the fact that $\eta(\MM)$ is a partition.
\end{definition}

\noindent By definition, $|\pmb{\mu_{1}}(\MM)|=|\pmb{\mu_{2}}(\MM)|=|\MM|$.

The parts of each partition are listed in decreasing order. In particular, $\ell(\pmb{\mu_1})$ is the number of faces and $\ell(\pmb{\mu_2})=v_0$ is the number of color-$0$ vertices. Our face degree counts color-$0$ corners, rather than the full boundary length; each such corner contributes once to each of the two partitions. The sequence $\eta$ has finite support by the leaf-extension construction, and $|\eta|=\sum_{i\ge1}(d-v_i)$.

\begin{proposition}\cite[Eqn. (5)]{CD}\label{prop:euler}
    For any $\MM\in C_{d,\bullet}^{(\infty)}$ of genus $g$, we have
\begin{equation}
    |\eta(\MM)|-\ell\big(\pmb{\mu_1}(\MM)\big) -\ell\big(\pmb{\mu_2}(\MM)\big) = 2g-2.
\end{equation}
\end{proposition}

\begin{proof}
    Take a large enough $k$, so that $\eta_i(\MM)=0$, for all $i\ge k$. 
Take a finite $k$-constellation representing $\MM$ at such a level; the counts below refer to this finite representative. It follows that the number of edges of $\MM$ is $k\cdot |\MM|$, the number of vertices of $\MM$ is $\sum_{i=0}^k{v_i(\MM)} = \sum_{i=1}^k{v_i(\MM)}+\ell(\pmb{\mu_2}(\MM))$, and the number of faces of $\MM$ is $\ell(\pmb{\mu_1}(\MM))$. Since $E-V-F=2g-2$, the result follows.
\end{proof}

 The exponent $\nu_\rho(\MM,c)$ is a nonnegative integer
statistic of the rooted map $(\MM,c)$ measuring its
non-orientability. It is defined by a recursive edge-deletion
procedure; see \cite[Sections~3 and~6.1]{CD} for details. It satisfies the following two properties:
\[
0\le \nu_\rho(\MM,c)\le 2g(\MM),
\qquad
\nu_\rho(\MM,c)=0\quad\Longleftrightarrow\quad\MM\text{ is orientable}.
\]

\begin{theorem}\cite[Theorem 5.6]{CX}\label{thm:keyexpansion}
    
Let $n\in\N$, $\theta>0$, and $\alpha=\theta^{-1}$. Suppose $N>\max(\alpha,1)n$, equivalently $N>n$ and $N\theta>n$. For each $a\in\C^N$, define the logarithm  $\ln B_a(\vec{x};\theta)$ by $\ln(1+x)=\sum_{r=1}^{\infty}\frac{(-1)^{r-1}}{r}x^r$. Its Taylor expansion through total degree $n$ in $\vec{x}$ is
    \begin{equation}\label{eq_keyexpansion}
    \begin{split}
        \ln\Big(B_{a}(\vecx;\theta)\Big)&\\
=\sum_{d=1}^{n}\frac{1}{d}&\ \sum_{g\in \frac{1}{2}\Z_{\ge 0}}\sum_{(\mathbf{M},c)\in C^{(\infty)}_{d,\bullet}, \ \mathbf{M}\ \text{normal},\ g(\mathbf{M})=g}\alpha^{d-1}(\alpha-1)^{\nu_{\rho}(\mathbf{M},c)}\prod_{f\in F(\mathbf{M})}\frac{p_{\deg(f)}(a/N)}{N}\\
        &\times\prod_{v\in V_{0}(\mathbf{M})}\frac{p_{\deg(v)}(\vec{x})}{N}\cdot N^{2-2g}\cdot f_{\eta(\mathbf{M})}\Big(\big\{(-1)^k\big\}_{k\ge 1}\Big)+O\big(\|x\|^{n+1}\big),
        \end{split}
    \end{equation}
where $f_{\eta}(e_1,e_2,\dots)$ is the polynomial that expresses the monomial symmetric function $m_{\eta}$ in terms of the elementary symmetric polynomials $e_1,e_2,\dots$, and $\MM$ in the inner sum ranges over all rooted, connected, normal, infinite constellations of size $d$. 
The series over $g$ on the right-hand side, with the finite sums at each genus performed first, converges absolutely coefficientwise. 
\end{theorem}

We reproduce the proof of Theorem \ref{thm:keyexpansion} in Appendix~\ref{app:topological-expansion} using the same argument\footnote{\cite[Theorem 5.6]{CX} used a different basis $\prod_{f\in F(\mathbf{M})}\frac{p_{\deg(f)}(\vec{x})}{N}\prod_{v\in V_{0}(\mathbf{M})}\frac{p_{\deg(v)}(a/N)}{N}$, which does not affect the result since $B_a(x;\theta)=B_{x}(a;\theta)$.} in \cite[Theorem 5.6]{CX}.

\section{Lukasiewicz paths and covariance--cumulant formulas}\label{sec:luk-cov}

We divide the proofs of Theorems~\ref{thm:fixedtempclt} and
\ref{thm:hightempclt} into the \emph{if} and \emph{only if} directions.
Sections~\ref{sec:luk-cov}, \ref{sec:dunklaction}, and
\ref{sec:vanishing-higher-cumulants} establish the \emph{if} direction.  The
present section develops the Lukasiewicz-path framework and states the limiting
covariance formulas; Section~\ref{sec:dunklaction} proves these formulas, while
Section~\ref{sec:vanishing-higher-cumulants} proves the vanishing of the higher
joint cumulants.  The \emph{only if} direction is proved in
Section~\ref{sec:only-if}.

Recall from Definition \ref{def:clambda} that \[
\ln G_{\theta,N}(\vec{x}) = \sum_{\la\in\Y}c_\la p_\la(\vec{x}).
\]
We start by recalling the moment LLN results in the fixed- and high-temperature regimes. Throughout the statements below, $\mu_N$ satisfies the standing assumptions of Section~\ref{sec:bgf}, and all convergence is understood in terms of its normalized moment functional. The moment--cumulant formulas in the two regimes, together with the two covariance--cumulant formulas in Theorems \ref{thm:fixedtempclt} and \ref{thm:hightempclt}, fall into a unified combinatorial framework in terms of Lukasiewicz paths.

\begin{definition}\label{def:luk}
A Lukasiewicz path of length $k$ is a function $E:\{0,1,2,\cdots,k\}\rightarrow \Z_{\ge 0}$, such that $E(0)=E(k)=0$, and for $t=1,2,...,k$, $E(t)-E(t-1)\in \Z_{\ge -1}$.
    For $k\in \Z_{\ge 1}$, denote the set of Lukasiewicz paths of length $k$ by $L(k)$.
\end{definition}

\begin{theorem}\label{thm:lln}\cite[Theorem 1.1]{CX}
    In the fixed-temperature regime, the LLN holds, i.e., \begin{align*}
    \lim_{N\rightarrow\infty}{\E_{\mu_N}\left[\prod_{j=1}^s\tilde{p}_{k_j}^{N}\right]} = \prod_{j=1}^s{m_{k_j}},
    \end{align*} if and only if for all $d\in \Z_{\ge 1}$,
    \begin{equation}\label{eq_llnfixed}
        \begin{split}
            \frac{d}{N}c_{(d)}&\rightarrow \theta^{1-d}\kappa_{d};\\
            \frac{1}{N^{\ell(\lambda)}}c_{\lambda}&\rightarrow 0\quad\quad \text{for}\ \ell(\lambda)\ge 2
        \end{split}
    \end{equation}
    when $N\rightarrow\infty$. Moreover, the limiting moments $m_k$, $k\ge1$, are given in terms of the free cumulants $\kappa_{l}$'s defined as above, by the equalities
    \begin{equation}
        m_{k}=\sum_{E\in L(k)}\prod_{B_{i}\in E}\kappa_{|B_{i}|},
    \end{equation}
    where $|B_i|$ is one plus the size of the $i$th nonnegative increment of $E$.
\end{theorem}
\begin{theorem}\cite[Theorem 3.8 and 3.10]{BCG}
    In the high-temperature regime, the LLN holds, i.e., \begin{align*}
    \lim_{N\rightarrow\infty}{\E_{\mu_N}\left[\prod_{j=1}^sp_{k_j}^{N}\right]} = \prod_{j=1}^s{m_{k_j}},
    \end{align*} if and only if for all $d\in \Z_{\ge 1}$,
    \begin{equation}\label{eq_llnhigh}
        \begin{split}
            d\cdot c_{(d)}&\rightarrow\kappa^{\gamma}_{d};\\
    c_{\lambda}&\rightarrow 0\quad\quad \text{for}\ \ell(\lambda)\ge 2
        \end{split}
    \end{equation}
    when $N\rightarrow\infty$. Moreover, the limiting moments $m_k$, $k\ge1$, are given in terms of the $\gamma$-cumulants $\kappa_{l}^{\gamma}$ defined as above, by the equalities
    \begin{equation}
        m_{k}=\sum_{E\in L(k)}\prod_{t:\ E(t)-E(t-1)=-1}(\gamma+E(t-1))\prod_{B_{i}\in E}\kappa^\gamma_{|B_{i}|},
    \end{equation}
    where $|B_{i}|$ are defined in the same way as above.
\end{theorem}

\begin{remark}
    The condition given in Eqn.~\eqref{eq_llnhigh} is a reformulation of \cite[Theorem 3.8]{BCG}. The combinatorial formula in \cite{BCG} is a weighted sum over partitions which are not necessarily non-crossing, but \cite[Proposition 5.12]{X} identifies it with a weighted sum over non-crossing partitions, which is equivalent to Eqn.~\eqref{eq_llnhigh} by the well-known bijection between non-crossing partitions and Lukasiewicz paths, see e.g., \cite[Proposition 9.8]{NS}.
\end{remark}

To formulate our next results, we introduce some further notions related to the configurations of Lukasiewicz paths.

For $m,n\in \Z_{\ge 1}$, consider Lukasiewicz paths of length $m+n$, with a marked point at time $m$, such that each path is a function $E(t):\{0,1,2,...,m+n\}\rightarrow \Z_{\ge 0}$ with $E(0)=E(m+n)=0$.  We denote the set of such paths by $\tilde{L}(m,n)$, and its subset where $E(m)>0$ by $L(m,n)$.  We extend the domain of $E$ to $[0,m+n]$ by taking $E(t)=E(\lfloor t\rfloor)$, and still denote the new function by $E(t)$, which is c\`adl\`ag by definition, with discontinuity points only at a subset of $\{1,2,...,m+n\}$.

For $E\in L(m,n)$, let $\{q_{i}(t)\}_{i=1}^{\infty}:[0,m+n]\rightarrow \Z_{\ge 0}$ be a collection of nonnegative integer-valued functions, c\`adl\`ag except for $i=1,2$ at $t=m$, that are identically zero for all but finitely many $i$'s. Let $q(t)=\sum_{i=1}^{\infty}q_{i}(t)$.  We call $t$ the time variable, and  $\{q_{i}\}_{i=1,2,...}$ the \emph{block functions of $E$}, if the following properties hold:

\begin{itemize}
    \item  Each $q_{i}(t)$ has discontinuity points only at a subset of $\{1,2,...,m+n\}$, and the sets of left-discontinuity points of the $q_i$'s are all disjoint.
    \item $q(t)\le E(t)$ for all $t\in [0,m+n]$.
    \item $q_{1}(t)=0$ for $t\in [0,m]$, and $q_1(m+)=E(m)-q(m)$. $q_{2}(t)=0$ for $t\in (m,m+n]$.

    \item  At each left-discontinuity point $t$ of $q_{i}$ for some $i$, there are two cases:

    (1). If  $E(t)-E(t-)=-1$, then we say that $q_{i}$ has a \emph{jump} at $t$, and let $h=q_{i}(t)-q_{i}(t-)$ be the \emph{jump size} of $q_{i}$ at $t$. Moreover,
    \begin{equation}
        \begin{split}
            &|h| \le |E(t-)-q(t-)-q_i(t-)-\frac{1}{2}(\text{sgn}(h)+1)|,\\
            &h\cdot (E(t-)-q(t-)-q_i(t-)) > 0.
        \end{split}
    \end{equation}

    (2). If $E(t)-E(t-)\ge 0$, then $$k_{2}:=q_{i}(t)-q_{i}(t-)\in \Z_{\ge 1},$$ and let $k_{1}=E(t)-E(t-)\in \Z_{\ge 0}$.
\end{itemize}

 Similarly, for $k\in \Z_{\ge 1}$, $E\in L(k)$, extend the domain of $E$ in the same way so that $E(t):[0,k]\rightarrow \Z_{\ge 0}$ becomes a c\`adl\`ag function with integer-valued discontinuity points. Introduce $\{q_{i}(t)\}_{i=1}^{\infty}:[0,k]\rightarrow\Z_{\ge 0}$ in a similar way, and  we call $\{q_{i}\}_{i=1,2,...}$ the block functions of $E$ if the following holds:

\begin{itemize}
    \item Each $q_{i}(t)$ has discontinuity points only at a subset of $\{1,2,\ldots,k\}$, and the sets of discontinuity points of the $q_i$'s are all disjoint.

    \item $q(t)\le E(t)$ for all $t\in [0,k]$.
    \item $q_{1}(t)=0$ for $t\in [0,k]$.
    \item At each discontinuity point $t$ of $q_{i}$ for some $i\ge 2$, we have that $E(t)-E(t-)=-1$. Again we say that $q_{i}$ has a \emph{jump} at $t$, $t$ is one of the \emph{jump times} of $q_{i}$, and let $h=q_{i}(t)-q_{i}(t-)$ be the \emph{jump height} of $q_{i}$ at $t$. Moreover,
    \begin{equation}
        \begin{split}
            &|h| \le |E(t-)-q(t-)-q_i(t-)-\frac{1}{2}(\text{sgn}(h)+1)|,\\
            &h\cdot (E(t-)-q(t-)-q_i(t-)) > 0.
        \end{split}
    \end{equation}

\end{itemize}

\begin{definition}\label{def:walk-blocks-configuration}

For $E(t)\in L(m,n)$ or $L(k)$, let $D(E)$ be the set of discontinuity points $t$ of $E$ where $E(t)-E(t-)=-1$, and let $U(E)=\{1,2,...,m+n\}\setminus D(E)$ ($\{1,2,...,k\}\setminus D(E)$ resp.) be the collection of times that the path goes up or horizontally.

The data consisting of a path $E(t)\in L(m,n)$ or $L(k)$ and its block functions $(E,\{q_{i}\}_{i=1,2,...})$ form a \emph{walk-block configuration}, which we call a configuration for short; for simplicity of notation, we write it as $(E,\vec{q})$. Given a $(E,\vec{q})$, denote the set of jump times by $J(E,\vec{q})\subset D(E)$, and let $$P(E,\vec{q})=D(E)\setminus J(E,\vec{q}).$$

For $E\in L(m,n)$, denote $C=E(m)$, and let \begin{equation}
    \begin{split}
        &\tau=\inf\{t\ge m:E(t)<C\},\\
        &\bar{\tau}=\sup\{t\in[0,m]:E(t)<C\}.
    \end{split}
\end{equation}
Furthermore, let \begin{equation}
    \begin{split}
        &A(E)=|\{\bar{\tau}+1,...,m\}\cap D(E)|+1,\\
        &B(E)=|\{m+1,...,\tau-1\}\cap D(E)|+1.
    \end{split}
\end{equation}
\end{definition}
\begin{example}\label{ex:tildetauandtau}
    Consider $E\in L(7,4)$, such that $E(1)=1,E(2)=2, E(3)=1, E(4)=4,E(5)=3,E(6)=2, E(7)=2, E(8)=3, E(9)=2, E(10)=1$. Then  $C=2,\bar{\tau}=4,\tau=10$. Moreover, $$\{\bar{\tau}+1,...,m\}\cap D(E)=\{5,6\},\quad\quad \{m+1,...,\tau-1\}\cap D(E)=\{9\},$$
    so $A(E)=3, B(E)=2$.

\end{example}

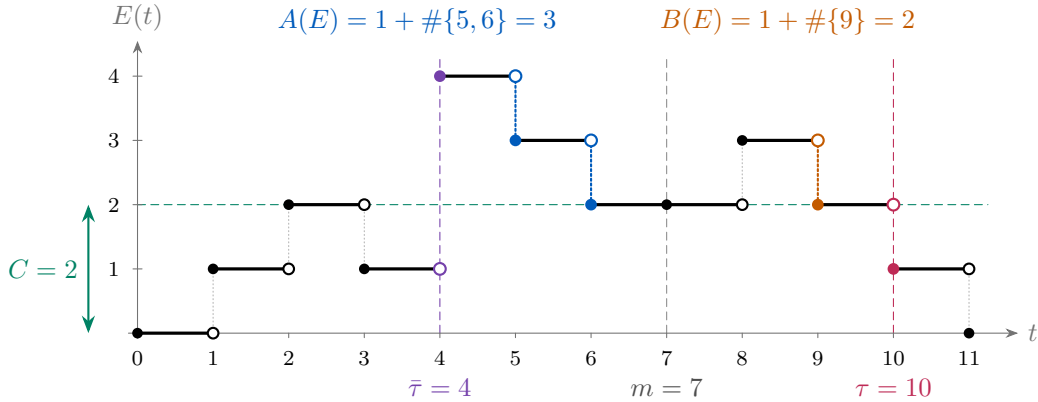
\begin{figure}[htbp]
\centering
\begin{tikzpicture}[x=1cm,y=0.85cm,>=Stealth,
    every node/.style={font=\small},line cap=round]
\definecolor{walkA}{RGB}{0,91,187}
\definecolor{walkB}{RGB}{195,90,0}
\definecolor{walkC}{RGB}{0,130,100}
\definecolor{walkLeft}{RGB}{116,64,170}
\definecolor{walkRight}{RGB}{190,42,85}

\draw[walkC,densely dashed] (0,2) -- (11.25,2);
\draw[walkC,<->,thick] (-0.65,0) -- (-0.65,2)
    node[midway,left] {$C=2$};
\draw[walkLeft,densely dashed] (4,0) -- (4,4.3);
\draw[black!45,densely dashed] (7,0) -- (7,4.3);
\draw[walkRight,densely dashed] (10,0) -- (10,4.3);

\draw[->,black!55] (-0.12,0) -- (11.65,0) node[right] {$t$};
\draw[->,black!55] (0,-0.08) -- (0,4.55) node[above] {$E(t)$};
\foreach \t in {0,...,11}{
    \draw[black!55] (\t,0) -- (\t,-0.07);
    \node[below=3pt] at (\t,0) {\scriptsize $\t$};
}
\foreach \h in {1,...,4}{
    \draw[black!55] (-0.06,\h) -- (0.06,\h);
    \node[left=3pt] at (0,\h) {\scriptsize $\h$};
}

\foreach \t/\a/\b in {1/0/1,2/1/2,3/2/1,8/2/3,11/1/0}{
    \draw[black!30,densely dotted] (\t,\a) -- (\t,\b);
}
\draw[walkA,thick,densely dotted] (5,4) -- (5,3);
\draw[walkA,thick,densely dotted] (6,3) -- (6,2);
\draw[walkB,thick,densely dotted] (9,3) -- (9,2);

\draw[very thick]
    (0,0) -- (1,0)
    (1,1) -- (2,1)
    (2,2) -- (3,2)
    (3,1) -- (4,1)
    (4,4) -- (5,4)
    (5,3) -- (6,3)
    (6,2) -- (8,2)
    (8,3) -- (9,3)
    (9,2) -- (10,2)
    (10,1) -- (11,1);

\foreach \t/\h in {1/0,2/1,3/2,8/2,11/1}{
    \draw[thick,fill=white] (\t,\h) circle (2pt);
}
\foreach \t/\h in {0/0,1/1,2/2,3/1,7/2,8/3,11/0}{
    \fill (\t,\h) circle (2pt);
}
\foreach \t/\h in {5/4,6/3}{
    \draw[walkA,thick,fill=white] (\t,\h) circle (2.2pt);
}
\foreach \t/\h in {5/3,6/2}{
    \fill[walkA] (\t,\h) circle (2.2pt);
}
\draw[walkB,thick,fill=white] (9,3) circle (2.2pt);
\fill[walkB] (9,2) circle (2.2pt);
\draw[walkLeft,thick,fill=white] (4,1) circle (2.2pt);
\fill[walkLeft] (4,4) circle (2.2pt);
\draw[walkRight,thick,fill=white] (10,2) circle (2.2pt);
\fill[walkRight] (10,1) circle (2.2pt);

\node[walkA] at (3.7,4.85) {$A(E)=1+\#\{5,6\}=3$};
\node[walkB] at (8.6,4.85) {$B(E)=1+\#\{9\}=2$};
\node[walkLeft,below] at (4,-0.55) {$\bar\tau=4$};
\node[black!70,below] at (7,-0.55) {$m=7$};
\node[walkRight,below] at (10,-0.55) {$\tau=10$};
\end{tikzpicture}
\caption{The Lukasiewicz path in Example~\ref{ex:tildetauandtau}.
Solid dots give the function values; open dots give the distinct left limits.
The blue and orange down jumps determine $A(E)$ and $B(E)$, respectively.}
\label{fig:cadlag-path-ABC}
\end{figure}

\begin{definition}\label{def:firstorderconf}

Let $k=m+n$. For $E\in L(k)$, denote $C=E(m)$. We say $(E,\vec{q})$ is a \emph{limiting first-order configuration} if \textbf{$C>0$}, and

(a). $$q_{i}(t)\equiv 0\ \text{for\ all}\ i=1,2,...$$

For $E\in L(m,n)$, we say $(E,\vec{q})$ is a \emph{limiting first-order configuration} if \textbf{$C>0$}, and it is in one of the following three cases:

(b). $q_{i}(t)\equiv 0$ for $i\ge 2$. $q_{1}(t)=C$ on $(m,\tau)$, and $q_{1}(t)=0$ elsewhere.

(c). $q_{i}(t)\equiv 0$ for $i\ne 2$. $q_{2}(t)=C$ on $[s,m]$ for some $s\in \{\bar{\tau}+1,...,m\}\cap D(E)$, and $q_{2}(t)=0$ elsewhere.

(d). $q_{i}(t)\equiv 0$ for $i\ne 3$.  $q_{3}(t)=C$ on $[s,\tau)$ for some $s\in \{\bar{\tau}+1,...,m\}\cap D(E)$, and $q_{3}(t)=0$ elsewhere.

 For $m\ge n, m,n\in \Z_{\ge 1}$, denote the collection of all limiting first-order configurations in cases (a)--(d), whose paths have total length $m+n$, by $Conf_{1}(m,n)$.
 \end{definition}
 One can check by definition that all such configurations satisfy the properties of block functions.

\begin{remark}
    As  we  will see later, configurations in case (a) are put in one-to-one correspondence with configurations in case (b), and configurations in case (c) are put in one-to-one correspondence with configurations in case (d), where the corresponding configurations on two sides share the same Lukasiewicz path $E(t)$.

\end{remark}

\begin{definition}\label{def:secondorderconf}

For $E\in L(m,n)$, we say $(E,\vec{q})$ is a \emph{limiting second-order configuration} if \textbf{$C>0$}, and it is in one of the following two cases:

(e). $q_{i}(t)\equiv 0$ for $i\ne 2$. $q_{2}(t)=C$ on $[s,m]$ for some $s\in \{\bar{\tau},...,m\}\cap U(E)$, and $q_{2}(t)=0$ elsewhere.

(f).  $q_{i}(t)\equiv 0$ for $i\ne 3$. $q_{3}(t)=C$ on $[s,\tau)$ for some $s\in \{\bar{\tau},...,m\}\cap U(E)$, and $q_{3}(t)=0$ elsewhere.

In both cases, we require $E(s)-E(s-)-q(s)\ge0$.

 For $m\ge n, m,n\in \Z_{\ge 1}$, denote the collection of all limiting second-order configurations in cases (e)--(f), whose paths have total length $m+n$, by $Conf_{2}(m,n)$.
 \end{definition}

 Again by definition, all such configurations satisfy the properties of block functions.

 Recall that in classical free probability, the limiting covariance is expressed as a combination of two parts: a polynomial in the $\kappa_d$'s and a polynomial in both the $\kappa_d$'s and the $\kappa_{k_1,k_2}$'s.  Next we define two types of functionals on configurations that correspond to these two parts, respectively.

\begin{definition}\label{def:functional1}
    Given a configuration $(E,\vec{q})\in Conf_{1}(m,n)$, for $l=1,2,...$, let $I^{(1)}_{l}(E,\vec{q})$ be a functional of the configuration, such that
    $$I^{(1)}_{l}(E,\vec{q})=(-1)^{m(\vec{q})}\prod_{B_i\in E}\kappa_{|B_{i}|}\sum_{i_{1}<...<i_{l}}\prod_{j=1}^{l}[E(i_{j}-)-q(i_{j}-)],$$
    where $B_{i}$ gives the size of the up and flat steps in $U(E)$, $m(\vec{q})$ counts the total number of jumps of $q_{i}$'s of negative size and $i_{1},...,i_{l}\in P(E,\vec{q})$.
\end{definition}

\begin{definition}\label{def:functional2}
    Given a configuration $(E,\vec{q})\in Conf_{2}(m,n)$, let $s=s(E,\vec{q})$ be the unique time $t\le m$ where $q(t)-q(t-)=C>0$ (see Definition \ref{def:secondorderconf}), and let $k_{1}(E,\vec{q})=E(s)-E(s-)-q(s)+1$, $k_{2}(E)=C(E)$. For $l=1,2,...$, let $I^{(2)}_{l}(E,\vec{q})$ be a functional of the configuration, such that
    $$I^{(2)}_{l}(E,\vec{q})=(-1)^{m(\vec{q})}\frac{\kappa_{k_{1}(E,\vec{q}),k_{2}(E)}}{k_{2}(E)}\prod_{B_i\in E,\ t(B_{i})\ne s(E,\vec{q})}\kappa_{|B_{i}|}\sum_{i_{1}<...<i_{l}}\prod_{j=1}^{l}[E(i_{j}-)-q(i_{j}-)],$$
    where $B_{i}$ gives the size of the up and flat steps in $U(E)$, $m(\vec{q})$ counts the total number of discontinuity points of $q_{i}$'s, where $q_{i}(t)-q_{i}(t-)<0$, and $i_{1},...,i_{l}\in P(E,\vec{q})$.
\end{definition}

After the preparation above, we state the two main theorems of this paper, which constitute the ``if'' parts of Theorems~\ref{thm:intro-fixed-iff} and~\ref{thm:intro-high-iff}.

\begin{theorem}\label{thm:fixedtempclt}
    In the fixed-temperature regime, if
    for all $d\in \Z_{\ge 1}$, $k_{1},k_{2}\in \Z_{\ge 1}$,
    \begin{equation}\label{eq_cltfixed}
        \begin{split}
            \frac{d}{N}c_{(d)}&\rightarrow \theta^{1-d}\kappa_{d};\\
            k_{1}k_{2}(1+\mathbf{1}_{k_1=k_2})c_{(k_{1},k_{2})}&\rightarrow \theta^{1-k_{1}-k_{2}}\kappa_{k_{1},k_{2}};\\
            c_{\lambda}&\rightarrow 0\quad\quad \text{for}\ \ell(\lambda)\ge 3,
        \end{split}
    \end{equation}
    when $N\rightarrow\infty$, then the CLT holds. Moreover, the limiting covariance $\alpha_{m,n}$ is given in terms of first-order free cumulants $\kappa_{l}$'s and second-order free cumulants $\kappa_{k_{1},k_{2}}$'s defined as above, by the equalities
    \begin{equation}\label{eq_fixedtempcov}
    \begin{split}
        \alpha_{m,n}=&\frac{1}{\theta}\sum_{(E,\vec{q})\in Conf_{1}(m,n)}I^{(1)}_{1}(E,\vec{q})+\frac{1}{\theta}\sum_{(E,\vec{q})\in Conf_{2}(m,n)}I^{(2)}_{1}(E,\vec{q})\\
                    =&\frac{1}{\theta}\sum_{E\in L(m,n)}A(E)B(E)C(E)\prod_{B_i\in E}\kappa_{|B_{i}|}\\
                    &+\frac{1}{\theta}\sum_{E\in L(m,n)}\sum_{s\in \{\bar{\tau},...,m\}\cap U(E)}B(E)\kappa_{k_{1}(E),k_{2}(E)}\prod_{B_i\in E,\ t(B_{i})\ne s}\kappa_{|B_{i}|},
        \end{split}
    \end{equation}
    where $k_{1}(E):=E(s)-E(s-)-C(E)+1$, and $\kappa_{k_{1}(E),k_{2}(E)}:=0$ if $k_{1}(E)\le 0$.
\end{theorem}

For $\theta=\frac{1}{2},1$, this result is classical. The covariance--cumulant formula in terms of non-crossing annuli and marked pairs of non-crossing partitions is due to Collins, Mingo, \'{S}niady, and Speicher \cite{CMSS07}; see also \cite[Chapter 5]{MS}. We will show in Section \ref{sec:comb} that our Eqn.~\eqref{eq_fixedtempcov} is equivalent to the classical formula when $\theta=\frac{1}{2},1$, as it should be. Since our formula depends on $\theta>0$ only through a constant, it identifies the second-order free convolution for all fixed-temperature $\theta$, for which both the classical formula and the alternative analytic characterization in \cite[Section 5.5]{MS} also apply.

Using a similar approach, we prove the CLT for the high-temperature regime and provide an analogous combinatorial covariance--cumulant formula.

\begin{theorem}\label{thm:hightempclt}
    In the high-temperature regime, if
    for all $d\in \Z_{\ge 1}$, $k_{1},k_{2}\in \Z_{\ge 1}$,
    \begin{equation}\label{eq_clthigh}
        \begin{split}
            d\cdot c_{(d)}&\rightarrow \kappa^{\gamma}_{d};\\
           Nk_{1}k_{2}(1+\mathbf{1}_{k_1=k_2})c_{(k_{1},k_{2})}&\rightarrow \kappa^{\gamma}_{k_{1},k_{2}};\\
          N^{\ell(\lambda)/2}  c_{\lambda}&\rightarrow 0\quad\quad \text{for}\ \ell(\lambda)\ge 3.
        \end{split}
    \end{equation}
    when $N\rightarrow\infty$, $\theta\rightarrow 0$, $N\theta\rightarrow\gamma>0$, then the CLT holds. Moreover, the limiting covariance $\alpha_{m,n}$ is given in terms of first-order free  $\gamma$-cumulants $\kappa^{\gamma}_{l}$'s and second-order free $\gamma$-cumulants $\kappa^{\gamma}_{k_{1},k_{2}}$'s defined as above, by the equalities
    \begin{equation}\label{eq_hightempcov}
\begin{split}
\alpha_{m,n}
={}&
\sum_{(E,\vec q)\in Conf_1(m,n)}
\sum_{k=1}^{|P(E,\vec q)|}
\gamma^{|D(E)|-k}
I_k^{(1)}(E,\vec q;\gamma)
\\
&+
\sum_{(E,\vec q)\in Conf_2(m,n)}
\sum_{k=1}^{|P(E,\vec q)|}
\gamma^{|D(E)|-k}
I_k^{(2)}(E,\vec q;\gamma).
\end{split}
\end{equation}
    where $I_{k}^{(1)}(E,\vec{q};\gamma)$, $I_{k}^{(2)}(E,\vec{q};\gamma)$ are defined in the same way as in Definitions~\ref{def:functional1} and~\ref{def:functional2}, except that we replace $\kappa_{l}$'s, $\kappa_{k_{1},k_{2}}$'s by $\kappa_{l}^{\gamma}$'s, $\kappa_{k_{1},k_{2}}^{\gamma}$'s respectively.
\end{theorem}

\begin{remark}
    From the above two statements, one can see that $$\kappa_{k_{1},k_{2}}=\kappa_{k_{2},k_{1}}, \quad \kappa^{\gamma}_{k_{1},k_{2}}=\kappa^{\gamma}_{k_{2},k_{1}}.$$
    Hence it suffices to consider $k_{1}\ge k_{2}$. We leave the order of $k_{1},k_{2}$ free since this coordinates better with the later presentation.
\end{remark}

\section{Action of Dunkl operators}\label{sec:dunklaction}
We split the proofs of Theorems \ref{thm:fixedtempclt} and \ref{thm:hightempclt} into two parts. The corresponding laws of large numbers already give the first-cumulant limits. Thus, by the moment--cumulant relations and Definitions~\ref{def:fixedtempclt} and~\ref{Definition_CLT_sat}, it remains to establish two assertions for the power sums under the respective scalings: their second joint cumulants, equivalently their covariances, converge to the limits stated in the theorems, and every joint cumulant of order at least three converges to zero. We prove the former in this section and the latter in Section~\ref{sec:vanishing-higher-cumulants}. Throughout these two parts, $\mathfrak c_l(X_1,\ldots,X_l)$ denotes the classical joint cumulant of $X_1,\ldots,X_l$; this notation is distinct from $c_\lambda$, which continues to denote the BGF coefficient from Definition \ref{def:clambda}.

Recall the type A Dunkl operators
\begin{equation}
\D_i := \partial_{i}+\theta\cdot\sum_{j\colon j\ne i}\frac{1}{x_{i}-x_{j}}(1-\sigma_{ij}), \quad\text{for } i=1,2,...,N,
\end{equation}
defined in~\cite{D}.
They pairwise commute, i.e.~$\D_i\D_j=\D_j\D_i$, for any $1\le i,j\le N$.

Consider the operators
\[
\mathcal{P}_k := \sum_{i=1}^N(\D_i)^k,\qquad
\mathcal{\widetilde{P}}_k := \frac{\mathcal{P}_k}{N^k},\quad k\ge 1.
\]
The symbols with tildes will be used for the fixed-temperature regime and those without tildes for the high-temperature regime.

For $\mu_N$ satisfying the standing assumptions of Section~\ref{sec:bgf}, consider the polynomial observables
\[
p_k(\vec{a}) := \sum_{i=1}^N{(a_i)^k},\qquad
\widetilde{p}_k(\vec{a}) := \frac{p_k(\vec{a})}{N^k},\quad k\ge 1.
\]

Throughout this section we use
\[
F_{N,\theta}(\vec{x}):=\log G_{N,\theta}(\vec{x})
   =\sum_{\lambda\in\mathbb Y}c_\lambda p_\lambda(\vec{x}).
\]
Thus every coefficient below is the BGF coefficient $c_\lambda$ from Definition~\ref{def:clambda}; its dependence on $N$ and $\theta$ is suppressed.

For fixed-temperature scaling, we have
\begin{multline}\label{eq_covthroughdunkl}
\Cov\big(\widetilde{p}_m(\vec{a}), \widetilde{p}_n(\vec{a})\big) = \left[\mathcal{\widetilde{P}}_m\mathcal{\widetilde{P}}_n\exp(\ln G_{\theta,N})-\mathcal{\widetilde{P}}_m\exp(\ln G_{\theta,N})\cdot\mathcal{\widetilde{P}}_n\exp(\ln G_{\theta,N})\right]\Big|_{\vec{x}=0}\\
= N(N-1)\left(\frac{\D_2}{N}\right)^{n}\left(\frac{\D_1}{N}\right)^{m}\exp(\ln G_{\theta,N})\Big|_{\vec{x}=0}+N\left(\frac{\D_{1}}{N}\right)^{m+n}\exp(\ln G_{\theta,N})\Big|_{\vec{x}=0}\\
-N^{2}\left(\frac{\D_1}{N}\right)^m\exp(\ln G_{\theta,N})\Big|_{\vec{x}=0}\left(\frac{\D_1}{N}\right)^{n}\exp(\ln G_{\theta,N})\Big|_{\vec{x}=0}.
\end{multline}

For high-temperature scaling, we have
\begin{multline}\label{eq_covthroughdunkl2}
N^{-1}\Cov\big(p_m(\vec{a}), p_n(\vec{a})\big) =N^{-1} \left[\mathcal{P}_m\mathcal{P}_n\exp(\ln G_{\theta,N})-\mathcal{P}_m\exp(\ln G_{\theta,N}) \cdot\mathcal{P}_n\exp(\ln G_{\theta,N})\right]\Big|_{\vec{x}=0}\\
= (N-1)\D_2^n\D_1^m\exp(\ln G_{\theta,N})\Big|_{\vec{x}=0}+\D_1^{m+n}\exp(\ln G_{\theta,N})\Big|_{\vec{x}=0}\\
-N\D_1^m\exp(\ln G_{\theta,N})\Big|_{\vec{x}=0}\D_1^n\exp(\ln G_{\theta,N})\Big|_{\vec{x}=0}.
\end{multline}

Let $h(\vec{x})$ be a polynomial of $x_{1},\dots,x_{N}$, let $\kappa$ be any partition, and recall that $p_{\kappa}(\vec{x})=\prod_{i=1}^{\ell(\kappa)}p_{\kappa_{i}}(\vec{x})$. Straightforward calculations yield the following equalities:
\begin{align*}
\partial_i\left( h(\vec{x})\,p_{\kappa}(\vec{x})\, e^{F_{N,\theta}(\vec{x})} \right) =& \ \left[\partial_i h(\vec{x})p_\kappa(\vec{x})+h(\vec{x})\sum_{s=1}^{\ell(\kappa)}{ \kappa_s x_i^{\kappa_s-1}\prod_{j\ne s}{p_{\kappa_j}(\vec{x})} } \right]\cdot e^{F_{N,\theta}(\vec{x})}\\
&+ \bigg(\sum_{d=1}^\infty{dc_{(d)} x_i^{d-1}}\bigg) h(\vec{x})\,p_{\kappa}(\vec{x})\ e^{F_{N,\theta}(\vec{x})}\\
&+ \bigg(\sum_{\la\colon\ell(\la)\ge 2}{ c_\la \sum_{j=1}^{\ell(\la)}{ \la_j x_i^{\la_j-1}\prod_{k\ne j}{p_{\la_k}(\vec{x})} } }\bigg)h(\vec{x})\,p_{\kappa}(\vec{x}) e^{F_{N,\theta}(\vec{x})},
\end{align*}

\begin{equation*}
\frac{1-\sigma_{ij}}{x_i - x_j} \left( h(\vec{x})\,p_{\kappa}(\vec{x})\,e^{F_{N,\theta}(\vec{x})} \right) = \frac{1-\sigma_{ij}}{x_i - x_j} \big( h(\vec{x}) \big) p_{\kappa}(\vec{x})\,e^{F_{N,\theta}(\vec{x})},
\end{equation*}

\begin{equation}\label{eq_sigmaaction}
\frac{1-\sigma_{ij}}{x_i - x_j} \Big( x_i^{m_i} x_j^{m_j} \Big) = \begin{cases}
    -\sum_{a=m_i}^{m_j-1} {x_i^a x_j^{m_i+m_j-1-a}}, & \text{if }m_i<m_j,\\
    \\
    \sum_{a=m_j}^{m_i-1} {x_i^a x_j^{m_i+m_j-1-a}}, & \text{if }m_i\ge m_j.
\end{cases}
\end{equation}

\medskip

Based on the above calculation,
we then introduce the following linear operators on polynomials, uniquely defined by their actions on monomials $x_1^{m_1}\cdots x_N^{m_N}$, for any $m_1,...,m_N\in \Z_{\ge 0}$.
\begin{itemize}
    \item For $1\le i\le N$ and $d\ge1$, $$U^{(d)}_{i}(x_{1}^{m_{1}}\cdots x_{N}^{m_{N}}):=dc_{(d)}x_{i}^{d-1}\cdot x_{1}^{m_{1}}\cdots x_{N}^{m_{N}}.$$
    \item For $1\le i\le N$, and any partition $\la$ such that $\ell(\lambda)\ge 2$ and $|\la|\le |\mu|$,
    $$U_{i}^{(\lambda)}(x_{1}^{m_{1}}\cdots x_{N}^{m_{N}}) := c_\la \sum_{j=1}^{\ell(\la)}{ \la_j x_i^{\la_j-1}\prod_{k\ne j}{p_{\la_k}(\vec{x})} }\cdot x_{1}^{m_{1}}\cdots x_{N}^{m_{N}}.$$
    In particular, if $\la=(\la_1,\la_2)$, let \begin{align*}U_i^{(\lambda,1,k)}(x_{1}^{m_{1}}\cdots x_{N}^{m_{N}}) := c_\la  \la_1 x_i^{\la_1-1}x_k^{\la_2} \cdot x_{1}^{m_{1}}\cdots x_{N}^{m_{N}},\\
    U_i^{(\lambda,2,k)}(x_{1}^{m_{1}}\cdots x_{N}^{m_{N}}) := c_\la  \la_2 x_i^{\la_2-1}x_k^{\la_1} \cdot x_{1}^{m_{1}}\cdots x_{N}^{m_{N}},\end{align*}
    for $k=1,2,...,N$. Then $U^{(\la)}_i=\sum_{k=1}^{N}\bigl(U^{(\la,1,k)}_{i}+U^{(\la,2,k)}_i\bigr)$.
    \item For $1\le i\le N$, $1\le j\le N$, $j\ne i$,  $$d_{i,j}(x_{1}^{m_{1}}\cdots x_{N}^{m_{N}}):=\begin{cases}
        x_{1}^{m_{1}}\cdots x_{i-1}^{m_{i-1}}x_{i}^{m_{i}-1}x_{i+1}^{m_{i+1}}\cdots x_{N}^{m_{N}},& m_{i}>m_{j};\\
        0,& m_{i}\le m_{j}.
    \end{cases}$$
\item For $1\le i\le N$, $1\le j\le N$, $j\ne i$, $a\in \Z_{\ge 1}$,
\begin{equation}\label{eq_jumpaction}\tilde{d}^{(a)}_{i,j}(x_{1}^{n_{1}}\cdots x_{N}^{n_{N}}):=\begin{cases}
    x_{1}^{n_{1}}\cdots x_{N}^{n_{N}}\cdot {x_i^{-a} x_j^{a-1}}, & \text{if } n_i>n_j,\ 1< a\le n_{i}-n_{j}\\
    -x_{1}^{n_{1}}\cdots x_{N}^{n_{N}}\cdot x_i^{a-1} x_j^{-a}, & \text{if } n_i < n_j,\ 1\le a\le n_{j}-n_{i}\\
    0, & \text{otherwise}.
\end{cases}\end{equation}
\item For $1\le i\le N$, $$\partial_{i}(x_{1}^{m_{1}}\cdots x_{N}^{m_{N}})=m_{i}\cdot d_{i}(x_{1}^{m_{1}}\cdots x_{N}^{m_{N}}),$$
where $$d_{i}(x_{1}^{m_{1}}\cdots x_{N}^{m_{N}}):= \mathbf{1}_{\{m_i\ge 1\}}\cdot x_{1}^{m_{1}}\cdots x_{i-1}^{m_{i-1}}x_i^{m_i-1}x_{i+1}^{m_{i+1}}\cdots x_{N}^{m_{N}}.$$
\end{itemize}

\medskip

Then for a partition\footnote{In this text, we only need the case $\mu=(m,n)$ and $\mu=(k)$.} $\mu=(\mu_{1},...,\mu_{s})$, we have
\begin{equation}
\begin{split}
    &\Big[ \D_s^{\mu_s}\cdots\D_2^{\mu_2}\D_1^{\mu_1} \Big] e^{ F_{N,\theta}(\vec{x}) } \,\Big|_{\vec{x}=(0^N)}\\
    =&\left(\sum_{d}U^{(d)}_{s}+\sum_{\la=(\la_1,\la_2)}\sum_{k=1}^{N}(U^{(\la,1,k)}_{s}+U^{(\la,2,k)}_s)+\theta\sum_{j\ne s} d_{s,j}+\theta\sum_{j\ne s}\sum_{a\ge 1}\Tilde{d}^{(a)}_{s,j}+\partial_{s}+\sum_{\ell(\la)\ge 3}U^{(\la)}_s\right)^{\mu_{s}}\cdots \\
&\left(\sum_{d}U^{(d)}_{1}+\sum_{\la=(\la_1,\la_2)}\sum_{k=1}^{N}(U^{(\la,1,k)}_{1}+U^{(\la,2,k)}_1)+\theta\sum_{j\ne 1} d_{1,j}+\theta\sum_{j\ne 1}\sum_{a\ge 1}\Tilde{d}^{(a)}_{1,j}+\partial_{1}+\sum_{\ell(\la)\ge 3}U^{(\la)}_1\right)^{\mu_{1}} (1)\Big|_{\vec{x}=(0^N)}\\
=&\sum_{T(t):1\le t\le |\mu| }T(|\mu|)\cdots T(1)(1)\Big|_{\vec{x}=(0^N)},
\end{split}
\end{equation}
where the summation in the last line is over the finitely many choices of the operators $T(t)$. More precisely, $T(t)$ is chosen among $U^{(d)}_{i}$, $U^{(\la,j,k)}_{i}\ (\ell(\lambda)=2)$, $\theta\cdot d_{i,j}$, $\theta\cdot \tilde{d}^{(a)}_{i,j}$, $\partial_{i}$ and $U^{(\la)}_{i}\ (\ell(\la)\ge 3)$  for $\sum_{j=1}^{i-1}\mu_{j}+1\le t\le \sum_{j=1}^{i}\mu_{j}$. 

In the order displayed, these are the following six types:
\[
\begin{array}{c|c}
\mathrm{I} & U_i^{(d)}\\
\mathrm{II} & U_i^{(\lambda,j,k)},\ \ell(\lambda)=2\\
\mathrm{III} & \theta d_{i,j}\\
\mathrm{IV} & \theta\widetilde d_{i,j}^{(a)}\\
\mathrm{V} & \partial_i\\
\mathrm{VI} & U_i^{(\lambda)},\ \ell(\lambda)\ge3.
\end{array}
\]
We denote the corresponding classes by $J_{\mathrm I},\ldots,J_{\mathrm{VI}}$. Type~II is the monomial expansion of the grouped operator $U_i^{(\lambda)}$ with $\ell(\lambda)=2$.

In this way, we may view each term in the sum as a discrete-time process of polynomials in $x_1,\ldots,x_N$ starting at time zero from the monomial $P(0)=1$ with the polynomial at time $t>0$ given by $P(t)=T(t)\cdots T(1)(1),\ t=1,2,\ldots,|\mu|$. Moreover, if none of the operators $T(t)$ is of Type VI, each $P(t)$ is a monomial.

Next, we introduce a way to visualize each monomial sequence $P(t),t=1,2,...,|\mu|$, where each $T(t)$ is of one of Types I--V.

\begin{definition}
    For $N\in \Z_{\ge 1}$, let $P(\vec{x})=c\cdot x_{1}^{n_{1}}\cdots x_N^{n_{N}}$ be a monomial in $x_1,\ldots,x_N$. We denote
    $$\operatorname{Deg}(P):=\sum_{i=1}^{N}n_{i},\quad \operatorname{Deg}_{i}(P):=n_{i}\quad \text{for}\ i=1,2,\ldots,N.$$
\end{definition}

 Observe that after the action of $T(t)\ (t\in \llbracket 1,|\mu|\rrbracket)$, if $P(t)\ne 0$, the degree has either
\begin{itemize}
\item increased by $d-1$ if $T(t) = U^{(d)}_{i}$ or $U^{(\lambda)}_{i}$ for some $i$, $d\ge 1$ or some two-row partition $\lambda$ where $|\lambda|=d$, or
\item decreased by 1 otherwise.
\end{itemize}

\begin{definition}\label{def:walk}
     Fix $N\in \Z_{\ge 1}$ and a partition $\mu$ of length 1 or 2. For each sequence of monomials $P(t)=T(t)\cdots T(1)(1),\ t=1,2,\ldots,|\mu|$, such that $P(|\mu|)|_{\vec{x}=0}\ne 0$, let $E(t):[0,|\mu|]\rightarrow \Z_{\ge 0}$ be a c\`adl\`ag function defined by
        \begin{equation}
        E(t)=\operatorname{Deg}(P(\lfloor t\rfloor)).
        \end{equation}
        We refer to $E(t)$ as a walk. 
\end{definition}

If $\mu=(k)$, we regard the walk as an element of $L(k)$. If $\mu=(m,n)$, we regard it as an element of $\tilde{L}(m,n)$, with a distinguished splitting point at time $m$. If $\mu=(m+n)$, we regard it as an element of $L(m+n)$,  where no splitting point is part of the data.

While the walk records the total degree, we also need to keep track of the degree of each individual variable. Due to the actions of the operators $\tilde{d}^{(a)}_{i,j}$, each $\operatorname{Deg}_{i}(P(t))$ changes in certain nontrivial ways, which we record by a collection of \emph{blocks} $q_{i}(t)$. More precisely, we have the following. \begin{definition}\label{def:block}
        Fix $N\in \Z_{\ge 1}$ and a partition $\mu$ of length 1 or 2. For $\ell(\mu)=1$ we define $\mu_2=0$. For each sequence of monomials $P(t)=T(t)\cdots T(1)(1),\ t=1,2,\ldots,|\mu|$,  let $q_{j}(t):[0,|\mu|]\rightarrow \Z_{\ge 0}$ be a piecewise function, which is c\`adl\`ag on $[0,|\mu|]\setminus\{0,\mu_{1},...,|\mu|\}$, defined by
        \begin{equation}
            q_{j}(t)=\begin{cases}
            \operatorname{Deg}_{j}(P(\lfloor t\rfloor)),\quad &j\ge 3,\ t\in [0,|\mu|]\\
                \operatorname{Deg}_{j}(P(\lfloor t\rfloor)),\quad &j=1,\ t\in [0,|\mu|]\setminus [0,\mu_1]\\
                \operatorname{Deg}_{j}(P(\lfloor t\rfloor)),\quad &j=2,\ t\in [0,|\mu|]\setminus (\mu_{1},\mu_1+\mu_2]\\
                0,\quad &\text{otherwise}.
            \end{cases}
        \end{equation}
        Let $$q(t)=\sum_{j=1}^N q_{j}(t),\quad t\in [0,|\mu|].$$
         We call $\vec{q}(t)=(q_{1}(t),\ldots,q_{N}(t))$ the blocks of $\{P(t)\}_{t=0}^{|\mu|}$.
  \end{definition}

We call $(E,\vec q)$ the configuration of the monomial sequence. During the $i$th interval $\sum_{j<i}\mu_j<t\le\sum_{j\le i}\mu_j$, the main variable is $x_i$. The path records total degree, while the blocks record how this degree is distributed among variables that are not currently main. Recall that $P(E,\vec q)=D(E)\setminus J(E,\vec q)$ is the set of down-step times that are not block jumps. At each such time, the same decrease of total degree may be produced either by Type~III or by Type~V. Thus, after auxiliary labels are fixed, $(E,\vec q)$ together with the subset $t(\partial)\subset P(E,\vec q)$ of Type~V times determines the monomial branch; the complementary times in $P(E,\vec q)$ are Type~III times.

  The right discontinuities at the splitting point merely record the change of main variable. The left discontinuities caused by operator actions are of the following two types. If the discontinuity is given by an operator $\tilde{d}^{(a)}_{i,j}$, $q_{j}(t)$ increases by $a-1\ (a\ge 2)$ or decreases by $a\ (a\ge 1)$, and $E(t)$ decreases by 1. We call this a \emph{jump of $x_{j}$}. If the discontinuity is given by an operator $U^{(\lambda,1,k)}_{i}$ (or $U^{(\la,2,k)}_i$ resp.) for some $k\ne i$, we have $q_k(t)-q_k(t-)=\la_2$ (or $\la_1$ resp.). We call this a \emph{nontrivial up step}.

\begin{remark}
   For $i\ne 2$, the last jump of each $x_i$  must send $q_{i}(t)$ back to 0, and by Eqn.~(\ref{eq_jumpaction}), $E(t-)=q(t-)$.
\end{remark}

For a partition $\mu=(\mu_{1},\ldots,\mu_{\ell(\mu)})$ and indices
$\mathbf{i}=(i_{1},\ldots,i_{\ell(\mu)})$, let
$\mathscr{P}(\mu;\mathbf{i})$ denote the collection of all polynomial
sequences
\[
\vec{P}=\{P(t)\}_{t=0}^{|\mu|},
\qquad
P(0)=1,
\qquad
P(t)=T(t)\cdots T(1)(1),
\]
obtained from the operator expansion above, where the main variable is
$x_{i_{r}}$ for
\[
\mu_{1}+\cdots+\mu_{r-1}<t\le \mu_{1}+\cdots+\mu_{r}.
\]
In particular, introduce the three collections
\[
\mathscr{P}_{m,n}^{12}
    :=\mathscr{P}\big((m,n);(1,2)\big),
\qquad
\mathscr{P}_{m+n}^{1}
    :=\mathscr{P}\big((m+n);(1)\big),
\]
and
\[
\mathscr{P}_{m,n}^{\times}
    :=\mathscr{P}\big((m);(1)\big)
      \times\mathscr{P}\big((n);(1)\big).
\]
We regard these as disjoint collections. For an element in one of these
collections, define its signed covariance weight by
\[
\operatorname{wt}_{N}(\vec{P})
:=
\begin{cases}
\displaystyle
\frac{N(N-1)}{N^{m+n}}\,
P(m+n)\Big|_{\vec{x}=0},
& \vec{P}\in\mathscr{P}_{m,n}^{12},
\\[9pt]
\displaystyle
\frac{N}{N^{m+n}}\,
P(m+n)\Big|_{\vec{x}=0},
& \vec{P}\in\mathscr{P}_{m+n}^{1},
\\[9pt]
\displaystyle
-\frac{N^{2}}{N^{m+n}}\,
P^{(1)}(m)\Big|_{\vec{x}=0}\,
P^{(2)}(n)\Big|_{\vec{x}=0},
&
\vec{P}=(\vec{P}^{(1)},\vec{P}^{(2)})
\in\mathscr{P}_{m,n}^{\times}.
\end{cases}
\]

\begin{proposition}
\label{prop:covariance-weighted-expansion}
Let $\mathscr{P}_{m,n}^{>0}$ be the collection of all
$\vec{P}\in\mathscr{P}_{m,n}^{12}\sqcup\mathscr{P}_{m+n}^{1}$ such that

\begin{enumerate}[label=\textup{(\roman*)}]
    \item none of the operators occurring in $\vec{P}$ is of Type VI; and
    \item the associated walk-block configuration $(E,\vec{q})$ satisfies
    \[
    E(m)>0.
    \]
\end{enumerate}

Let $\mathscr{P}_{m,n}^{\mathrm{VI}}$ be the collection of all terms from
\[
\mathscr{P}_{m,n}^{12}
\sqcup
\mathscr{P}_{m+n}^{1}
\sqcup
\mathscr{P}_{m,n}^{\times}
\]
which contain at least one Type VI operator; in the product collection,
this means that at least one of the two operator sequences contains a
Type VI operator. Then
\begin{equation}\label{eq:covariance-weighted-expansion}
\Cov\big(\widetilde{p}_{m}(\vec{a}),
         \widetilde{p}_{n}(\vec{a})\big)
=
\sum_{\vec{P}\in\mathscr{P}_{m,n}^{>0}}
    \operatorname{wt}_{N}(\vec{P})
+
\sum_{\vec{P}\in\mathscr{P}_{m,n}^{\mathrm{VI}}}
    \operatorname{wt}_{N}(\vec{P}).
\end{equation}
\end{proposition}

\begin{proof}
Expanding each Dunkl action in Eqn.~\eqref{eq_covthroughdunkl} in terms of
the operators of Types I--VI gives
\begin{align*}
\Cov\big(\widetilde{p}_{m}(\vec{a}),
         \widetilde{p}_{n}(\vec{a})\big)
={}&
\frac{N(N-1)}{N^{m+n}}
\sum_{\vec{P}\in\mathscr{P}_{m,n}^{12}}
P(m+n)\Big|_{\vec{x}=0}
\\
&+
\frac{N}{N^{m+n}}
\sum_{\vec{P}\in\mathscr{P}_{m+n}^{1}}
P(m+n)\Big|_{\vec{x}=0}
-
\frac{N^{2}}{N^{m+n}}
\sum_{\substack{
\vec{P}^{(1)}\in\mathscr{P}((m);(1))\\
\vec{P}^{(2)}\in\mathscr{P}((n);(1))
}}
P^{(1)}(m)\Big|_{\vec{x}=0}\,
P^{(2)}(n)\Big|_{\vec{x}=0}.
\end{align*}

It remains to identify the cancellation of the terms without Type VI
operators whose associated walk satisfies $E(m)=0$. For such a term,
$P(m)$ is a monomial of total degree zero, and is therefore a constant.
Consequently, the operator sequence splits at time $m$ into two
independent sequences of lengths $m$ and $n$.

Conversely, every ordered pair of operator sequences of lengths $m$ and
$n$, neither of which contains a Type VI operator, gives both an element
of $\mathscr{P}_{m,n}^{12}$ with $E(m)=0$ and, after concatenation, an
element of $\mathscr{P}_{m+n}^{1}$ with $E(m)=0$. In the first case one
uses the symmetry-preserving relabeling of the main variable from
$x_{1}$ to $x_{2}$ in the second segment. The internal weights of these
three corresponding terms are identical, while their exterior
multiplicities satisfy
\[
N(N-1)+N-N^{2}=0.
\]
Hence all terms without Type VI operators and with $E(m)=0$ cancel
exactly in Eqn.~\eqref{eq_covthroughdunkl}. The remaining terms are
precisely the two collections appearing on the right-hand side of
Eqn.~\eqref{eq:covariance-weighted-expansion}.
\end{proof}

The first sum in Eqn.~\eqref{eq:covariance-weighted-expansion} is the
configuration contribution that will be analyzed below. The second sum
is retained for the moment as an error term. We will prove separately
that, under the assumptions of the corresponding fixed-temperature
analysis,
\[
\sum_{\vec{P}\in\mathscr{P}_{m,n}^{\mathrm{VI}}}
\operatorname{wt}_{N}(\vec{P})
\sim o(1),
\qquad N\to\infty.
\]
Thus, operator sequences containing a Type VI action do not contribute
to the limiting covariance.

\subsection{Fixed-temperature covariance in Theorem \ref{thm:fixedtempclt}}\label{sec:fixedtempcov}

We now analyze the three signed sums in Proposition~\ref{prop:covariance-weighted-expansion} under the fixed-temperature assumptions. Lemma~\ref{lem:typeII-typeVI-fixed} removes Type~VI and all but at most one grouped Type~II action. The contribution with no Type~II action yields the polynomial in first-order free cumulants; the contribution with one Type~II action yields the term involving one second-order free cumulant.

We assume that for all $d\in \Z_{\ge 1}$, $k_{1},k_{2}\in \Z_{\ge 1}$, as $N\rightarrow\infty$,
    \begin{equation}
        \begin{split}
            \frac{d}{N}c_{(d)}&\rightarrow \theta^{1-d}\kappa_{d};\\
            k_{1}k_{2}(1+\mathbf{1}_{k_1=k_2})c_{(k_{1},k_{2})}&\rightarrow \theta^{1-k_{1}-k_{2}}\kappa_{k_{1},k_{2}};\\
            c_{\lambda}&\rightarrow 0\quad\quad \text{for}\ \ell(\lambda)\ge 3.
        \end{split}
    \end{equation}

The first observation concerns scaling. In the fixed-temperature regime $\theta$ is a constant, so we only keep track of the degree of $N$. Each action of $\frac{\D_i}{N}$ gives $N^{-1}$. 

\begin{lemma}\label{lem:typeII-typeVI-fixed}
In the fixed-temperature regime, the following statements hold for the
operator expansion of the covariance in Eqn.~\eqref{eq_covthroughdunkl}.

\begin{enumerate}[label=\textup{(\alph*)}]
    \item The total weight of all operator strings containing at least one
    Type VI operator is \(o(1)\) as \(N\to\infty\).

    \item After excluding operator strings containing Type VI operators, the
    total weight of all strings containing at least two Type II operators is
    \(o(1)\) as \(N\to\infty\). Consequently, an operator string contributing
    to the limiting covariance contains at most one Type II operator. If such
    an operator occurs, then, at leading order, it is paired with exactly one
    Type V operator.
\end{enumerate}
\end{lemma}

\begin{proof}
The argument is carried out after grouping together all the monomial
terms arising from the same operator \(U_i^{(\lambda)}\). In particular,
when \(\ell(\lambda)=2\), we first write
\[
U_i^{(\lambda)}
=
\sum_{k=1}^{N}
\left(
U_i^{(\lambda,1,k)}
+
U_i^{(\lambda,2,k)}
\right).
\]
Thus, the assertions below concern the total weight of the corresponding
grouped terms, rather than each individual monomial sequence.

Let \(\lambda\) be a partition with
\[
r:=\ell(\lambda)\ge 2.
\]
Each summand in \(U_i^{(\lambda)}\) has the form
\[
c_{\lambda}\lambda_j x_i^{\lambda_j-1}
\prod_{\substack{1\le k\le r\\k\ne j}}
p_{\lambda_k}(\vec{x}),
\]
and hence introduces \(r-1\) power-sum factors.

After the action of \(U_i^{(\lambda)}\), split each subsequent Dunkl
operator into its partial-derivative part and its difference part. If
\(S(\vec{x})\) is any product of power sums, then \(S\) is symmetric and
therefore
\[
\frac{1-\sigma_{a,b}}{x_a-x_b}
\bigl(h(\vec{x})S(\vec{x})\bigr)
=
S(\vec{x})
\frac{1-\sigma_{a,b}}{x_a-x_b}h(\vec{x}).
\]
Consequently, the difference part of a subsequent Dunkl operator passes
all the power-sum factors through unchanged. On the other hand,
\[
\partial_a\bigl(h(\vec{x})S(\vec{x})\bigr)
=
\partial_a h(\vec{x})\,S(\vec{x})
+
h(\vec{x})\,\partial_a S(\vec{x}),
\]
and, in each term of the Leibniz expansion of
\(\partial_a S(\vec{x})\), the derivative acts on exactly one power-sum
factor.

Any power-sum factor that is never acted on by a partial derivative
remains present until the end and makes the resulting expression vanish
after setting \(\vec{x}=0\). Hence, for any branch of the monomial
expansion that gives a nonzero contribution at \(\vec{x}=0\), at least
\(r-1\) subsequent Dunkl operators must use their partial-derivative
parts to act on the \(r-1\) power-sum factors introduced by
\(U_i^{(\lambda)}\). After expanding into monomials, these derivative
actions are precisely Type V actions.

More generally, if an operator string contains
\[
U_{i_1}^{(\lambda^{(1)})},
\ldots,
U_{i_s}^{(\lambda^{(s)})},
\]
then every nonzero contribution contains at least
\[
\sum_{a=1}^{s}
\bigl(\ell(\lambda^{(a)})-1\bigr)
\]
Type V actions.

We now combine this observation with the fixed-temperature scaling.
Each occurrence of \(U_i^{(\lambda)}\) comes from an action of
\(\D_i/N\), and therefore carries an explicit factor \(N^{-1}\).
For \(\ell(\lambda)=2\), its coefficient is \(O(1)\), whereas for
\(\ell(\lambda)\ge 3\), its coefficient is \(o(1)\) by the assumptions
of Eqn.~\eqref{eq_cltfixed}. Each required Type V action contributes one
further factor \(O(N^{-1})\). All the remaining operator actions are of
order \(O(1)\) or smaller after the sums over their auxiliary indices
are taken into account, as described in the scaling discussion below.
Finally, the exterior multiplicative factors in
Eqn.~\eqref{eq_covthroughdunkl} are of order at most \(N^2\).

Suppose first that an operator string contains a Type VI operator
\(U_i^{(\lambda)}\), where \(r=\ell(\lambda)\ge 3\). This action,
together with the \(r-1\) Type V actions forced by the preceding
argument, contributes
\[
o(N^{-1})\,O(N^{-(r-1)})\sim o(N^{-r}).
\]
After multiplication by the exterior \(O(N^2)\) factor, its contribution
is
\[
o(N^{2-r})\sim o(1).
\]
Additional Type II, Type V, or Type VI actions can only decrease the
order. This proves part~\textup{(a)}.

Next, suppose that an operator string contains \(s\) Type II operators
and no Type VI operator. Each Type II operator introduces one power-sum
factor, so every nonzero contribution contains at least \(s\) Type V
actions. The Type II and the forced Type V actions therefore contribute
at most
\[
O(N^{-s})\,O(N^{-s})\sim O(N^{-2s}).
\]
After multiplication by the exterior \(O(N^2)\) factor, this is
\[
O(N^{2-2s}),
\]
which tends to zero whenever \(s\ge 2\). Thus, an operator string
contributing to the limiting covariance contains at most one Type II
operator.

If there is exactly one Type II operator, the preceding argument forces
at least one Type V action. One Type II action together with one Type V
action is of order \(O(N^{-2})\), and can therefore survive after
multiplication by the exterior \(N^2\) factor. Any additional Type V
action produces an extra factor \(N^{-1}\) and has vanishing
contribution. Hence, at leading order, the unique Type II operator is
paired with exactly one Type V operator. This proves
part~\textup{(b)}.
\end{proof}

We now record the scaling of the remaining operator types more explicitly.
\begin{itemize}
    \item By Eqn.~\eqref{eq_cltfixed}, the Type I action, which corresponds to a nonnegative increment of $E$, is of order $O(1)$.
    \item The Type III action corresponds to a ``canonical'' down step of $E$, that only decreases the height of the path by 1. There are $N-O(1)$ choices of the variable $x_j$ in $\frac{1-\sigma_{ij}}{x_i-x_j}$. So this is also of order $O(1)$.
   A Type~IV action itself carries $N^{-1}$. If it creates a block with a previously unused auxiliary label $x_j$, there are $N-O(1)$ choices of $j$; this free choice compensates for $N^{-1}$, and we call it a \emph{first jump}. Every later jump of the same block has only $O(1)$ possible labels and heights and therefore costs $O(N^{-1})$; we call it an \emph{auxiliary jump}. The same $O(N^{-1})$ cost applies to a first jump of the prescribed variable $x_2$ before time $m$, because its label is fixed rather than freely chosen.
    \item The Type V action by $\partial_i$ is of order $O(N^{-1})$.
\end{itemize}

By Lemma \ref{lem:typeII-typeVI-fixed}, in the remainder of Section \ref{sec:fixedtempcov}, we will only consider $\vec{P}\in\mathscr{P}_{m,n}^{\mathrm{>0}}$, defined in Proposition \ref{prop:covariance-weighted-expansion}.

\subsubsection{First-order asymptotics}\label{sec:firstorderasymp}

We first treat strings with no Type~II action. Every auxiliary jump and every Type~V action costs $N^{-1}$, whereas a first jump with a freely chosen label has net order $O(1)$. The exterior factor is $O(N^2)$ for the two-index term and $O(N)$ for the single-index term. Hence a leading two-index configuration has at most two costly actions, while a leading single-index configuration has at most one. Because Proposition~\ref{prop:covariance-weighted-expansion} retains only $E(m)>0$, a two-index configuration must contain an auxiliary jump that removes the positive $x_1$-degree after the splitting point; a single-index configuration has no analogous requirement.

From the discussion above, we classify the configurations that might contribute to the limiting covariance into the following groups.
Group~A consists of two-index configurations with one auxiliary jump and no Type~V action, together with single-index configurations with no costly action and no Type~V action. 

Group~B consists of two-index configurations with two auxiliary jumps and no Type~V action, together with single-index configurations with one auxiliary jump and no Type~V action.

Group~C consists of two-index configurations with one auxiliary jump and one Type~V action, together with single-index configurations with one Type~V action and no auxiliary jump. 

Groups~A and B cancel completely. In Group~C, some terms cancel with each other, while the remaining configurations give $I^{(1)}_{1}(E,\vec{q})$.

\subsubsection{Enumeration of first-order configurations}\label{sec:enumeration}
We now enumerate the configurations allowed by this cost bound. In the list, \emph{jump} means every discontinuity of a block, including its first jump; only the jumps classified above as auxiliary cost $N^{-1}$. Up to relabeling variables $x_i$ with $i\ge3$, the following twelve block patterns are the possibilities.

(i) $E\in L(m+n)$, no jump;

(ii) $E\in L(m,n)$, and there is one jump of $x_1$;

(iii) $E\in L(m,n)$, and there is one jump of $x_2$;

(iv) $E\in L(m,n)$, and there are two jumps of $x_i\ (i\ge 3)$.

It will be shown that the above are all the cases that contribute to the limiting covariance. The contributions of the following cases vanish in the limit after certain weight cancellations.

(v) $E\in L(m+n)$, and there are two jumps of $x_i\ (i\ge 2)$;

(vi) $E\in L(m,n)$, and there is one jump of $x_1$ and one jump of $x_2$;

(vii) $E\in L(m,n)$, and there is one jump of $x_1$ and two jumps of $x_i\ (i\ge 3)$;

(viii) $E\in L(m,n)$, and there is one jump of $x_2$ and two jumps of $x_i\ (i\ge 3)$;

(ix) $E\in L(m,n)$, and there are two jumps of $x_1$;

(x) $E\in L(m,n)$, and there are two jumps of $x_2$;

(xi) $E\in L(m,n)$, and there are three jumps of $x_i\ (i\ge 3)$;

(xii) $E\in L(m,n)$, and there are two jumps of $x_i$ and two jumps of $x_j\ (3\le i<j)$.

\begin{definition}\label{def:equiv}
For $E_1,E_2\in L(m,n)\cup L(m+n)$, write $E_1\sim E_2$ when they coincide as functions on $[0,m+n]$. Thus the equivalence forgets whether the common path is produced by the two-index action with a splitting point at $m$ or by the single-index action.
\end{definition}

{Next, we further classify the configurations in the above 12 groups as a collection of disjoint pairs and three triples, such that the Lukasiewicz paths in the same pair/triple are equivalent.}

For brevity, we say a jump is before time $m$, if its time of occurrence is among $t=1,2,...,m$, and a jump is after time $m$ if its time of occurrence is among $m+1,m+2,...,m+n$.
For each pair/triple, there is at least one configuration $(E,\vec{q})$ in it such that $E\in L(m,n)$, and $\vec{q}$ has a jump of $x_1$ or $x_2$. We start with this configuration and transform it into the other elements in the pair/triple, under the following rules:

\medskip
\noindent{\textbf{I.}} If $(E,\vec{q})$ has a unique jump of $x_1$ after time $m$, then replace $E$ by the equivalent path in $L(m+n)$, and remove that jump.

\medskip
\noindent{\textbf{II.}} If $(E,\vec{q})$ has a jump of $x_2$ before time $m$, then replace all the jumps of $x_2$ by the jumps of $x_j$ for some $j\ge 3$. If $q_{j}(t)\equiv 0$ in $(E,\vec{q})$, add a last jump of $x_j$ that sends $\deg(x_{j})$ back to 0.

\medskip
\noindent{\textbf{III.}} If $(E,\vec{q})$ has two jumps of $x_1$ after time $m$, then replace $E$ by the equivalent path in $L(m+n)$, and replace these two jumps by jumps of $x_i$ for some $i\ge 3$.

Each rule is reversible after the common path and the jump times and heights are fixed. Rule~I changes whether the block crossing the splitting point is carried by $x_1$ or absorbed into the single-index realization. Rule~II replaces the prescribed variable $x_2$ before the splitting point by a freely labeled auxiliary block. Rule~III performs the analogous replacement for a two-jump $x_1$-block after the splitting point. These are the only path-preserving changes that alter the exterior multiplicity or the sign, so they organize the cancellations below.

\medskip

More precisely, the following pairs use Rule~\textbf{I}.

(1.1) $E_{1}\in L(m,n)$, and $q_{1}(t)=C$ on $(m,\tau)$.

(1.2) $E_{2}\in L(m+n)$, no jump.

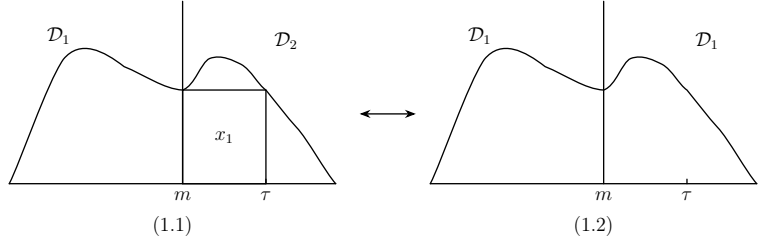
\WBPair
{%
    \WBBlock{\WBxM}{\WBxTau}{0}{\WBC}{x_1}
}
{%
    \def\WBRightDomainLabel{\mathcal D_1}
}
{\textup{(1.1)}}
{\textup{(1.2)}}

\vspace{10pt}

(2.1) $E_{1}\in L(m,n)$. There is one jump of $x_{i}$ before time $m$, another jump of $x_{i}$ at time $t_{1}>\tau$, for some $i\ge 3$, and one jump of $x_{1}$ at time $t_{2}\in [m+1,t_{1}-1]$.

(2.2) $E_{2}\in L(m+n)$. There is one jump of $x_{i}$ before time $m$, and another jump of $x_{i}$ at time $t_{1}>\tau$, for some $i\ge 2$.

Moreover\footnote{This implies that $E(m)-q(m)>0$, in contrast to (5.2).}, $q^{(1)}_{i}(t)=q^{(2)}_{i}(t)$ for $t\in [0,m+n]$.

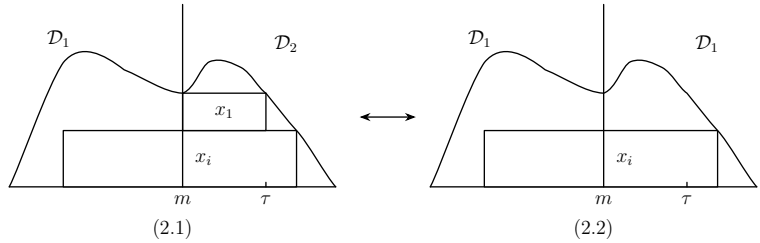
\WBPair
{%
    \WBBlockXiRight{\WBxA}{\WBxPost}{0}{\WBLow}

    \WBBlock{\WBxM}{\WBxTau}{\WBLow}{\WBC}{x_1}
}
{%
    \def\WBRightDomainLabel{\mathcal D_1}
    \WBBlockXiRight{\WBxA}{\WBxPost}{0}{\WBLow}
}
{\textup{(2.1)}}
{\textup{(2.2)}}

\vspace{10pt}

(3.1) $E_{1}\in L(m,n)$. There are two jumps of $x_{i}$ before time $m$, for some $i\ge 2$, and one jump of $x_{1}$ at time $\tau$.

(3.2) $E_{2}\in L(m+n)$. There are two jumps of $x_{i}$ before time $m$, for some $i\ge 2$.

Moreover, $q^{(1)}_{i}(t)=q^{(2)}_{i}(t)$ for $t\in [0,m+n]$.

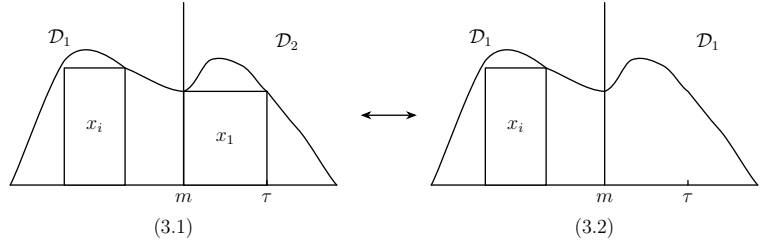
\WBPair
{%
    \WBBlock{\WBxA}{\WBxB}{0}{\WBPreTop}{x_i}

    \WBBlock{\WBxM}{\WBxTau}{0}{\WBC}{x_1}
}
{%
    \def\WBRightDomainLabel{\mathcal D_1}
    \WBBlock{\WBxA}{\WBxB}{0}{\WBPreTop}{x_i}
}
{\textup{(3.1)}}
{\textup{(3.2)}}

\vspace{10pt}

(4.1) $E_{1}\in L(m,n)$. There are two jumps of $x_{i}$ at times $t_1,t_2$ satisfying either $m<t_1<t_2<\tau$ or $\tau<t_1<t_2\le m+n$, for some $i\ge 3$, and one jump of $x_{1}$ at time $\tau$. Moreover,
$q^{(1)}(t)=q^{(2)}(t)$ on $[t_{1},t_{2})$.

(4.2) $E_{2}\in L(m+n)$. There are two jumps of $x_{i}$ at times $t_1,t_2$ satisfying either $m<t_1<t_2<\tau$ or $\tau<t_1<t_2\le m+n$, for some $i\ge 2$.

\WBPair
{%
    \WBBlock{\WBxM}{\WBxTau}
        {0}{\WBC}{x_1}

    \WBBlock{\WBxTOne}{\WBxTTwo}
        {\WBC}{\WBMidTop}{x_i}
}
{%
    \def\WBRightDomainLabel{\mathcal D_1}
    \WBBlock{\WBxTOne}{\WBxTTwo}
        {0}{\WBMidTop}{x_i}
}
{\textup{(4.1)}}
{\textup{(4.2)}}

\medskip

{The following configurations use Rule~\textbf{II}.}

(5.1) $E_{1}\in L(m,n)$, there is one jump of $x_{2}$ before time $m$.

(5.2) $E_{2}\in L(m,n)$, there is one jump of $x_{i}$ before time $m$, and another jump of $x_{i}$ after time $m$, for some $i\ge 3$.

{(5.3) $E_{3}\in L(m+n)$. For some $i\ge2$, there is one jump of $x_i$ before time $m$, and its closing jump occurs at time $\tau$. The block has height $C=E(m)$ between these two jumps.

More precisely, the block functions in all three configurations agree before the change of main variable:
\[
q_2^{(1)}(t)=q_i^{(2)}(t)=q_i^{(3)}(t),\qquad 0\le t\le m.
\]
They vanish before the common first-jump time and equal $C$ thereafter until $m$. In (5.2) and (5.3), the auxiliary block continues with height $C$ until it closes at $\tau$.}

{
\WBTriple
{%
    \WBBlock{\WBxA}{\WBxM}{0}{\WBC}{x_2}
}
{%
    \WBBlockXiRight{\WBxA}{\WBxTau}{0}{\WBC}
}
{%
    \def\WBRightDomainLabel{\mathcal D_1}
    \WBBlockXiRight{\WBxA}{\WBxTau}{0}{\WBC}
}
{\textup{(5.1)}}
{\textup{(5.2)}}
{\textup{(5.3)}}
}

\vspace{10pt}

(6.1) $E_{1}\in L(m,n)$. There is one jump of $x_{2}$ before time $m$ and one jump of $x_{1}$ after time $m$.

(6.2) $E_{2}\in L(m,n)$. There is one jump of $x_{i}$ before time $m$, another jump of $x_{i}$ after time $t_{1}>m$, for some $i\ge 3$, and one jump of $x_1$ at time $t_{2}>t_{1}$.

Moreover, $q^{(1)}_{2}(t)=q^{(2)}_{i}(t)$ for $t\in [0,m)$.

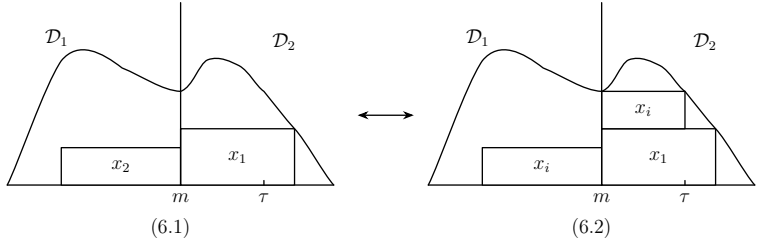
\WBPair
{%
    \WBBlock{\WBxA}{\WBxM}
        {0}{\WBComp}{x_2}

    \WBBlock{\WBxM}{\WBxPost}
        {0}{\WBLow}{x_1}
}
{%
    \WBBlock{\WBxA}{\WBxM}
        {0}{\WBComp}{x_i}

    \WBBlock{\WBxM}{\WBxPost}
        {0}{\WBLow}{x_1}

    \WBBlock{\WBxM}{\WBxTau}
        {\WBLow}{\WBC}{x_i}
}
{\textup{(6.1)}}
{\textup{(6.2)}}

\vspace{10pt}

(7.1) $E_{1}\in L(m,n)$. There are two jumps of $x_{i}$ at times $t_1,t_2$ satisfying either $1\le t_1<t_2\le m$ or $\tau<t_1<t_2\le m+n$, for some $i\ge 3$, and one jump of $x_{2}$ before time $m$.

(7.2) $E_{2}\in L(m,n)$. There are two jumps of $x_{i}$ at times $t_1,t_2$ satisfying either $1\le t_1<t_2\le m$ or $\tau<t_1<t_2\le m+n$, for some $i\ge 3$, one jump of $x_{j}$ before time $m$, and another jump of $x_{j}$ after time $m$, for some $j\ne 1,2,i$.

Moreover, $q^{(1)}_{i}(t)=q^{(2)}_{i}(t)$ for $t\in [0,m+n]$.

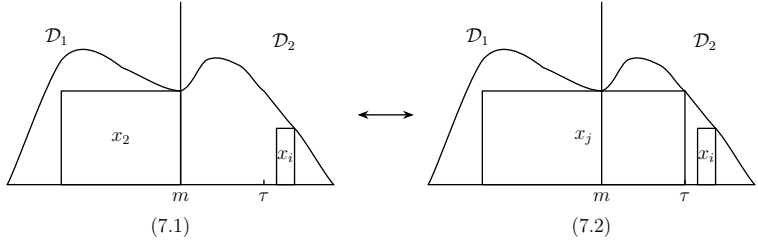
\WBPair
{%
    \WBBlock{\WBxA}{\WBxM}
        {0}{\WBC}{x_2}

    \WBBlock{\WBxAfterTau}{\WBxPost}
        {0}{\WBLow}{x_i}
}
{%
    \WBBlock{\WBxA}{\WBxTau}
        {0}{\WBC}{x_j}

    \WBBlock{\WBxAfterTau}{\WBxPost}
        {0}{\WBLow}{x_i}
}
{\textup{(7.1)}}
{\textup{(7.2)}}

\vspace{10pt}

(8.1) $E_{1}\in L(m,n)$. There is one jump of $x_{2}$ before time $m$, and two jumps of $x_{i}$ after time $m$ for some $i\ge 3$, where the first jump of $x_{i}$ occurs at time $\tau$.

(8.2) $E_{2}\in L(m,n)$. There is one jump of $x_{i}$ before time $m$, and another two jumps of $x_{i}$ after time $m$ for some $i\ge 3$, where the second jump occurs at time $\tau$.

Moreover, $q^{(1)}_{2}(t)=q^{(2)}_{i}(t)$ for $t\in [0,m)$, and $q^{(1)}_{i}(t)=q^{(2)}_{i}(t)$ for $t\in [\tau,m+n]$.

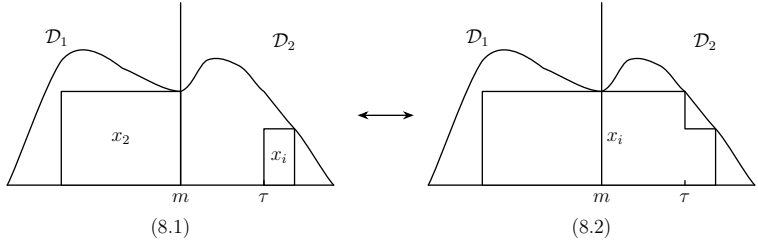
\WBPair
{%
    \WBBlock{\WBxA}{\WBxM}
        {0}{\WBC}{x_2}

    \WBBlock{\WBxTau}{\WBxPost}
        {0}{\WBLow}{x_i}
}
{%
    \WBPolyBlockText
    {(%
      \WBxA,0) --
      (\WBxA,\WBC) --
      (\WBxTau,\WBC) --
      (\WBxTau,\WBLow) --
      (\WBxPost,\WBLow) --
      (\WBxPost,0) -- cycle}
    {4.15}{1.04}{x_i}
}
{\textup{(8.1)}}
{\textup{(8.2)}}

\vspace{10pt}

(9.1) $E_{1}\in L(m,n)$. There are two jumps of $x_{2}$ before time $m$.

(9.2) $E_{2}\in L(m,n)$. There are two jumps of $x_{i}$ before time $m$, and another jump of $x_{i}$ after time $m$, for some $i\ge 3$.

Moreover, $q^{(1)}_{2}(t)=q_{i}^{(2)}(t)$ on $[0,m)$.

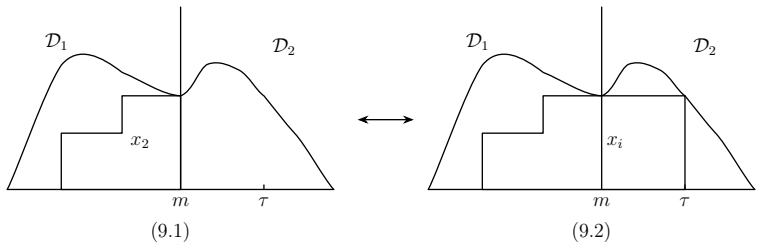
\WBPair
{%
    \WBPolyBlockText
    {(%
      \WBxA,0) --
      (\WBxA,\WBLow) --
      (\WBxB,\WBLow) --
      (\WBxB,\WBC) --
      (\WBxM,\WBC) --
      (\WBxM,0) -- cycle}
    {2.95}{1.04}{x_2}
}
{%
    \WBPolyBlockText
    {(%
      \WBxA,0) --
      (\WBxA,\WBLow) --
      (\WBxB,\WBLow) --
      (\WBxB,\WBC) --
      (\WBxTau,\WBC) --
      (\WBxTau,0) -- cycle}
    {4.15}{1.04}{x_i}
}
{\textup{(9.1)}}
{\textup{(9.2)}}

\vspace{10pt}

(10.1) $E_{1}\in L(m,n)$. There is one jump of $x_{2}$ before time $m$, one jump of $x_{i}$ before time $m$, and another jump of $x_{i}$ after time $m$, for some $i\ge 3$.

(10.2) $E_{2}\in L(m,n)$. There is one jump of $x_{j}$ before time $m$, another jump of $x_{j}$ after time $m$, for some $j\ge 3$, and one jump of $x_{i}$ before time $m$, another jump of $x_{i}$ after time $m$, for some $i\ne 1,2,j$.

Moreover, $q^{(1)}(t)=q^{(2)}(t)$ for $t\in [0,m)$.

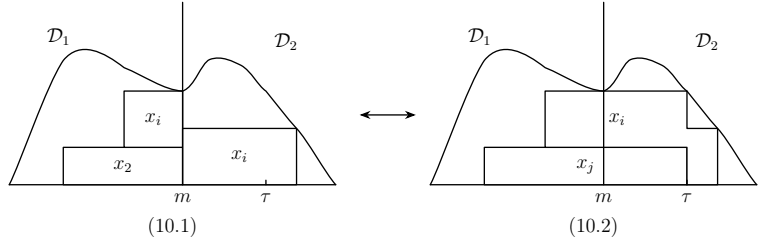
\WBPair
{%
    \WBBlock{\WBxA}{\WBxM}
        {0}{\WBComp}{x_2}

    \WBBlock{\WBxB}{\WBxM}
        {\WBComp}{\WBC}{x_i}

    \WBBlock{\WBxM}{\WBxPost}
        {0}{\WBLow}{x_i}
}
{%
    \WBBlock{\WBxA}{\WBxTau}
        {0}{\WBComp}{x_j}

    \WBPolyBlockText
    {(%
      \WBxB,\WBComp) --
      (\WBxB,\WBC) --
      (\WBxTau,\WBC) --
      (\WBxTau,\WBLow) --
      (\WBxPost,\WBLow) --
      (\WBxPost,0) --
      (\WBxTau,0) --
      (\WBxTau,\WBComp) -- cycle}
    {4.15}{1.46}{x_i}
}
{\textup{(10.1)}}
{\textup{(10.2)}}

\vspace{10pt}

(11.1) $E_{1}\in L(m,n)$. There is one jump of $x_{2}$ before time $m$, and two jumps of $x_{i}$ at times $t_{1},t_{2}\in [m+1,\tau-1]$, for some $i\ge 3$.

(11.2) $E_{2}\in L(m,n)$. There is one jump of $x_{j}$ before time $m$, two jumps of $x_{i}$ at times $t_{1},t_{2}\in [m+1,\tau-1]$, for some $i\ge 3$, and another jump of $x_{j}$ at time $\tau$, for some $j\ne 1,2,i$.

Moreover, $q^{(1)}(t)=q^{(2)}(t)$ on $[0,m)\cup[t_{1},t_{2})$.

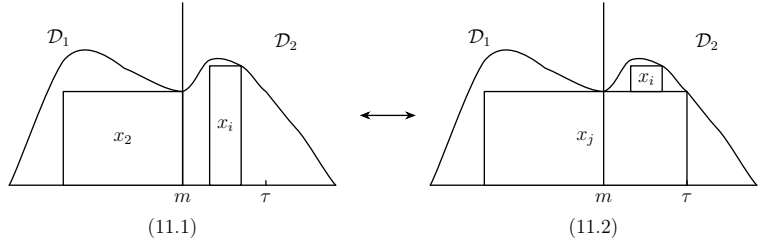
\WBPair
{%
    \WBBlock{\WBxA}{\WBxM}
        {0}{\WBC}{x_2}

    \WBBlock{\WBxTOne}{\WBxTTwo}
        {0}{\WBMidTop}{x_i}
}
{%
    \WBBlock{\WBxA}{\WBxTau}
        {0}{\WBC}{x_j}

    \WBBlock{\WBxTOne}{\WBxTTwo}
        {\WBC}{\WBMidTop}{x_i}
}
{\textup{(11.1)}}
{\textup{(11.2)}}

\vspace{10pt}

{(12.1) $E_{1}\in L(m,n)$. There is one jump of $x_{2}$ before time $m$, one jump of $x_{i}$ at time $t_{1}\in [m+1,\tau-1]$, and another jump of $x_{i}$ at time $t_{2}\in [\tau,m+n]$, for some $i\ge 3$.}

{(12.2) $E_{2}\in L(m,n)$. There is one jump of $x_{i}$ before time $m$, one jump of $x_{i}$ at time $t_{1}\in [m+1,\tau-1]$, and its closing jump at time $t_{2}$, where $t_1<t_2\le m+n$, for some $i\ge 3$.}

{(12.3) $E_{3}\in L(m,n)$. There is one jump of $x_{i}$ before time $m$, its closing jump at time $s_{1}$, and one jump of $x_{j}$ at time $t_{1}\in [m+1,\tau-1]$, followed by its closing jump at time $t_{2}$, where $m<t_{1}<s_{1}<t_{2}\le m+n$, for some $i\ge 3$, $j\ne 1,2,i$.}

{In (12.2) and (12.3), the last closing time may lie on either side of $\tau$. The restriction $t_2\ge\tau$ in (12.1) avoids overlap with (11.1).}

The index in the third configuration is $i$, not $j$:
\[
q^{(1)}_2(t)=q^{(2)}_i(t)=q^{(3)}_i(t),\qquad 0\le t<m.
\]
The $x_j$-block in (12.3) is created only after the splitting point, at time $t_1$.

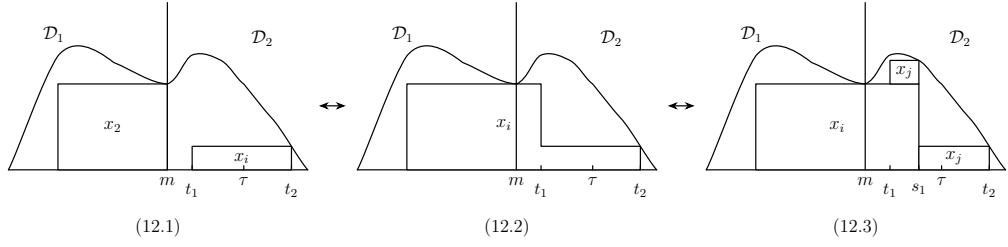
\WBTriple
{%
    \WBBlock{\WBxA}{\WBxM}
        {0}{\WBC}{x_2}

    \WBBlock{\WBxTOne}{\WBxTail}
        {0}{\WBSmall}{x_i}

    \WBMark{\WBxTOne}{t_1}
    \WBMark{\WBxTail}{t_2}
}
{%
    \WBPolyBlockText
    {(%
      \WBxA,0) --
      (\WBxA,\WBC) --
      (\WBxTOne,\WBC) --
      (\WBxTOne,\WBSmall) --
      (\WBxTail,\WBSmall) --
      (\WBxTail,0) -- cycle}
    {3.55}{1.10}{x_i}

    \WBMark{\WBxTOne}{t_1}
    \WBMark{\WBxTail}{t_2}
}
{%
    \WBBlock{\WBxA}{\WBxTTwo}
        {0}{\WBC}{x_i}

    \WBBlock{\WBxTOne}{\WBxTTwo}
        {\WBC}{\WBMidTop}{x_j}

    \WBBlock{\WBxTTwo}{\WBxTail}
        {0}{\WBSmall}{x_j}

    \WBMark{\WBxTOne}{t_1}
    \WBMark{\WBxTTwo}{s_1}
    \WBMark{\WBxTail}{t_2}
}
{\textup{(12.1)}}
{\textup{(12.2)}}
{\textup{(12.3)}}

\medskip

The following pair uses Rule~\textbf{III}.

(13.1) $E_{1}\in L(m,n)$. There are two jumps of $x_{1}$ after time $m$, and the first one  occurs at time $\tau$.

(13.2) $E_{2}\in L(m+n)$. There are two jumps of $x_{i}$ after time $m$ for some $i\ge 2$, and the first one occurs at time $\tau$.

Moreover, $q_{1}^{(1)}(t)=q^{(2)}_{i}(t)$ for $t\in [\tau,m+n]$.

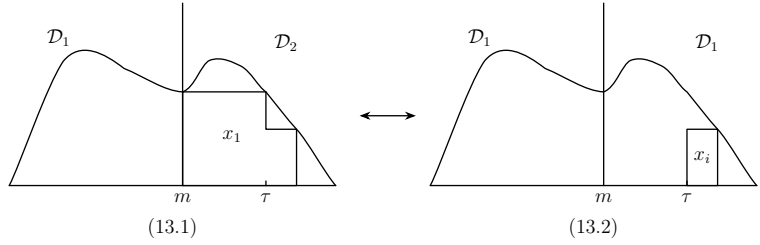
\WBPair
{%
    \WBPolyBlockText
    {(%
      \WBxM,0) --
      (\WBxM,\WBC) --
      (\WBxTau,\WBC) --
      (\WBxTau,\WBLow) --
      (\WBxPost,\WBLow) --
      (\WBxPost,0) -- cycle}
    {4.95}{1.10}{x_1}
}
{%
    \def\WBRightDomainLabel{\mathcal D_1}
    \WBBlock{\WBxTau}{\WBxPost}
        {0}{\WBLow}{x_i}
}
{\textup{(13.1)}}
{\textup{(13.2)}}

\vspace{10pt}

Finally, the following triple combines Rules~\textbf{I} and~\textbf{III}.

{(14.1) $E_{1}\in L(m,n)$. There is one jump of $x_{1}$ at time $s_{1}$, and one jump of $x_{i}$ at time $t_{1}\in [m+1,\tau-1]$, followed by its closing jump at time $s_{2}$, where $m<t_{1}<s_{1}<s_{2}\le m+n$, for some $i\ge 3$.}

{(14.2) $E_{2}\in L(m+n)$. There is one jump of $x_{i}$ at time $t_{1}\in [m+1,\tau-1]$, and its closing jump at time $t_{2}\in [\tau,m+n]$, for some $i\ge 2$.}

{(14.3) $E_{3}\in L(m,n)$. There is one jump of $x_{1}$ at time $t_{1}\in [m+1,\tau-1]$, and its closing jump at time $t_{2}$, where $t_1<t_2\le m+n$.}

{The last closing time is unrestricted relative to $\tau$ in (14.1) and (14.3). In (14.2), the restriction $t_2\ge\tau$ avoids overlap with (4.2).}

No block is present before the splitting point in any of the three configurations:
\[
q^{(1)}(t)=q^{(2)}(t)=q^{(3)}(t)=0,
\qquad 0\le t<m.
\]

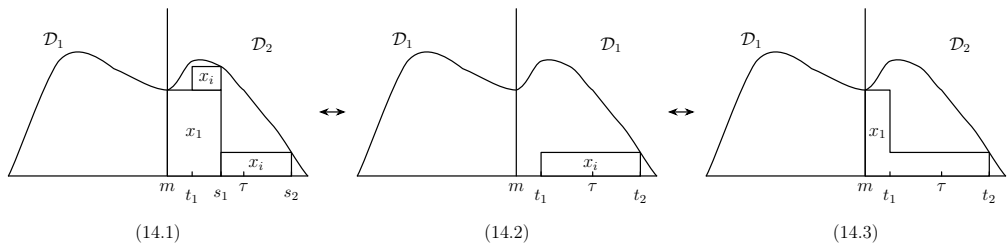
\WBTriple
{%
    \WBBlock{\WBxM}{\WBxTTwo}
        {0}{\WBC}{x_1}

    \WBBlock{\WBxTOne}{\WBxTTwo}
        {\WBC}{\WBMidTop}{x_i}

    \WBBlock{\WBxTTwo}{\WBxTail}
        {0}{\WBSmall}{x_i}

    \WBMark{\WBxTOne}{t_1}
    \WBMark{\WBxTTwo}{s_1}
    \WBMark{\WBxTail}{s_2}
}
{%
    \def\WBRightDomainLabel{\mathcal D_1}
    \WBBlock{\WBxTOne}{\WBxTail}
        {0}{\WBSmall}{x_i}

    \WBMark{\WBxTOne}{t_1}
    \WBMark{\WBxTail}{t_2}
}
{%
    \WBPolyBlockText
    {(%
      \WBxM,0) --
      (\WBxM,\WBC) --
      (\WBxTOne,\WBC) --
      (\WBxTOne,\WBSmall) --
      (\WBxTail,\WBSmall) --
      (\WBxTail,0) -- cycle}
    {4.15}{1.04}{x_1}

    \WBMark{\WBxTOne}{t_1}
    \WBMark{\WBxTail}{t_2}
}
{\textup{(14.1)}}
{\textup{(14.2)}}
{\textup{(14.3)}}

\vspace{10pt}
{These eleven pairs and three triples exhaust the cost-admissible configurations. First record whether a prescribed variable $x_1$ or $x_2$ carries a block across the splitting point, and then whether there are zero, one, or two freely labeled auxiliary blocks. For each active block, nonnegativity and the requirement that the terminal monomial be constant force its closing jump and height. Pair~1 realizes cases~(i)--(ii). Triple~5 contains cases~(iii)--(iv), together with the boundary case of (v) whose block crosses $m$ and closes at $\tau$. The remaining entries subdivide cases~(v)--(xii) according to their jump order, including the closing times on either side of $\tau$ specified in triples~12 and~14.}

\subsubsection{Cancellation of first-order weights}\label{sec:cancellation1}
In this subsection, we prove the following lemma.

\begin{lemma}
\begin{enumerate}[label=\textup{(\alph*)}]
\item The total weight of the configurations in pairs~2--4, 6--11, and~13 is of order $o(1)$.
\item {The total weight of the zero-Type~V configurations in triples~5 and~12 is of order $o(1)$.}
\item The total weight of the zero-Type~V configurations in triple~14 and pair~1 is of order $o(1)$.
\end{enumerate}
\end{lemma}
\begin{remark}
{The Group~A patterns are (1.1)--(1.2) and (5.1)--(5.2); with one Type~V action they form Group~C. Configuration (5.3), the entries in part~\textup{(a)}, and triples~12 and~14 use the full Group~B cost, so an additional Type~V action makes their normalized contribution $o(1)$. Thus (5.3) is needed for the zero-Type~V cancellation in part~\textup{(b)}, but does not change the asymptotic contribution evaluated in Section~\ref{sec:firstordercumulants}.}
\end{remark}
\begin{proof}
For part~\textup{(a)}, fix a common path and a matched pair of block patterns, with the correspondence specified in Section~\ref{sec:enumeration}. The potentially nonvanishing terms belong to Group B, so each side has normalized weight $O(1)$. The up and flat steps give exactly the same product of $c_{(d)}$-factors. The numbers of negative block jumps have opposite parity, so the signs are opposite.

The auxiliary-label and the canonical down steps give multiplicities of the form $N-c$, with $c=O(1)$ for fixed $m,n$. Rules~I--III balance their leading powers of $N$ together with the exterior factors. After these powers are extracted, the two index products are both $1+O(N^{-1})$. Keeping the common bounded, normalized coefficient product unchanged, the paired sum is therefore $O(N^{-1})$. Terms with a further costly action are already $o(1)$. There are only finitely many paths and unlabeled block patterns, which proves~\textup{(a)}.

For part~\textup{(b)}, fix $E$ and the common first-jump time before $m$ in groups~5 and~12. Put $C=E(m)$ and, at a down step $m<t<\tau$, write $r=r(t):=E(t-1)-C\ge1$. We keep the common $N$-dependent factors, including the coefficients and powers of $\theta$, until after cancellation; their normalized product is bounded and tends to $\prod_{B_i\in E}\kappa_{|B_i|}$.

First combine (5.2) and (5.3). Apart from the choices of auxiliary variable $x_i$ and the exterior factor, their zero-Type~V monomial weights agree, including their negative sign. There are $N-2$ auxiliary labels in (5.2) and $N-1$ in (5.3). Consequently their combined exterior and label factor is
\[
\frac{N(N-1)(N-2)+N(N-1)}{N^{m+n}}
=\frac{N(N-1)^2}{N^{m+n}}.
\]
This matches the exterior factor of (5.1) and its $N-1$ canonical choices at $\tau$. Before $m$ and after $\tau$, the remaining index factors agree. At a down step $m<t<\tau$, the canonical multiplicity is $N-1$ in (5.1), but $N-1-\mathbf1_{\{r(t)\le C\}}$ in (5.2) and (5.3). Hence, after extracting the common bounded normalized factors, the total weight of triple 5 has leading order term
\[
\sum_{\substack{t\in D(E)\\m<t<\tau}}
\mathbf1_{\{r(t)\le C\}}.
\]

We now compute the total weight of triple~12. Fix its first jump time $t_1=t\in D(E)$ after $m$, so that $m<t<\tau$. In (12.1), the new block height $r'$ ranges over $1,\ldots,C$: its closing time is at least $\tau$, with equality exactly when $r'=C$. Including the closing sign gives signed multiplicity $-C$. In (12.2), the block height changes from $C$ to a value in $\{r,\ldots,C-1\}$ when $r\le C$, or in $\{C+1,\ldots,r-1\}$ when $r>C$. The signed multiplicity, including the final closing jump, is respectively $C-r$ or $-(r-C-1)$. In (12.3), the new block height ranges over $1,\ldots,r-1$; both subsequent closing jumps are negative, so the signed multiplicity is $r-1$. For each allowed height, the closing times are forced by the common path and the specified closing order.

All three configurations have normalized weight $O(1)$, so their remaining index products can be replaced by their common leading value with an $O(N^{-1})$ error. Their signed multiplicities sum to
\[
\begin{cases}
-C+(C-r)+(r-1)=-1,&r\le C,\\
-C-(r-C-1)+(r-1)=0,&r>C.
\end{cases}
\]
Summing over the choices of $t_1$ gives the leading order term \[-\sum_{\substack{t\in D(E)\\m<t<\tau}}
\mathbf1_{\{r(t)\le C\}}.\]
This proves~\textup{(b)}.

For part~\textup{(c)}, fix $E$ and use the same $C$ and $r(t)$. In pair~1, the $N-1$ canonical choices at $\tau$ in (1.2) match the extra exterior factor $N-1$ in (1.1), whose closing jump at $\tau$ has negative sign. At each down step $m<t<\tau$, the respective canonical counts are $N-1$ and $N-1-\mathbf1_{\{r(t)\le C\}}$. The same  calculation therefore gives a leading order term
\[
\sum_{\substack{t\in D(E)\\m<t<\tau}}\mathbf1_{\{r(t)\le C\}}.
\]

At a fixed first jump time $t_1=t$ in triple~14, (14.2) has signed multiplicity $-C$, since there are $C$ choices of the height of block $x_i$. In (14.3), the change of the $x_1$-block from height $C$ gives signed multiplicity $C-r$ for $r\le C$ and $-(r-C-1)$ for $r>C$. Finally, (14.1) has $r-1$ possible heights for its new block, with positive sign because of its two closing jumps. These are exactly the three multiplicities just computed for triple~12, so their sum is $-\mathbf1_{\{r(t)\le C\}}$. They cancel the contribution of pair~1 after summing over $t$, which proves (c).
\end{proof}

\subsubsection{First-order free cumulants in covariance}\label{sec:firstordercumulants}

{It remains to evaluate the one-Type~V parts of the pairs (1.1)--(1.2) and (5.1)--(5.2), which belong to Group~C. The two-index terms have one costly auxiliary jump and one Type~V action, compensated by the exterior factor $N(N-1)\sim N^2$, while (1.2) has one Type~V action and exterior factor $N$. The block patterns give the following correspondence with Definition~\ref{def:firstorderconf}:}
\[
(1.2)\leftrightarrow\textup{(a)},\qquad
(1.1)\leftrightarrow\textup{(b)},\qquad
(5.1)\leftrightarrow\textup{(c)},\qquad
(5.2)\leftrightarrow\textup{(d)}.
\]
Match the two entries of each pair by their common path $E$. Their up and flat steps give $\prod_{B_i\in E}\kappa_{|B_i|}$, and the parity of the negative block jumps gives the sign in Definition~\ref{def:functional1}.

A Type~V action at $t\in P(E,\vec q)$ contributes $E(t-)-q(t-)$. Derivative times admissible for both configurations in a matched pair cancel with opposite signs. The two block functions differ by a block of height $C(E)$ on $(m,\tau)$, so the unmatched contribution is $C(E)$ at each time in
\[
\bigl(\{m+1,\ldots,\tau-1\}\cap D(E)\bigr)\cup\{\tau\}.
\]

{There are $B(E)$ such times. Pair~1 gives the configurations with no jump before the splitting point, while (5.1)--(5.2) permit the first jump at any time in $\{\bar{\tau}+1,\ldots,m\}\cap D(E)$. Together these are exactly the $A(E)$ choices. Hence the surviving contribution for a fixed path is}
\[
\frac1\theta A(E)B(E)C(E)
\prod_{B_i\in E}\kappa_{|B_i|}.
\]
Summing over $E\in L(m,n)$ gives
\begin{equation}
\frac{1}{\theta}\sum_{(E,\vec q)\in Conf_1(m,n)}I_1^{(1)}(E,\vec q)
=
\frac{1}{\theta}\sum_{E\in L(m,n)}
A(E)B(E)C(E)\prod_{B_i\in E}\kappa_{|B_i|}.
\end{equation}
This is the contribution of first-order free cumulants to $\alpha_{m,n}$ in Eqn.~\eqref{eq_fixedtempcov}.

\subsubsection{Reduction of the second-order contribution}
\label{sec:secondorderasmp}

We now consider the contribution involving the coefficients
$c_{\lambda}$ with $\ell(\lambda)=2$. By
Lemma~\ref{lem:typeII-typeVI-fixed}, after discarding an $o(1)$
contribution, every operator string contributing to this part contains
exactly one Type II operator and exactly one Type V operator.

At the level where the Type II operator is kept in the grouped form
$U_i^{(\lambda)}$, the unique Type V operator is forced by the
power-sum factor introduced by $U_i^{(\lambda)}$. After expanding
$U_i^{(\lambda)}$ again into its monomial summands
$U_i^{(\lambda,r,k)}$, some Type IV monomial actions may appear in the
difference parts of the subsequent Dunkl operators. We will see below
that no additional Type IV block can contribute at leading order. The
only possible Type IV action is the unique closing jump forced when the
Type II action introduces a fresh variable $x_j$, $j\ge 3$.

We first observe that the single-index term
\[
N\left(\frac{\D_1}{N}\right)^{m+n}
\exp(F_{N,\theta})\Big|_{\vec{x}=0}
\]
has vanishing second-order contribution. Indeed, a Type II action and a
Type V action together contribute a factor $O(N^{-2})$. If the monomial
selected from the power sum uses $x_1$, there is no summation over an
additional index, and the exterior factor $N$ leaves an $O(N^{-1})$
contribution. If it uses some $x_j$, $j\ne 1$, there are $O(N)$ choices
of $j$, but a nonzero terminal value requires an additional closing
Type IV action, which contributes another factor $O(N^{-1})$. The total
contribution is again $O(N^{-1})$. Thus, only the term
\[
N(N-1)
\left(\frac{\D_2}{N}\right)^n
\left(\frac{\D_1}{N}\right)^m
\exp(F_{N,\theta})\Big|_{\vec{x}=0}
\]
needs to be considered.

Let $s$ be the time of the unique Type II action, and write its selected
monomial in the form
\begin{equation}\label{eq:typeII-selected-monomial}
c_{\lambda}\lambda_r
x_1^{\lambda_r-1}x_k^{\lambda_{3-r}},
\qquad
r\in\{1,2\}.
\end{equation}
We claim that, at leading order,
\[
s\le m
\qquad\text{and}\qquad
k\ne 1.
\]

Indeed, if $s>m$, then before the Type II action there is no 
block carrying the positive degree $E(m)>0$. Hence some positive power
of $x_1$ is present at time $m$. Since the main variable after time
$m$ is $x_2$, removing this power of $x_1$ requires an additional
auxiliary Type IV action. By the fixed-temperature scaling discussion,
this produces an extra factor $O(N^{-1})$, and the resulting
contribution vanishes.

Similarly, if $k=1$, the Type II action creates no block.
Consequently, in the absence of an additional auxiliary Type IV action,
a monomial sequence ending at a constant must satisfy $E(m)=0$. Such
terms have already been removed by
Proposition~\ref{prop:covariance-weighted-expansion}. Adding the
auxiliary Type IV action needed when $E(m)>0$ again makes the
contribution $o(1)$.

It follows that the unique Type II action occurs at a time
$s\le m$ and selects either $k=2$ or a fresh index $k=j\ge 3$.

Fix the associated path $E\in L(m,n)$ and set
\[
C=C(E):=E(m)>0.
\]
At time $m$, the block created by the unique Type~II action is the only auxiliary block for configurations with $O(1)$ weight. A second block would require either a second Type~II action or an additional auxiliary Type~IV action, both already shown to be $o(1)$. The $x_1$-degree must also be zero at the splitting point: after time $m$ the main variable is $x_2$, so removing the remaining $x_1$-degree would require another Type~IV action. Thus the entire height $C(E)=E(m)$ is carried by the auxiliary variable selected in Eqn.~\eqref{eq:typeII-selected-monomial}.
Therefore, the power assigned to the auxiliary variable
in Eqn.~\eqref{eq:typeII-selected-monomial} must equal $C$:
\[
\lambda_{3-r}=C.
\]
Since the Type II action gives an increment
$\lambda_r+C-1$ to the path, we define
\begin{equation}\label{eq:k1k2-from-path}
k_1(E,s):=E(s)-E(s-)-C(E)+1,
\qquad
k_2(E):=C(E).
\end{equation}
Thus $\lambda_r=k_1(E,s)$ and $\lambda_{3-r}=k_2(E)$.

The Type II action is an up or flat step, so $s\in U(E)$. In addition,
after the block of height $C$ is created, the inequality
$q(t)\le E(t)$ forces $E(t)\ge C$ up to time $m$. Consequently,
\[
s\in\{\bar{\tau},\ldots,m\}\cap U(E).
\]
As before, we use the convention that
$\kappa_{k_1(E,s),k_2(E)}=0$ whenever $k_1(E,s)\le 0$.

\subsubsection{The two monomial realizations}
\label{sec:enumeration-secondorder}

Fix $E\in L(m,n)$ and
\[
s\in\{\bar{\tau},\ldots,m\}\cap U(E).
\]
There are precisely two leading-order monomial realizations of the
power-sum factor introduced by the Type II action.

\medskip

\noindent\textbf{Case I: the monomial indexed by \(k=2\).}

In this case the Type II action introduces the factor $x_2^C$. Before
time $m$, this is recorded by the block function
\[
q_2(t)=
\begin{cases}
C,&t\in[s,m],\\
0,&\text{otherwise},
\end{cases}
\qquad
q_i(t)\equiv 0,\quad i\ne 2.
\]
This is exactly the limiting second-order configuration in
case~\textup{(e)} of Definition~\ref{def:secondorderconf}.

\medskip

\noindent\textbf{Case II: the monomial indexed by \(k=j\ge 3\).}

Here the Type II action introduces $x_j^C$ for some 
$j\ge 3$. The corresponding block initially satisfies
\[
q_j(t)=C,\qquad t\in[s,m].
\]
Since $x_j$ never becomes a main variable, this block must eventually
be removed by a Type IV action.
Thus
\[
q^{\,j}_j(t)=
\begin{cases}
C,&t\in[s,\tau),\\
0,&\text{otherwise},
\end{cases}
\qquad
q^{\,j}_i(t)\equiv 0,\quad i\ne j.
\]
A minus sign enters the weight at the closing Type~IV action at time $\tau$: this action decreases $q_j$ from $C$ to zero and is the unique negative jump of the configuration.

There are $N-2$ possible choices of $j$. After summing over them, all
the resulting block configurations are identical up to a relabeling of
the fresh variable. We therefore use $j=3$ as a canonical
representative. This is exactly case~\textup{(f)} of
Definition~\ref{def:secondorderconf}.

Any further Type IV action creates an additional auxiliary-jump factor
$O(N^{-1})$ and hence gives an $o(1)$ contribution. Therefore the two
classes above exhaust all the monomial sequences contributing to the
second-order limit.

\subsubsection{Second-order asymptotics}
\label{sec:cancellation2}
\label{secondordercumulants}

Fix $E\in L(m,n)$ and
\[
s\in\{\bar{\tau},\ldots,m\}\cap U(E),
\]
and write
\[
k_1=k_1(E,s):=E(s)-E(s-)-C(E)+1,
\qquad
k_2=k_2(E):=C(E).
\]
Let $\vec q^{(2)}$ and $\vec q^{(3)}$ denote the two configurations
described above, corresponding respectively to the choices $k=2$ and
$k=j\ge 3$ in the monomial expansion of the unique Type II action.

Both classes have finite limits. Indeed, by
Eqn.~\eqref{eq_cltfixed},
\[
k_1(1+\mathbf{1}_{k_1=k_2})c_{(k_1,k_2)}
\longrightarrow
\theta^{1-k_1-k_2}\frac{\kappa_{k_1,k_2}}{k_2},
\]
while every other up or flat step contributes the corresponding
first-order free cumulant. In the $k=2$ case, the normalization of the unique
Type II and Type V actions is
\[
\frac{N(N-1)}{N^2}\longrightarrow 1.
\]
In the $k=j\ge 3$ case, the $N-2$ choices of the fresh variable
compensate for the extra factor $N^{-1}$ contributed by the forced
closing Type IV action at time $\tau$, and
\[
\frac{N(N-1)(N-2)}{N^3}\longrightarrow 1.
\]
The latter class has the opposite sign because its closing jump has
negative jump height. Consequently, after summing over all fresh
indices $j\ge 3$, the combined contribution of the two classes
converges to
\begin{align}
\frac{1}{\theta}
\Big(
I^{(2)}_1(E,\vec q^{(2)})
+
I^{(2)}_1(E,\vec q^{(3)})
\Big)
={}&
\frac{1}{\theta}
\frac{\kappa_{k_1,k_2}}{k_2}
\prod_{\substack{B_i\in E\\t(B_i)\ne s}}
\kappa_{|B_i|}
\nonumber\\
&\times
\left[
\sum_{t\in D(E)}
\bigl(E(t-)-q^{(2)}(t-)\bigr)
-
\sum_{t\in D(E)\setminus\{\tau\}}
\bigl(E(t-)-q^{(3)}(t-)\bigr)
\right].
\label{eq:combined-secondorder-limit}
\end{align}

It remains to evaluate the difference in brackets. The two block
functions agree before time $m$:
\[
q^{(2)}(t)=q^{(3)}(t),
\qquad t<m,
\]
and both vanish after time $\tau$. On the interval $(m,\tau)$, however,
\[
q^{(2)}(t)=0,
\qquad
q^{(3)}(t)=C(E).
\]

Moreover, the time $\tau$ is available for the Type V action in the
configuration $\vec q^{(2)}$, whereas it is occupied by the closing
Type IV action in $\vec q^{(3)}$, and
\[
E(\tau-)=C(E).
\]
Therefore,
\begin{align*}
&\sum_{t\in D(E)}
\bigl(E(t-)-q^{(2)}(t-)\bigr)
-
\sum_{t\in D(E)\setminus\{\tau\}}
\bigl(E(t-)-q^{(3)}(t-)\bigr)
\\
&\qquad
=
C(E)
\left(
\bigl|\{m+1,\ldots,\tau-1\}\cap D(E)\bigr|+1
\right)
=
C(E)B(E).
\end{align*}
Since $k_2=C(E)$, the factor $C(E)$ cancels the denominator $k_2$ in
Eqn.~\eqref{eq:combined-secondorder-limit}. Thus, for every fixed pair
$(E,s)$, the two monomial classes together converge to
\[
\frac{1}{\theta}
B(E)\,
\kappa_{k_1(E,s),k_2(E)}
\prod_{\substack{B_i\in E\\t(B_i)\ne s}}
\kappa_{|B_i|}.
\]

Finally, summing over all admissible paths and all possible positions
of the Type II action gives
\begin{align}
\frac{1}{\theta}
\sum_{(E,\vec q)\in Conf_2(m,n)}
I^{(2)}_1(E,\vec q)
=
\frac{1}{\theta}
\sum_{E\in L(m,n)}
\sum_{s\in\{\bar{\tau},\ldots,m\}\cap U(E)}
B(E)\,
\kappa_{k_1(E,s),k_2(E)}
\prod_{\substack{B_i\in E\\t(B_i)\ne s}}
\kappa_{|B_i|}.
\label{eq:secondorder-final-fixed}
\end{align}
This is precisely the second-order contribution in
Eqn.~\eqref{eq_fixedtempcov}.

\subsection{High-temperature covariance in Theorem
\ref{thm:hightempclt}}

Recall that
\[
p_k^N=\frac{1}{N}p_k(\vec a)
=\frac{1}{N}\sum_{i=1}^N a_i^k.
\]
Therefore, by Eqn.~\eqref{eq_covthroughdunkl2}, the covariance appearing
in the high-temperature CLT is
\begin{align}
N\Cov\bigl(p_m^N,p_n^N\bigr)
={}&
(N-1)\D_2^n\D_1^m
\exp(F_{N,\theta})\Big|_{\vec{x}=0}
+
\D_1^{m+n}
\exp(F_{N,\theta})\Big|_{\vec{x}=0}
\nonumber\\
&-
N\,
\D_1^m\exp(F_{N,\theta})\Big|_{\vec{x}=0}\,
\D_1^n\exp(F_{N,\theta})\Big|_{\vec{x}=0}.
\label{eq:hightemp-dunkl-covariance}
\end{align}
Thus, the exterior factors in the three terms are respectively of
orders \(N\), \(1\), and \(N\).

We assume throughout this subsection that
\[
d\,c_{(d)}\longrightarrow\kappa_d^\gamma,
\qquad
Nk_1k_2(1+\mathbf{1}_{k_1=k_2})c_{(k_1,k_2)}
\longrightarrow\kappa_{k_1,k_2}^\gamma,
\qquad
N^{\ell(\lambda)/2}c_\lambda\longrightarrow0
\quad\text{for }\ell(\lambda)\ge3.
\]

The terms with $E(m)=0$ cancel as in Proposition~\ref{prop:covariance-weighted-expansion}. We first record the high-temperature orders of the remaining actions. Types~I and~V are $O(1)$, and a Type~III action contributes $\theta(N-O(1))\to\gamma$ after summing over its available indices. Each Type~IV action carries $\pm\theta\sim O(N^{-1})$: only a first jump to a freely chosen auxiliary index has $N-O(1)$ choices compensating for this factor. So a costly auxiliary jump still costs $O(N^{-1})$.

To control Types~II and~VI, retain their power-sum factors as in the proof of Lemma~\ref{lem:typeII-typeVI-fixed}. The difference part passes every such symmetric factor through unchanged, so each factor must be hit by a Type~V action before evaluation at the origin. These derivatives introduce no free-index multiplicity, and the remaining grouped actions are $O(1)$. Thus $r$ Type~II actions contribute $O(N^{-r})$ before the exterior factor, whereas a Type~VI action indexed by $\lambda$ contributes $o(N^{-\ell(\lambda)/2})$. Since the exterior factors are at most $O(N)$, only zero or one Type~II action and no Type~VI action can survive. Unlike at fixed temperature, further Type~V actions do not lower the order.

{Without Type~II, the two-index term allows at most one costly auxiliary jump, and the one-index term allows none. Hence Group~B, including (5.3), is $o(1)$; the remaining block patterns are the pairs (1.1)--(1.2) and (5.1)--(5.2) in Group~A, namely $Conf_1(m,n)$. The two members of each pair have opposite signs and equal zero-Type~V leading weights, which cancel. With Type~V actions, we retain the full signed weight difference of each pair, allowing any nonempty set of derivative times; their derivative factors need not agree.}

Now suppose there is exactly one Type~II action. Its single-index contribution is $O(N^{-1})$, so vanishes in the limit. In the two-index term, a leading contribution cannot afford an additional costly jump beyond those needed for the Type~II factor. Consequently the Type~II time satisfies $s\le m$: otherwise the positive degree at time $m$ would require an extra costly jump after the splitting point. For either summand of this Type~II action, expand
\[
k_1c_{(k_1,k_2)}x_1^{k_1-1}p_{k_2}
=k_1c_{(k_1,k_2)}\sum_{i=1}^N x_1^{k_1-1}x_i^{k_2}.
\]
If $i=1$, the remaining $x_1$-degree must be removed after time $m$ by a Type~IV action, so the contribution is $O(N^{-1})$. If $i=2$, no closing jump is needed and the contribution is $O(1)$. If $i\ge3$, one closing Type~IV action is necessary, but its $N^{-1}$ factor is compensated by the $N-2$ choices of $i$, again giving $O(1)$. These orders include both the exterior $O(N)$ factor and the Type~II coefficient $O(N^{-1})$; any further costly jump makes the contribution $o(1)$.

In the last two cases the $x_1$-degree at time $m$ must therefore be zero. Thus
\[
k_2=C=E(m),\qquad k_1=E(s)-E(s-)-C+1.
\]
 The $i=2$ block occupies $[s,m]$, whereas the $i\ge3$ block occupies $[s,\tau)$ and closes by a negative Type~IV action at $\tau$, which also introduces a minus sign to the weight. These are exactly cases~\textup{(e)} and~\textup{(f)} of Definition~\ref{def:secondorderconf}, with the same admissible Type~II times as in Section~\ref{sec:enumeration-secondorder}. Pair them by their common path and Type~II time, retaining both Type~II summands with their multiplicities. With no Type~V action, the closing factor $\theta(N-2)$ on the $i\ge 3$ side matches the canonical-down factor $\theta(N-O(1))$ on the $i=2$ side at $\tau$. All other leading factors agree, with opposite signs, so the paired contribution is $o(1)$. For one or more Type~V actions, we instead sum the actual signed derivative weights of the two configurations. This yields the $Conf_2(m,n)$ contribution evaluated below.

Fix one such configuration \((E,\vec q)\), and suppose that Type V
actions occur at
\[
t_1<\cdots<t_r,
\qquad
t_1,\ldots,t_r\in P(E,\vec q).
\]
The \(r\) Type V actions contribute
\[
\prod_{j=1}^r
\bigl(E(t_j-)-q(t_j-)\bigr).
\]
The remaining $|D(E)|-r$ down steps use the Type III actions. After summing over free auxiliary indices and incorporating the exterior factor $\theta$, their combined limiting factor, apart from the jump signs, is
\[
\gamma^{|D(E)|-r}.
\]
The up and flat steps contribute the first-order free
\(\gamma\)-cumulants, and, in the presence of a Type II action, the
distinguished step contributes the corresponding second-order free
\(\gamma\)-cumulant. Therefore, after summing over all $r$-element subsets
\[
\{t_1,\ldots,t_r\}\subset P(E,\vec q),
\]
(the subset sum is exactly the one built into Definitions~\ref{def:functional1} and~\ref{def:functional2}), the limiting contribution is
\[
\gamma^{|D(E)|-r}I_r^{(1)}(E,\vec q;\gamma)
\]
for \((E,\vec q)\in Conf_1(m,n)\), and
\[
\gamma^{|D(E)|-r}I_r^{(2)}(E,\vec q;\gamma)
\]
for \((E,\vec q)\in Conf_2(m,n)\).

The $r=0$ terms cancel at leading order after the configurations are paired, as shown above. Hence only nonempty subsets of $P(E,\vec q)$ remain.

Consequently,
\begin{equation}\label{eq:hightemp-covariance-asymptotics}
\begin{split}
\alpha_{m,n}
={}&
\sum_{(E,\vec q)\in Conf_1(m,n)}
\sum_{r=1}^{|P(E,\vec q)|}
\gamma^{|D(E)|-r}
I_r^{(1)}(E,\vec q;\gamma)
\\
&+
\sum_{(E,\vec q)\in Conf_2(m,n)}
\sum_{r=1}^{|P(E,\vec q)|}
\gamma^{|D(E)|-r}
I_r^{(2)}(E,\vec q;\gamma).
\end{split}
\end{equation}

This is exactly the covariance formula in
Eqn.~\eqref{eq_hightempcov}. The configuration classes are the same as
in the fixed-temperature regime. The difference is that, at fixed
temperature, every Type V action contributes an additional factor
\(N^{-1}\), so only one Type V action survives, whereas in the
high-temperature regime Type V actions are of order \(O(1)\), and all
nonempty subsets of \(P(E,\vec q)\) contribute. 

\section{The ``only if'' direction of Theorems~\ref{thm:intro-fixed-iff} and~\ref{thm:intro-high-iff}}\label{sec:only-if}
We start with the following combinatorial result on the coefficients of $\ln B_a(\vec{x};\theta)$. Recall that $\ln B_a(\vec{x};\theta)$ is analytic in $a$ and $\vec{x}$. Expand it in the power sums $\frac{p_{\pmu_1}(a/N)}{N^{\ell(\pmu_{1})}}p_{\pmu_2}(\vec{x})$ ($|\pmu_1|=|\pmu_2|=d$ for some $d\in \Z_{\ge 1}$), and let $C_{\pmu_1,\pmu_2}$ denote the corresponding coefficient.
\begin{lemma}\label{lem:onlyiffixedtemp}
    In the fixed-temperature regime, for $\pmu_1,\pmu_2\vdash d$, the $N$-degree of $C_{\pmu_1,\pmu_2}$ is at most $2-\ell(\pmu_2)$, and $N^{\ell(\pmu_2)-2}C_{\pmu_1,\pmu_2}$ converges as $N\to\infty$.

\end{lemma}
\begin{proof}
    This follows directly from Theorem \ref{thm:keyexpansion}.
\end{proof}

We shall use the following product formula for cumulants, usually called the
Leonov--Shiryaev formula; see \cite{LeonovShiryaev1959,Speed1983}.

\begin{lemma}\label{lem:leonov}
Let \(\mathcal I\) be a finite set of labels, and let
\((\xi_i)_{i\in\mathcal I}\) be commuting observables of a normalized moment functional, so that all moments appearing below are defined. Let
\[
\tau=\{I_1,\ldots,I_r\}
\]
be a partition of \(\mathcal I\), and define
\[
Y_n=\prod_{i\in I_n}\xi_i,\qquad n=1,\ldots,r.
\]
Then
\[
\mathfrak c_r(Y_1,\ldots,Y_r)
=
\sum_{\substack{\pi\in\Pi(\mathcal I)\\ \pi\vee\tau=\hat 1}}
\prod_{B\in\pi}
\mathfrak c_{|B|}\bigl(\xi_i:i\in B\bigr).
\]
Here \(\Pi(\mathcal I)\) denotes the set of all set partitions of
\(\mathcal I\), \(\hat 1\) is the one-block partition of \(\mathcal I\), and
\(\pi\vee\tau\) is the join of the two partitions in the partition lattice,
i.e. the finest partition which is coarser than both \(\pi\) and \(\tau\).

\end{lemma}

\begin{remark}\label{rem:leonov-connectedness}
The condition
\[
\pi\vee\tau=\hat 1
\]
means that $\pi$ connects the product blocks $I_1,\ldots,I_r$. In particular, for $r\ge2$, every product block contains a label that belongs to a non-singleton block of $\pi$, and therefore
\[
\sum_{B\in\pi:\,|B|\ge2}|B|\ge r.
\]
Equality can hold only when $\pi$ has a unique non-singleton block, this block meets every $I_j$ exactly once, and all remaining blocks of $\pi$ are singletons. This observation will be used below.
\end{remark}

\begin{theorem}\label{thm:fixedtemponlyif}
    In the fixed-temperature regime, if the CLT of Definition~\ref{def:fixedtempclt} holds, then
    for all $d\in \Z_{\ge 1}$, $k_{1},k_{2}\in \Z_{\ge 1}$,
    \begin{equation}\label{eq_fixedtempcoefficients}
        \begin{split}
            \frac{d}{N}c_{(d)}&\rightarrow \theta^{1-d}\kappa_{d};\\
            k_{1}k_{2}(1+\mathbf{1}_{k_1=k_2})c_{(k_{1},k_{2})}&\rightarrow \theta^{1-k_{1}-k_{2}}\kappa_{k_{1},k_{2}};\\
            c_{\lambda}&\rightarrow 0\quad\quad \text{for}\ \ell(\lambda)\ge 3,
        \end{split}
    \end{equation}
    for certain constants $\kappa_{d},\kappa_{k_{1},k_{2}}$ when $N\rightarrow\infty$.
    \end{theorem}
\begin{proof}
    The first line of Eqn.~\eqref{eq_fixedtempcoefficients} follows directly from \cite[Theorem 3.4]{CX}, since CLT implies LLN.

   For the third line of Eqn.~\eqref{eq_fixedtempcoefficients}, recall that \begin{equation}
   \begin{split}&\ln G_{\theta, N}(\vec{x})=\ln \E[\exp(\ln (B_a(\vec{x};\theta)))]
   =\ln \E\left[\exp\left(\sum_{\pmu_{2}}\left(\sum_{\pmu_{1}}C_{\pmu_{1},\pmu_{2}}\frac{p_{\pmu_{1}}(a/N)}{N^{\ell(\pmu_1)}}\right)p_{\pmu_{2}}(\vec{x})\right)\right].
   \end{split}\end{equation} Note that for $d\in \Z_{\ge 1}$, there are finitely many summands in $\pmu_{1},\pmu_{2}$ of size $d$ above.

   Then, by definition of the joint cumulants, each $c_{\lambda}$ is a finite sum of terms of the form \begin{align*}&\mathfrak c_{l}(C_{\pmu_{1}(1),\pmu_{2}(1)}\frac{p_{\pmu_{1}(1)}(a/N)}{N^{\ell(\pmu_1(1))}},...,C_{\pmu_{1}(l),\pmu_{2}(l)}\frac{p_{\pmu_{1}(l)}(a/N)}{N^{\ell(\pmu_{1}(l))}})= N^{-l}\prod_{j=1}^{l}C_{\pmu_{1}(j),\pmu_{2}(j)}\cdot N^{l}\mathfrak c_{l}(\frac{p_{\pmu_{1}(1)}(a/N)}{N^{\ell(\pmu_{1}(1))}},...,\frac{p_{\pmu_{1}(l)}(a/N)}{N^{\ell(\pmu_{1}(l))}}),\end{align*}
   where $\lambda=\pmu_{2}(1)\sqcup\cdots \sqcup \pmu_{2}(l)$.

The remaining task is to estimate the order of the joint cumulants of products using Lemma \ref{lem:leonov}.  More precisely, if $l\ge 3$,
   by Lemma \ref{lem:onlyiffixedtemp}, $N^{-l}\prod_{j=1}^{l}C_{\pmu_{1}(j),\pmu_{2}(j)}\lesssim O(1)$. Moreover, by Lemma \ref{lem:leonov} and the CLT assumption, $ N^{l}\mathfrak c_{l}(\frac{p_{\pmu_{1}(1)}(a/N)}{N^{\ell(\pmu_{1}(1))}},...,\frac{p_{\pmu_{1}(l)}(a/N)}{N^{\ell(\pmu_{1}(l))}})\sim o(1)$.

   If $l=2$, by Lemma \ref{lem:onlyiffixedtemp} $N^{-l}\prod_{j=1}^{l}C_{\pmu_{1}(j),\pmu_{2}(j)}\sim o(1)$ since at least one of the $\pmu_{2}(j)$ cannot be single-row, and by Lemma \ref{lem:leonov} and the CLT assumption, $ N^{2}\mathfrak c_{2}(\frac{p_{\pmu_{1}(1)}(a/N)}{N^{\ell(\pmu_{1}(1))}},\frac{p_{\pmu_{1}(2)}(a/N)}{N^{\ell(\pmu_{1}(2))}})\lesssim O(1)$.

   If $l=1$, $\ell(\pmu_{2})\ge 3$. By Lemma \ref{lem:onlyiffixedtemp} $C_{\pmu_{1},\pmu_{2}}\sim o(1)$, and by the CLT assumption, $\mathfrak c_{1}(p_{\pmu_{1}}(a/N)/N^{\ell(\pmu_1)})\lesssim O(1)$. This completes the proof of the third line of Eqn.~\eqref{eq_fixedtempcoefficients}.

    The second line of Eqn.~\eqref{eq_fixedtempcoefficients} follows from a similar argument. Note that $l$ can only be 1 or 2. If $l=1$, $\ell(\pmu_{2})=2$ and by Lemma \ref{lem:onlyiffixedtemp} $C_{\pmu_1,\pmu_2}$ converges. By the CLT assumption,  $\mathfrak c_{1}(p_{\pmu_{1}}(a/N)/N^{\ell(\pmu_1)})$ converges as $N\to\infty$. If $l=2$, the terms $N^{-2}\prod_{j=1}^{2}C_{\pmu_1(j),(d_j)}$ converge by Lemma \ref{lem:onlyiffixedtemp}, and the terms $N^{2}\mathfrak c_{2}(\frac{p_{\pmu_{1}(1)}(a/N)}{N^{\ell(\pmu_{1}(1))}},\frac{p_{\pmu_{1}(2)}(a/N)}{N^{\ell(\pmu_{1}(2))}})$ converge by Lemma \ref{lem:leonov} and the CLT assumptions.

\end{proof}

Next, we move on to the high-temperature regime. In this regime, let $C_{\pmu_1,\pmu_2}$ denote the coefficient of $\ln B_a(\vec{x};\theta)$ of $\frac{p_{\pmu_{1}}(a)}{N^{\ell(\pmu_1)}}p_{\pmu_{2}}(\vec{x})$.
We abbreviate $\nu_\rho(\MM,c)$ by $\nu(\MM)$ in the proof below, leaving the distinguished root implicit.

\begin{lemma}\label{lem:onlyifhightemp}
In the high-temperature regime, for \(\pmu_1,\pmu_2\vdash d\), the
\(N\)-degree of \(C_{\pmu_1,\pmu_2}\) is at most
\[
1-\ell(\pmu_2),
\]
and $N^{\ell(\pmu_2)-1}C_{\pmu_1,\pmu_2}$ converges as $N\to\infty$.
\end{lemma}

\begin{proof}
Set
\[
u=N\theta,
\qquad
\alpha=\theta^{-1}=\frac Nu.
\]
First, consider $u>d$ and $N>d$, so that the hypotheses of Theorem~\ref{thm:keyexpansion} hold with $n=d$ and $\theta=u/N$. We extract the coefficient of
$
\frac{p_{\pmu_1}(a)}{N^{\ell(\pmu_1)}}p_{\pmu_2}(\vec x)$ from Theorem \ref{thm:keyexpansion}, which reads as
\begin{equation}\label{eq_hightempcoefficientexpansion}
\begin{split}
\sum_{\MM\in C^{(\infty)}_{d,\bullet}\colon \MM\text{ normal},\ \pmu_1(\MM)=\pmu_1,\ \pmu_2(\MM)=\pmu_2}&\frac1d
\left(\frac Nu\right)^{d-1}
N^{2-2g(\MM)}\\
&\quad\quad\quad\left(\frac Nu-1\right)^{\nu(\MM)}
N^{-\ell(\pmu_2)}N^{-d}
f_{\eta(\MM)}\bigl(\{(-1)^k\}_{k\ge1}\bigr).
\end{split}
\end{equation}
We claim that this is a formal power series in the variables $u^{-1}$ and $N^{-1}$. Indeed, for each constellation $\MM$, the degree of $u$ is clearly nonpositive, and the degree of $N$ is at most $1-2g(\MM)+\nu(\MM)-\ell(\pmu_2)$. One can verify from \cite[Definitions 3.6\,--\,3.8]{CD} that $0\le\nu(\MM)\le 2g(\MM)$, so this is also at most 0.

Moreover, for real $u>0$ sufficiently large and $N\ge u$ sufficiently large, the sum over individual maps converges absolutely. To see this, group the constellations $\MM$ in the sum by their statistics $\nu(\MM)\in \Z_{\ge 0}$ and $2g(\MM)-\nu(\MM)\in \Z_{\ge 0}$. Then we have the following two facts:

(1).  \[
    \bigg| f_{\eta}\Big(\big\{(-1)^{-k}\big\}_{k\ge 1}\Big) \bigg|  \le (8e)^{|\eta|}\le (8e)^{2g+2d},\quad \text{and}
    \]

(2). The total number of constellations in the above sum with genus $g$ is at most $d^{4d(d+g)}$.

We refer the reader to Proposition \ref{prop:euler} and \cite[Remark 5.7]{CX} for the proofs of these two facts. Based on (1) and (2), we can upper bound the absolute value of the inner sum of $\MM$ for fixed genus and get a series \[\frac{1}{d}u^{1-d}N^{1-\ell(\pmu_2)}d^{4d^{2}}(8e)^{2d}\sum_{\nu(\MM)=0}^{\infty}d^{2d\nu(\MM)}(u/8e)^{-\nu(\MM)}\sum_{2g(\MM)-\nu(\MM)=0}^{\infty}d^{2d(2g-\nu(\MM))}(N/8e)^{-(2g-\nu(\MM))},\]
which converges when $N$ is large enough and $u>8ed^{2d}$.

It remains to extend the coefficient bound to arbitrary $u$ with $\Re u>0$. For fixed $d$, the coefficient $C_{\pmu_1,\pmu_2}$ depends only on the finite-degree part of the Jack expansion \eqref{eq_BesselJack}. After substituting $\theta=u/N$, this is a rational function of $z=N^{-1}$ and $u$. Its Laurent expansion at $z=0$ has a finite principal part:
\[
C_{\pmu_1,\pmu_2}=\sum_{r\le R}N^r h_r(u)
\]
for some finite integer $R$. The finite Jack formula shows that each $h_r$ is holomorphic on $\{\Re u>0\}$, and the expansion converges locally uniformly there for sufficiently large $N$. For each real $u>8ed^{2d}$, the preceding map estimate is $O(N^{1-\ell(\pmu_2)})$, so uniqueness of the Laurent expansion forces $h_r(u)=0$ whenever $r>1-\ell(\pmu_2)$. The identity theorem for holomorphic functions then gives
\[
h_r(u)=0
\qquad
\text{for all } u>0,\ r>1-\ell(\pmu_2).
\]
In the high-temperature regime, $N\to\infty$ and $u=u_{N}\to\gamma>0$.
Hence
\[
\deg_N C_{\pmu_1,\pmu_2}
\le
1-\ell(\pmu_2).
\]
The convergence as $N\to\infty$ follows immediately from the analyticity of $h_r(u)$ and the fact that $u\to\gamma$.
\end{proof}

\begin{theorem}\label{thm:hightemponlyif}
    In the high-temperature regime, if the CLT of Definition~\ref{Definition_CLT_sat} holds, then
    for all $d\in \Z_{\ge 1}$, $k_{1},k_{2}\in \Z_{\ge 1}$,
    \begin{equation}\label{eq_hightempcumulant}
        \begin{split}
            d\cdot c_{(d)}&\rightarrow \kappa^{\gamma}_{d};\\
           Nk_{1}k_{2}(1+\mathbf{1}_{k_1=k_2})c_{(k_{1},k_{2})}&\rightarrow \kappa^{\gamma}_{k_{1},k_{2}};\\
          N^{\ell(\lambda)/2}  c_{\lambda}&\rightarrow 0\quad\quad \text{for}\ \ell(\lambda)\ge 3.
        \end{split}
    \end{equation}
    for certain constants $\kappa_{d}^{\gamma}, \kappa_{k_{1},k_{2}}^{\gamma}$ when $N\rightarrow\infty$.
    \end{theorem}
    \begin{proof}
        The proof is similar to the one in the fixed-temperature regime, but one must keep track of the different scaling. The first line of Eqn.~\eqref{eq_hightempcumulant} again follows directly. For the third line, by exactly the same argument as in the proof of Theorem \ref{thm:fixedtemponlyif}, it remains to estimate the asymptotic order of the terms of the form \[N^{-l}\prod_{j=1}^{l}C_{\pmu_{1}(j),\pmu_{2}(j)}\cdot N^{l}\mathfrak c_{l}(\frac{p_{\pmu_{1}(1)}(a)}{N^{\ell(\pmu_{1}(1))}},...,\frac{p_{\pmu_{1}(l)}(a)}{N^{\ell(\pmu_{1}(l))}}).\]

         If $l\ge 3$,
   by Lemma \ref{lem:onlyifhightemp},  $N^{-l}\prod_{j=1}^{l}C_{\pmu_{1}(j),\pmu_{2}(j)}\lesssim $  $N^{-l+\sum_{j=1}^{l}(1-\ell(\pmu_2(j)))}$. Moreover, by Lemma \ref{lem:leonov} and the 
CLT assumption,  $N^{l}\mathfrak c_{l}(\frac{p_{\pmu_{1}(1)}(a)}{N^{\ell(\pmu_{1}(1))}},...,\frac{p_{\pmu_{1}(l)}(a)}{N^{\ell(\pmu_{1}(l))}})=o(N^{l/2})$. Then the product of these two parts is $\lesssim N^{-\sum_{j=1}^{l}\ell(\pmu_{2}(j))+l/2}\le N^{-\frac{\ell(\lambda)}{2}}$.

   If $l=2$, by the same argument as above,  the whole term is at most $N^{-\sum_{j=1}^{l}\ell(\pmu_{2}(j))+l/2}$. But since at least one of the $\pmu_{2}(j)$ cannot be single-row, $-\sum_{j=1}^{l}\ell(\pmu_{2}(j))+l/2<-\frac{\sum_{j=1}^{l}\ell(\pmu_{2}(j))}{2}=-\frac{\ell(\lambda)}{2}$.

   If $l=1$, then $\lambda=\pmu_2$ and $\ell(\pmu_2)\ge 3$. By Lemma \ref{lem:onlyifhightemp} $ C_{\pmu_{1},\pmu_{2}}\sim O(N^{1-\ell(\pmu_2)})$, and by the CLT assumption, $\mathfrak c_{1}(p_{\pmu_{1}}(a)/N^{\ell(\pmu_1)})\lesssim O(1)$. So the product of these two terms and $N^{\ell(\pmu_2)/2}$ is at most $O(N^{1-\ell(\pmu_2)/2})=o(1)$. This completes the proof of the third line of Eqn.~\eqref{eq_hightempcumulant}.

   The second line of Eqn.~\eqref{eq_hightempcumulant} also follows similarly from Lemma \ref{lem:onlyifhightemp}, the CLT assumption and the above argument.
    \end{proof}
\section{Vanishing of the higher joint cumulants}\label{sec:vanishing-higher-cumulants}
In this section we complete the proofs of Theorems \ref{thm:fixedtempclt} and \ref{thm:hightempclt} by proving that, under the corresponding fixed- and high-temperature scalings, every classical joint cumulant of order at least three formed from the power sums converges to zero as $N\to\infty$.

\begin{theorem}\label{thm:vanishinghighercumulants}
Assume the fixed-temperature coefficient conditions in Eqn.~\eqref{eq_cltfixed}. Then, for every $s\ge 3$ and $k_1,\ldots,k_s\ge 1$,
\[
 N^s\mathfrak c_s\bigl(\widetilde p_{k_1}^N,\ldots,
 \widetilde p_{k_s}^N\bigr)\longrightarrow 0.
\]
Assume instead the high-temperature coefficient conditions in Eqn.~\eqref{eq_clthigh}. Then
\[
 N^{s/2}\mathfrak c_s\bigl(p_{k_1}^N,\ldots,p_{k_s}^N\bigr)
 \longrightarrow 0.
\]
Consequently, the moment--cumulant relations, together with the first- and second-cumulant limits proved earlier, give the moment convergence in Definitions~\ref{def:fixedtempclt} and~\ref{Definition_CLT_sat}. 
\end{theorem}

\begin{proof}
The proof has four steps. We first expand each $c_\lambda$  in classical joint cumulants. We then isolate the linear terms of total degree $d$ and identify their coefficient matrix. Next we show, by induction on $d$, that all products of smaller-total-degree cumulants form an $o(1)$ remainder. Finally we invert the resulting finite-dimensional linear system.

\medskip
\noindent\textbf{Step 1. Cumulant expansion and coefficient estimates.}
We use the same expansion that appeared in the proofs of Theorems
\ref{thm:fixedtemponlyif} and \ref{thm:hightemponlyif}, but invert its
highest-total-degree linear part. Write
\[
 Z_k^{\mathrm F,N}:=\widetilde p_k^N,
 \qquad Z_k^{\mathrm H,N}:=p_k^N,
 \qquad
 Z_\mu^{\diamond,N}:=\prod_{i=1}^{\ell(\mu)}Z_{\mu_i}^{\diamond,N},
 \quad \diamond\in\{\mathrm F,\mathrm H\}.
\]
Thus
\[
 Z_\mu^{\mathrm F,N}=\frac{p_\mu(a/N)}{N^{\ell(\mu)}},
 \qquad
 Z_\mu^{\mathrm H,N}=\frac{p_\mu(a)}{N^{\ell(\mu)}}.
\]
For the two regimes, respectively, write directly
\[
 \log B_a(\vec x;\theta)
 =\sum_{\nu}\sum_{\mu\vdash |\nu|}
 C_{\mu,\nu}^{\diamond,N}Z_\mu^{\diamond,N}p_\nu(\vec x).
\]
For a fixed total degree only finitely many terms occur. Since
$G_{N,\theta}=\mathbb E[\exp(\log B_a)]$, the classical cumulant expansion gives, for $\lambda\vdash d$,
\begin{equation}\label{eq:clambda-cumulant-expansion}
 c_\lambda
 =\sum_{r=1}^{\ell(\lambda)}\frac1{r!}
 \sum_{\substack{\nu^{(1)},\ldots,\nu^{(r)}\ne\varnothing\\
 \nu^{(1)}\sqcup\cdots\sqcup\nu^{(r)}=\lambda}}
 \ \sum_{\substack{\mu^{(j)}\vdash |\nu^{(j)}|\\1\le j\le r}}
 \left(\prod_{j=1}^r C_{\mu^{(j)},\nu^{(j)}}^{\diamond,N}\right)
 \mathfrak c_r\bigl(
 Z_{\mu^{(1)}}^{\diamond,N},\ldots,
 Z_{\mu^{(r)}}^{\diamond,N}\bigr).
\end{equation}
Here $\sqcup$ denotes the multiset union of the parts. The sums over
$\nu^{(1)},\ldots,\nu^{(r)}$ are ordered; the factor $1/r!$ is therefore important when $\lambda$ has repeated parts.

We now separate the two kinds of terms that arise from Eqn.~\eqref{eq:clambda-cumulant-expansion}. Apply the Leonov--Shiryaev formula, Lemma \ref{lem:leonov}, to each classical cumulant in that equation. For a fixed summand, write
\[
 \rho=\mu^{(1)}\sqcup\cdots\sqcup\mu^{(r)},
\]
and let $\tau$ be the partition of its $\ell(\rho)$ atomic power sums into the $r$ product blocks corresponding to $\mu^{(1)},\ldots,\mu^{(r)}$. The Leonov--Shiryaev expansion is indexed by set partitions $\pi$ satisfying $\pi\vee\tau=\hat1$.

We call the contribution with $\pi=\hat1$ the \emph{linear part}, because it contains one joint cumulant of all $\ell(\rho)$ atomic power sums. When $\ell(\rho)\ge3$, these are precisely the unknown higher joint cumulants that will form the square linear system in Step 2. When $\ell(\rho)\le2$, they involve only the first cumulants and covariance and will be shown to be $o(1)$ directly.

We call the contributions with more than one block in $\pi$ the \emph{nonlinear part}. By the Leonov--Shiryaev formula, each such contribution is a product of joint cumulants, and every factor has total degree strictly smaller than $d$. Step 3 will show by induction on $d$ that the whole nonlinear part is $o(1)$. Step 4 will then invert the linear system obtained from the first part.

We need two facts about the coefficients in Eqn.~\eqref{eq:clambda-cumulant-expansion}. Lemmas \ref{lem:onlyiffixedtemp} and \ref{lem:onlyifhightemp} give
\begin{equation}\label{eq:C-bounds-higher-cumulants}
 C_{\mu,\nu}^{\mathrm F,N}\sim O\bigl(N^{2-\ell(\nu)}\bigr),
 \qquad
 C_{\mu,\nu}^{\mathrm H,N}\sim O\bigl(N^{1-\ell(\nu)}\bigr).
\end{equation}
The triangular moment--cumulant relations in Theorem \ref{thm:lln} and \cite[Theorem 3.10]{BCG}, together with the first lines of Eqns.~\eqref{eq_llnfixed} and \eqref{eq_llnhigh}, imply that the normalized one-row diagonal coefficients converge to nonzero limits:
\begin{equation}\label{eq:C-diagonal-higher-cumulants}
 \frac{C_{(k),(k)}^{\mathrm F,N}}{N}
 \longrightarrow \text{a nonzero limit},
 \qquad
 C_{(k),(k)}^{\mathrm H,N}
 \longrightarrow \text{a nonzero limit}.
\end{equation}

\medskip
\noindent\textbf{Step 2. The linear system in total degree $d$.}
Fix $d\ge3$. We now construct, simultaneously for all partitions of total degree $d$, the finite-dimensional linear system whose unknowns are the scaled higher joint cumulants. Let
\[
 \mathcal P_{\ge3}(d):=\{\rho\vdash d:\ell(\rho)\ge3\}.
\]
For $\rho\in\mathcal P_{\ge3}(d)$ define
\[
 K_\rho^{\mathrm F,N}
 :=N^{\ell(\rho)}\mathfrak c_{\ell(\rho)}
 \bigl(Z_{\rho_1}^{\mathrm F,N},\ldots,Z_{\rho_{\ell(\rho)}}^{\mathrm F,N}\bigr),
\]
\[
 K_\rho^{\mathrm H,N}
 :=N^{\ell(\rho)/2}\mathfrak c_{\ell(\rho)}
 \bigl(Z_{\rho_1}^{\mathrm H,N},\ldots,Z_{\rho_{\ell(\rho)}}^{\mathrm H,N}\bigr).
\]
We shall obtain a square linear system for the vectors
$K_d^{\diamond,N}:=(K_\rho^{\diamond,N})_{\rho\in\mathcal P_{\ge3}(d)}$.

We first treat the linear terms with $\ell(\rho)\ge3$. Such a term is proportional to $K_\rho^{\diamond,N}$. In the fixed-temperature regime, its coefficient after extracting
$K_\rho^{\mathrm F,N}$ is, by Eqn.~\eqref{eq:C-bounds-higher-cumulants},
\[
 O\bigl(N^{2r-\ell(\lambda)-\ell(\rho)}\bigr).
\]
Indeed, the product of the $C$-coefficients is
$O(N^{2r-\ell(\lambda)})$, while the unscaled joint cumulant is
$N^{-\ell(\rho)}K_\rho^{\mathrm F,N}$. Since $r\le\ell(\lambda)$ and $r\le\ell(\rho)$, the exponent is nonpositive, and it is zero only if
$r=\ell(\lambda)=\ell(\rho)$. In that equality case all the $\nu^{(j)}$ and $\mu^{(j)}$ are
one-row partitions, so necessarily $\rho=\lambda$. Hence every
non-diagonal coefficient is $o(1)$.

The same argument in the high-temperature regime is applied after multiplying both sides of Eqn.~\eqref{eq:clambda-cumulant-expansion} by $N^{\ell(\lambda)/2}$, since, by assumption, $N^{\ell(\lambda)/2}c_\lambda\to 0$. The coefficient of
$K_\rho^{\mathrm H,N}$ is
\[
 O\bigl(N^{r-\ell(\lambda)/2-\ell(\rho)/2}\bigr)
 \sim O\bigl(N^{(2r-\ell(\lambda)-\ell(\rho))/2}\bigr),
\]
and equality in the exponent again occurs only for the diagonal term
$r=\ell(\lambda)=\ell(\rho)$ and $\rho=\lambda$.

It remains only to note that the diagonal limit is nonzero. The ordered single-row decompositions of $\lambda$, together with the factor $1/\ell(\lambda)!$ in Eqn.~\eqref{eq:clambda-cumulant-expansion}, contribute a fixed nonzero combinatorial factor depending on the multiplicities of the parts of $\lambda$. Apart from this factor, the fixed-temperature diagonal entry is
\[
 \prod_{i=1}^{\ell(\lambda)}
 \frac{C_{(\lambda_i),(\lambda_i)}^{\mathrm F,N}}{N},
\]
and the high-temperature diagonal entry is
\[
 \prod_{i=1}^{\ell(\lambda)}
 C_{(\lambda_i),(\lambda_i)}^{\mathrm H,N}.
\]
Both converge to nonzero limits by Eqn.~\eqref{eq:C-diagonal-higher-cumulants}.

We finally treat the linear terms with $\ell(\rho)\le2$. These terms are not coordinates of $K_d^{\diamond,N}$, since that vector is indexed by $\rho\in\mathcal P_{\ge3}(d)$. The case $\ell(\rho)=1$ consists of first cumulants, while $\ell(\rho)=2$ consists of covariances. For $\ell(\rho)=2$, the covariance convergence gives $K_\rho^{\diamond,N}\sim O(1)$. By the same argument as in case $\ell(\rho)\ge 3$, the coefficient after extracting
$K_\rho^{\mathrm \diamond,N}$ is $O(N^{2r-\ell(\lambda)-\ell(\rho)})$ in the fixed-temperature regime and $O(N^{r-\ell(\lambda)/2-\ell(\rho)/2})$ in the high-temperature regime, both of which are $o(1)$ since $r\le \ell(\rho)$ and $\ell(\lambda)\ge 3$. Combining the two factors above shows that the contribution of linear terms with $\ell(\rho)=2$ is $o(1)$. For $\ell(\rho)=1$, $r=1$ in Eqn.~\eqref{eq:clambda-cumulant-expansion}, and the corresponding terms in $c_\lambda$ ($N^{\ell(\lambda)/2}c_\lambda$ resp.) are of the form $C^{F,N}_{(d),\lambda}\cdot \E[Z^{\diamond,N}_{(d)}]$ ($N^{\ell(\lambda)/2}C^{H,N}_{(d),\lambda}\cdot \E[Z^{\diamond,N}_{(d)}]$ resp.). The LLN implies $\E[Z^{\diamond,N}_d]\sim O(1)$, and $C^{F,N}_{(d),\lambda}$ ($N^{\ell(\lambda)/2}C^{H,N}_{(d),\lambda}$ resp.) is $o(1)$ by Eqn.~\eqref{eq:C-bounds-higher-cumulants}. Thus the contribution of linear terms with $\ell(\rho)=1$ is also $o(1)$.

\medskip
\noindent\textbf{Step 3. The nonlinear remainder and the induction step.}
The goal of this step is to show that the nonlinear part defined after Eqn.~\eqref{eq:clambda-cumulant-expansion} is $o(1)$ at the relevant scale. Thus throughout this step $\pi$ has more than one block. We first remove the case $r=1$, then prove the claim for $r\ge2$ by induction on the total degree.

We first treat the nonlinear terms with $r=1$. The LLN implies that all joint moments, and hence all classical cumulants, of the normalized power sums are bounded. Fix a partition $\pi$ with more than one block in the Leonov--Shiryaev expansion of one such term. The product of cumulants indexed by the blocks of $\pi$ is therefore $O(1)$. Since $\nu^{(1)}=\lambda$, the corresponding complete summand in Eqn.~\eqref{eq:clambda-cumulant-expansion}, including its coefficient $C_{\mu^{(1)},\lambda}^{\mathrm F,N}$, is $O(N^{2-\ell(\lambda)})\sim o(1)$ in the fixed-temperature regime. In the high-temperature regime, the corresponding complete summand, including $C_{\mu^{(1)},\lambda}^{\mathrm H,N}$ and multiplied by the prescribed scale $N^{\ell(\lambda)/2}$, is $O(N^{1-\ell(\lambda)/2})\sim o(1)$. This disposes of every nonlinear term with one product block.

We next set up the estimate for $r\ge2$. Define
\[
 S(\pi):=\sum_{B\in\pi:\,|B|\ge2}|B|.
\]
Recall from Remark \ref{rem:leonov-connectedness} that $S(\pi)\ge r$. Moreover, because $\pi$ has more than one block and all parts of $\rho$ are positive, every $B\in\pi$ satisfies
\begin{equation}\label{eq:smaller-total-degree-block}
 \sum_{i\in B}\rho_i\le d-1.
\end{equation}
This is the point that makes induction on the total degree available: every cumulant factor in the nonlinear term has strictly smaller total degree than the cumulant currently under consideration.

We now carry out the induction on $d$, beginning with the base case $d=3$. If $|B|\ge3$, then the left-hand side of Eqn.~\eqref{eq:smaller-total-degree-block} is at least $3$, contradicting that it is at most $2$. Hence every block of $\pi$ has size one or two. The first-cumulant factors are $O(1)$ by the LLN; the second-cumulant factors are $O(N^{-2})$ in the fixed-temperature regime and $O(N^{-1})$ in the high-temperature regime. Their product is therefore $O(N^{-S(\pi)})$ in the fixed-temperature regime and $O(N^{-S(\pi)/2})$ in the high-temperature regime.

We next combine these base-case bounds with the coefficients of Eqn.~\eqref{eq:clambda-cumulant-expansion}. The whole fixed-temperature term is
\[
 O\bigl(N^{2r-\ell(\lambda)-S(\pi)}\bigr),
\]
and, after the high-temperature normalization, the whole high-temperature term is
\[
 O\bigl(N^{r-\ell(\lambda)/2-S(\pi)/2}\bigr).
\]
Both exponents are nonpositive because $\ell(\lambda)\ge r$ and $S(\pi)\ge r$. Equality would force $\ell(\lambda)=S(\pi)=r$. By Remark \ref{rem:leonov-connectedness}, this would require a non-singleton block of size $r=\ell(\lambda)\ge3$, which is impossible in the base case. Thus both exponents are strictly negative, and the nonlinear part is $o(1)$ for $d=3$.

For the induction step, let $d\ge4$ and assume that every higher joint cumulant of total degree at most $d-1$ vanishes at the claimed scale. Equation~\eqref{eq:smaller-total-degree-block}, together with $\rho_i\ge1$, shows that every block corresponding to a higher joint cumulant satisfies
\[
 3\le |B|\le \sum_{i\in B}\rho_i\le d-1.
\]
We may therefore apply the induction hypothesis separately to each block $B$ with $3\le |B|\le d-1$. Blocks with one element are controlled by the LLN, and blocks with two elements are controlled by the already proved covariance convergence. Consequently,
\[
 \mathfrak c_{|B|}\bigl(Z_{\rho_i}^{\mathrm F,N}:i\in B\bigr)
 =\begin{cases}
 O(1),&|B|=1,\\
 O(N^{-2}),&|B|=2,\\
 o(N^{-|B|}),&3\le |B|\le d-1,
 \end{cases}
\]
and
\[
 \mathfrak c_{|B|}\bigl(Z_{\rho_i}^{\mathrm H,N}:i\in B\bigr)
 =\begin{cases}
 O(1),&|B|=1,\\
 O(N^{-1}),&|B|=2,\\
 o(N^{-|B|/2}),&3\le |B|\le d-1.
 \end{cases}
\]
These are precisely the induction estimates needed for the individual factors of the nonlinear term.

We then multiply the estimates over all blocks of $\pi$. Ignoring for the moment the little-$o$ improvement, their product is $O(N^{-S(\pi)})$ in the fixed-temperature regime and $O(N^{-S(\pi)/2})$ in the high-temperature regime. Combining this with Eqn.~\eqref{eq:C-bounds-higher-cumulants}, the complete terms have the same two bounds as in the base case:
\[
 O\bigl(N^{2r-\ell(\lambda)-S(\pi)}\bigr)
 \quad\text{and}\quad
 O\bigl(N^{r-\ell(\lambda)/2-S(\pi)/2}\bigr),
\]
respectively. Again, both exponents are nonpositive.

It remains to handle the critical case where one of these exponents is zero. This forces $\ell(\lambda)=S(\pi)=r$. By the equality statement in Remark \ref{rem:leonov-connectedness}, $\pi$ has a unique non-singleton block, this block meets each product block once, and every other block is a singleton. Since $\pi$ has more than one block, the non-singleton block has total degree at most $d-1$; its size satisfies $3\le r=\ell(\lambda)\le d-1$. The induction hypothesis gives a little-$o$ estimate for this factor, so the full critical term is still $o(1)$. When the exponent is strictly negative, the displayed big-$O$ bound already gives the same conclusion.

We have now treated every nonlinear term. For fixed $d$ the sums in Eqn.~\eqref{eq:clambda-cumulant-expansion} are finite, so the sum of these terms is $o(1)$. This completes both the base case and the induction step from total degrees at most $d-1$ to total degree $d$.

\medskip
\noindent\textbf{Step 4. Inversion of the coefficient matrix.}
We have proved, for the two regimes respectively,
\[
 (c_\lambda)_{\lambda\in\mathcal P_{\ge3}(d)}
 =A_d^{\mathrm F,N}K_d^{\mathrm F,N}+o(1),
\]
\[
 (N^{\ell(\lambda)/2}c_\lambda)_{\lambda\in\mathcal P_{\ge3}(d)}
 =A_d^{\mathrm H,N}K_d^{\mathrm H,N}+o(1),
\]
where Step 2 shows that the two coefficient matrices converge entrywise to diagonal matrices with nonzero diagonal entries. For fixed $d$ these
are finite square matrices and their diagonal entries are nonzero, so
$A_d^{\diamond,N}$ is invertible for all sufficiently large $N$ and its
inverse is uniformly bounded. The left-hand sides tend to zero by
Eqns.~\eqref{eq_cltfixed} and \eqref{eq_clthigh}; hence
$K_d^{\mathrm F,N}\to0$ and $K_d^{\mathrm H,N}\to0$. This completes the
induction and the proof.
\end{proof}

\section{Combinatorial identification of classical second-order free cumulants}\label{sec:comb}
In this section, we show that in the fixed-temperature regime, up to the constant factor $\frac{1}{\theta}$, our covariance--cumulant formula is equivalent to the classical expressions in terms of non-crossing annular permutations and second-order non-crossing annular partitions, originating in \cite{CMSS07}; see also \cite[Chapter 5]{MS}. In particular, our proof gives a direct correspondence between these combinatorial objects and the walk-block configurations of Lukasiewicz paths.

For simplicity, for the rest of this section we set $\theta=1$, which corresponds to the case where we have complex Hermitian matrices.

We first recall the following classical notions. For $n\in \Z_{\ge 1}$, let $S_{n}$ denote the symmetric group of $n$ elements. Let $m,n\in \Z_{\ge 1}$ as before.

\begin{definition}
    For $n\in \Z_{\ge 1}$, a partition $\pi=\sqcup_{i} B_{i}$ of $\{1,2,...,n\}$ is \emph{non-crossing} if there are no $i_{1},i_{2}\in B_{i}$, $j_{1},j_{2}\in B_{j}$ from two distinct blocks such that $i_{1}<j_{1}<i_{2}<j_{2}$. We call $B_{i}$ the blocks of $\pi$.

    There is a canonical way to draw each partition as a graph on a horizontal line or circle, where there are $n$ points indexed by $1,2,...,n$, and the points with indices in the same block are connected by the so-called \emph{legs} and a \emph{roof}. We call the leg connected to $\min(B_i)$ of some $B_i$ a \emph{first leg}, and all the other legs \emph{auxiliary legs}. When $\pi$ is non-crossing, the legs from distinct blocks do not intersect. We denote the set of non-crossing partitions of $\{1,\ldots,n\}$ by $NC(n)$.
\end{definition}

\begin{definition}\cite[Section 5.1]{MS}\label{def:noncrossingpermutation}
    The $(m,n)$-annulus is the annulus with the integers 1 to $m$ arranged clockwise on the outer circle and $m+1$ to $m+n$ arranged counterclockwise on the inner circle. A permutation $\pi$ in $S_{m+n}$ is a \emph{non-crossing permutation} on the $(m,n)$-annulus if we can draw the cycles of $\pi$ between the circles of the annulus so that

    (i). the cycles do not cross,

    (ii). each cycle encloses a region between the circles homeomorphic to a disk with boundary oriented \textbf{clockwise}, and

    (iii). at least one cycle connects the two circles.

    We denote the set of non-crossing permutations on the $(m,n)$-annulus by $S_{NC}(m,n)$. In particular, if each cycle in the permutation is of size 2, we call it a non-crossing pairing, and denote the set of all non-crossing pairings on the $(m,n)$-annulus by $NC_{2}(m,n)$.
\end{definition}

\begin{proposition}\cite[Section 5.1]{MS}\label{prop:throughcycle}
    Consider $\pi\in S_{NC}(m,n)$. If $\pi$ has more than one through-cycle (a cycle containing elements from both circles), then one can identify $\pi$ with its corresponding partition of $\{1,2,...,m+n\}$. However, there might be more than one element in $S_{NC}(m,n)$ corresponding to the same partition if the partition has only one block with elements from both $\{1,2,...,m\}$ and $\{m+1,...,m+n\}$.
\end{proposition}

\begin{example}\label{ex:1234}
    For $m=n=2$, there are four elements in $S_{NC}(2,2)$ that correspond to the same partition $B_{1}=\{1,2,3,4\}$:
    $$(1,2,3,4),\quad (1,2,4,3),\quad (1,3,4,2),\quad (1,4,3,2).$$
\end{example}
The four annular drawings are shown in Figure~\ref{fig:snc22}.

\begin{figure}[htbp]
\centering
\begin{tikzpicture}[
    >={Stealth[length=1.6mm,width=1.15mm]},
    annular edge/.style={blue,thick,shorten <=1.2pt,shorten >=1.2pt,postaction={decorate},decoration={markings,mark=at position 0.55 with {\arrow{>}}}},
    every node/.style={font=\small}
]
\newcommand{\SecSevenAnnulusBase}{%
    \draw[thin] (0,0) circle (1.25);
    \draw[thin] (0,0) circle (0.55);
    \fill (0,1.25) circle (1.4pt) node[above=2pt] {$1$};
    \fill (0,-1.25) circle (1.4pt) node[below=2pt] {$2$};
    \fill (0.55,0) circle (1.4pt) node[left=2pt] {$3$};
    \fill (-0.55,0) circle (1.4pt) node[right=2pt] {$4$};
}

\begin{scope}[xshift=0cm]
    \SecSevenAnnulusBase
    \draw[annular edge] (0,1.25) .. controls (1.02,0.86) and (1.02,-0.86) .. (0,-1.25);
    \draw[annular edge] (0,-1.25) .. controls (0.18,-1.08) and (0.72,-0.54) .. (0.55,0);
    \draw[annular edge] (0.55,0) .. controls (0.55,0.90) and (-0.55,0.90) .. (-0.55,0);
    \draw[annular edge] (-0.55,0) .. controls (-0.72,0.54) and (-0.18,1.08) .. (0,1.25);
    \node at (0,-2.02) {$(1,2,3,4)$};
\end{scope}

\begin{scope}[xshift=4.0cm]
    \SecSevenAnnulusBase
    \draw[annular edge] (0,1.25) .. controls (1.02,0.86) and (1.02,-0.86) .. (0,-1.25);
    \draw[annular edge] (0,-1.25) .. controls (-0.18,-1.08) and (-0.72,-0.54) .. (-0.55,0);
    \draw[annular edge] (-0.55,0) .. controls (-0.55,-0.90) and (0.55,-0.90) .. (0.55,0);
    \draw[annular edge] (0.55,0) .. controls (0.72,0.54) and (0.18,1.08) .. (0,1.25);
    \node at (0,-2.02) {$(1,2,4,3)$};
\end{scope}

\begin{scope}[xshift=8.0cm]
    \SecSevenAnnulusBase
    \draw[annular edge] (0,1.25) .. controls (0.18,1.08) and (0.72,0.54) .. (0.55,0);
    \draw[annular edge] (0.55,0) .. controls (0.55,0.90) and (-0.55,0.90) .. (-0.55,0);
    \draw[annular edge] (-0.55,0) .. controls (-0.72,-0.54) and (-0.18,-1.08) .. (0,-1.25);
    \draw[annular edge] (0,-1.25) .. controls (-1.02,-0.86) and (-1.02,0.86) .. (0,1.25);
    \node at (0,-2.02) {$(1,3,4,2)$};
\end{scope}

\begin{scope}[xshift=12.0cm]
    \SecSevenAnnulusBase
    \draw[annular edge] (0,1.25) .. controls (-0.18,1.08) and (-0.72,0.54) .. (-0.55,0);
    \draw[annular edge] (-0.55,0) .. controls (-0.55,-0.90) and (0.55,-0.90) .. (0.55,0);
    \draw[annular edge] (0.55,0) .. controls (0.72,-0.54) and (0.18,-1.08) .. (0,-1.25);
    \draw[annular edge] (0,-1.25) .. controls (-1.02,-0.86) and (-1.02,0.86) .. (0,1.25);
    \node at (0,-2.02) {$(1,4,3,2)$};
\end{scope}
\end{tikzpicture}
\caption{The four elements of $S_{NC}(2,2)$ in Example~\ref{ex:1234}. The points $1,2$ are ordered clockwise on the outer circle and $3,4$ counterclockwise on the inner circle. Each cycle bounds a disk in the annulus and follows its boundary clockwise; arrows in the middle of the arcs record the cycle order.}
\label{fig:snc22}
\end{figure}
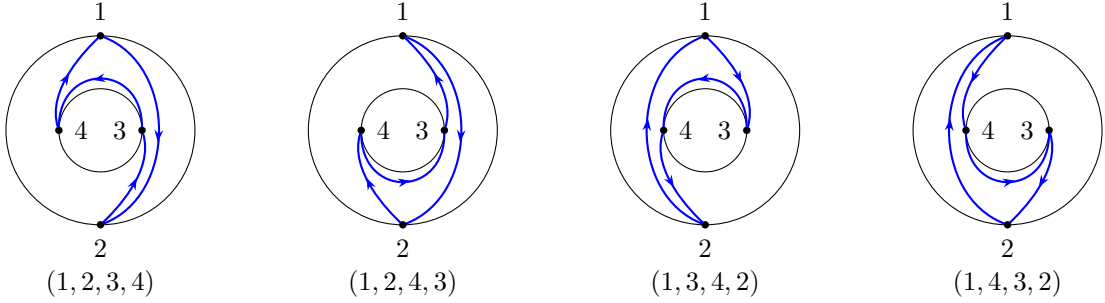

\begin{definition}\cite[Section 5.4]{MS}
    The second-order non-crossing annular partitions, $[NC(m)\times NC(n)]$, consist of elements $\sigma=(\pi_{1},M_{1})\times (\pi_{2},M_{2})$, where $\pi_{1}\times \pi_{2}\in NC(m)\times NC(n)$ and where $M_{1}\in \pi_{1}$ and $M_{2}\in \pi_{2}$. The blocks $M_{1}$ and $M_{2}$ are designated as marked blocks of $\pi_{1}$ and $\pi_{2}$, respectively.
\end{definition}

\begin{theorem}
    The classical covariance--cumulant formula is due to Collins, Mingo,
    \'{S}niady, and Speicher \cite{CMSS07}; see also \cite[Section 5.4]{MS}.
    In the fixed-temperature regime, for $\theta=1$, when the sequence of measures $\{\mu_{N}\}_{N\ge 1}$ satisfies a CLT, its limiting covariance $\alpha_{m,n}$ can be expressed in terms of the first-order free cumulants $\{\kappa_{l}\}_{l\ge 1}$ and second-order free cumulants $\{\kappa_{k_{1},k_{2}}\}_{k_{1},k_{2}\ge 1}$, by the following formula:
    \begin{equation}\label{eq_classical}
        \alpha_{m,n}=\sum_{\pi\in S_{NC}(m,n)}\prod_{B_{i}\in \pi}\kappa_{|B_{i}|}+\sum_{\sigma\in [NC(m)\times NC(n)]}\kappa_{\sigma},
    \end{equation}
    where for each $\sigma\in [NC(m)\times NC(n)]$ that consists of unmarked blocks $B_{1,i}$'s in $NC(m)$ and a marked block $M_{1}$ with size $k_{1}$, unmarked blocks $B_{2,j}$'s in $NC(n)$ and a marked block $M_{2}$ with size $k_{2}$, we let
    \begin{equation}
        \kappa_{\sigma}=\prod_{i}\kappa_{|B_{1,i}|}\prod_{j}\kappa_{|B_{2,j}|}\cdot \kappa_{k_{1},k_{2}}.
    \end{equation}
\end{theorem}

The main result of this section is the following.
\begin{theorem}
    For any $m,n\in \Z_{\ge 1}$, given the same collections of first-order free cumulants and second-order free cumulants, the two covariance--cumulant formulas given in Eqn.~\eqref{eq_fixedtempcov} and (\ref{eq_classical}) are equivalent, when $\theta=1$. More precisely, we have
    \begin{equation}\label{eq_equiv1}
        \sum_{E\in L(m,n)}A(E)B(E)C(E)\prod_{B_i\in E}\kappa_{|B_{i}|}=\sum_{\pi\in S_{NC}(m,n)}\prod_{B_{i}\in \pi}\kappa_{|B_{i}|},
                  \end{equation}
                  and
    \begin{equation}\label{eq_equiv2}
        \sum_{E\in L(m,n)}\sum_{s\in \{\bar{\tau},...,m\}\cap U(E)}B(E)\kappa_{k_{1}(E,s),k_{2}(E)}\prod_{B_i\in E,\ t(B_{i})\ne s}\kappa_{|B_{i}|}=\sum_{\sigma\in [NC(m)\times NC(n)]}\kappa_{\sigma}.
    \end{equation}
\end{theorem}
\begin{proof}
    The equivalence between Eqn.~\eqref{eq_fixedtempcov} and (\ref{eq_classical}) follows directly from Eqn.~\eqref{eq_equiv1} and (\ref{eq_equiv2}), and we start by proving Eqn.~\eqref{eq_equiv1}.
    First, it is classical (see, e.g., \cite[Proposition 9.8]{NS}) that one can associate any $\pi\in NC(m+n)$ with a unique element $E\in L(m+n)$, according to the following relation:
    \begin{itemize}
        \item For each block $B_{i}\in \pi$, let $t_{i}\in \{1,2,...,m+n\}$ be the minimal element in $B_{i}$. Then the corresponding $E$ of $\pi$ has an up or flat step at $t_{i}$ of height $|B_{i}|-1$. At every time not among the $t_i$'s, $E$ has a down step of height $-1$.
        \item One can check directly that the above map is bijective.
    \end{itemize}
    Therefore we identify $E$ with a partition $\pi\in NC(m+n)$. Furthermore, we say that $\tilde{\pi}\in S_{NC}(m,n)$ corresponds to $\pi\in NC(m+n)$ if there is a one-to-one correspondence of each cycle in $\tilde{\pi}$ with each block in $\pi$, where the two sides have the same minimal element in $\{1,2,...,m+n\}$ and cardinality. Clearly, by definition any $\tilde{\pi}$ corresponds to exactly one $\pi$.

    After all these preparations, it suffices to check that for a fixed $\pi\in NC(m+n)$, the number of elements $\tilde{\pi}$ in $S_{NC}(m,n)$ corresponding to $\pi$ (or its corresponding $E$) is given by $A(E)B(E)C(E)$.

    The trivial case is when $C(E)=0$: this corresponds to the non-crossing partitions on $(m,n)$-annulus, where there is no through-cycle, hence outside of $S_{NC}(m,n)$ by Definition \ref{def:noncrossingpermutation}.

    For $\pi$ where $C(E)\ge 1$, we set a ``default'' element $\tilde{\pi}_{0}$ in $S_{NC}(m,n)$ corresponding to $\pi$, such that as a non-crossing partition of $\{1,2,...,m+n\}$, $\tilde{\pi}_{0}$ is the same as $\pi$. In particular, when drawn on the annulus, the blocks/cycles\footnote{These two are identical when drawn purely on the inner or outer circle.} on $\{m+1,...,\tau-1\}$ (which might be empty) form a complete non-crossing partition on the inner circle, and at position $\tau$ on the inner circle there must be a path connecting the inner and outer circles.

This default drawing can be obtained directly from the disk drawing of $\pi$: cut the disk between the positions $m$ and $m+1$, and bend the boundary segment $m+1,\ldots,m+n$ into the inner circle, reversing its boundary orientation. The path height $E(m)=C(E)$ counts the legs open across this splitting point. Thus $C(E)>0$ guarantees at least one through-cycle, and the construction preserves every first-leg position and block cardinality.
We fix the cycle order of $\tilde\pi_0$ by listing the elements of each block of $\pi$ in increasing order. Throughout the reconnections below, position labels and first-leg positions stay fixed; only the attachments of auxiliary legs are reassigned. An enclosing block means a block whose span contains the block under consideration in the default disk partition.

\emph{Inner reconnection.}
Let $V_0$ be the inner positions belonging to through-cycles of $\tilde\pi_0$. Then $|V_0|=C(E)$ and $\min V_0=\tau$. The available auxiliary positions are those before $\tau$ together with $V_0$; auxiliary positions in purely inner blocks after $\tau$ do not participate.

Besides the default choice, choose an auxiliary position $e_i\in\{m+1,\ldots,\tau-1\}\cap D(E)$. Select $e_i$ and all auxiliary positions to its right but before $\tau$ belonging to its block or to an enclosing block. If there are $r$ selected positions, list them together with $V_0$ as
\[
    e_i=u_1<\cdots<u_{r+C(E)}.
\]
Reassign the selected attachments at $u_j$ to $u_{C(E)+j}$, $1\le j\le r$, preserving their block membership and order. Reassign the through-cycle attachments, in their previous order, to
\[
    V=\{u_1,\ldots,u_{C(E)}\}.
\]
All other attachments stay fixed, and overlap between the source and receiving sets is allowed. For the default choice, set $V=V_0$. In particular, $|V|=C(E)$ in every case.

The selected blocks form a nested chain. Their attachments retain their order, while every untouched block, together with its interior blocks, remains in a gap between the reconnected legs. Thus the reconnection is non-crossing and preserves all first legs and block sizes. Its choice is recovered from the least inner position belonging to a through-cycle: this is $\tau$ in the default case and $e_i$ otherwise. There are therefore
\[
    1+\bigl|\{m+1,\ldots,\tau-1\}\cap D(E)\bigr|=B(E)
\]
inner choices. They determine all purely inner cycles and the set $V$, without changing the outer-circle attachments.

\emph{Outer reconnection.}
Recall that $\bar\tau$ is the largest first-leg position among the default through-cycles. Besides the default choice, choose
\[
    e\in\{\bar\tau+1,\ldots,m\}\cap D(E).
\]
Select $e$ and all auxiliary positions to its right on the outer circle belonging to its block or to an enclosing block. Thus several nested blocks may participate. The pull-off cuts the attachments along this chain and reconnects them in cyclic order with the positions in $V$. The unchanged blocks remain in the intervening gaps, so no crossing is introduced.

For precision, the cyclic reassignment can be specified as follows. List the $r$ selected outer positions, followed by the $C(E)$ positions of $V$, in increasing order as $u_1,\ldots,u_{r+C(E)}$, with $u_1=e$. In every cycle of the intermediate configuration, reassign these auxiliary entries by $u_j\mapsto u_{j+r}$, reading subscripts modulo $r+C(E)$, and leave all other entries fixed. This preserves cycle sizes and first legs. The additional cyclic choice among the positions of $V$ is left for the final gluing below. In particular, the edge redirected into the inner circle starts at $\tilde\pi_0^{-1}(e)$, the \emph{default predecessor of $e$}. With no pull-off, the corresponding reference edge starts at $\tilde\pi_0^{-1}(\tau)$. There are
\[
    1+\bigl|\{\bar\tau+1,\ldots,m\}\cap D(E)\bigr|=A(E)
\]
outer choices, each leaving $C(E)$ inner positions for the final cyclic gluing. Figure~\ref{fig:flip-pulloff} illustrates the special case in which the selected attachments belong to a single block.

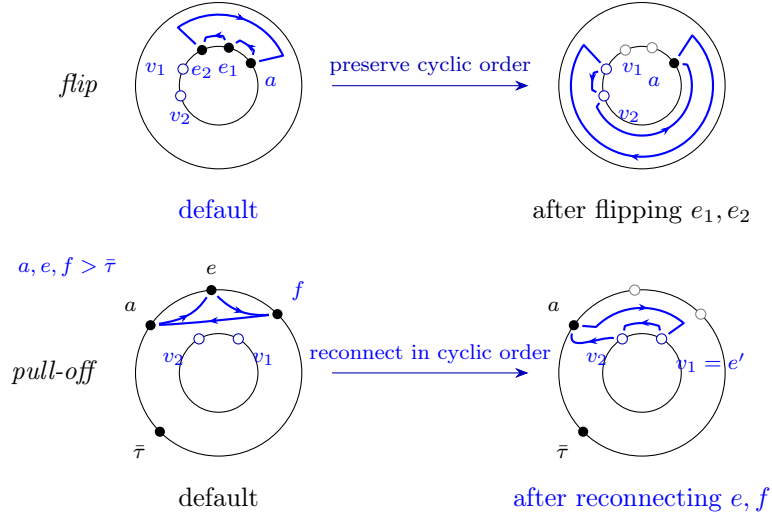
\begin{figure}[htbp]
\centering
\begin{tikzpicture}[
    >={Stealth[length=1.3mm,width=0.95mm]},
    perm edge/.style={blue,thick,line join=round,shorten <=1.1pt,shorten >=1.1pt,postaction={decorate},decoration={markings,mark=at position 0.55 with {\arrow{>}}}},
    moved edge/.style={perm edge},
    fixed point/.style={circle,fill=black,inner sep=1.35pt},
    available point/.style={circle,draw=blue!70!black,fill=white,inner sep=1.35pt},
    freed point/.style={circle,draw=gray,fill=white,inner sep=1.35pt},
    every node/.style={font=\small}
]
\newcommand{\ReconnectionAnnulus}{%
    \draw[thin] (0,0) circle (1.10);
    \draw[thin] (0,0) circle (0.52);
}

\begin{scope}[xshift=0cm,yshift=2.05cm]
    \ReconnectionAnnulus
    \node[fixed point,label={[font=\scriptsize,blue]below right:$a$}] (fa) at (35:0.52) {};
    \node[fixed point,label={[font=\scriptsize,blue]below:$e_1$}] (feone) at (75:0.52) {};
    \node[fixed point,label={[font=\scriptsize,blue]below:$e_2$}] (fetwo) at (115:0.52) {};
    \node[available point,label={[font=\scriptsize,blue]left:$v_1$}] (fvone) at (155:0.52) {};
    \node[available point,label={[font=\scriptsize,blue]below:$v_2$}] (fvtwo) at (195:0.52) {};
    \draw[perm edge] (fa) -- (45:0.66) arc[start angle=45,end angle=65,radius=0.66] -- (feone);
    \draw[perm edge] (feone) -- (85:0.66) arc[start angle=85,end angle=105,radius=0.66] -- (fetwo);
    \draw[perm edge] (fetwo) -- (125:0.94) arc[start angle=125,end angle=25,radius=0.94] -- (fa);
    \node[below=1.38cm,blue] at (0,0) {default};
\end{scope}

\begin{scope}[xshift=5.60cm,yshift=2.05cm]
    \ReconnectionAnnulus
    \node[fixed point,label={[font=\scriptsize,blue]below left:$a$}] (fba) at (35:0.52) {};
    \node[freed point] at (75:0.52) {};
    \node[freed point] at (115:0.52) {};
    \node[available point,label={[font=\scriptsize,blue]right:$v_1$}] (fbvone) at (155:0.52) {};
    \node[available point,label={[font=\scriptsize,blue]below right:$v_2$}] (fbvtwo) at (195:0.52) {};
    \draw[moved edge] (fba) -- (45:0.94) arc[start angle=45,end angle=-215,radius=0.94] -- (fbvone);
    \draw[moved edge] (fbvone) -- (165:0.66) arc[start angle=165,end angle=185,radius=0.66] -- (fbvtwo);
    \draw[moved edge] (fbvtwo) -- (205:0.66) arc[start angle=205,end angle=385,radius=0.66] -- (fba);
    \node[below=1.38cm] at (0,0) {after flipping $e_1,e_2$};
\end{scope}
\draw[blue!70!black,-{Stealth[length=1.8mm]}] (1.50,2.05) -- (4.10,2.05)
    node[midway,above,font=\scriptsize] {preserve cyclic order};
\node[anchor=east,font=\small] at (-1.45,2.05) {\emph{flip}};

\node[font=\scriptsize,blue] at (-2,-0.35) {$a,e,f > \bar\tau$};
\begin{scope}[xshift=0cm,yshift=-1.75cm]
    \ReconnectionAnnulus
    \node[fixed point,label={[font=\scriptsize]above left:$a$}] (pa) at (145:1.10) {};
    \node[fixed point,label={[font=\scriptsize]above:$e$}] (peone) at (95:1.10) {};
    \node[fixed point,label={[font=\scriptsize,blue]above right:$f$}] (petwo) at (45:1.10) {};
    \node[available point,label={[font=\scriptsize,blue]below right:$v_1$}] at (60:0.52) {};
    \node[available point,label={[font=\scriptsize,blue]below left:$v_2$}] at (120:0.52) {};
    \draw[perm edge] (pa) .. controls (123:0.84) and (111:0.84) .. (peone);
    \draw[perm edge] (peone) .. controls (78:0.84) and (63:0.84) .. (petwo);
    \draw[perm edge] (petwo) .. controls (75:0.74) and (118:0.74) .. (pa);
    \node[fixed point,label={[font=\scriptsize]below left:$\bar\tau$}] (pt) at (225:1.10) {};
    \node[below=1.38cm] at (0,0) {default};
\end{scope}

\begin{scope}[xshift=5.60cm,yshift=-1.75cm]
    \ReconnectionAnnulus
    \node[fixed point,label={[font=\scriptsize]above left:$a$}] (pba) at (145:1.10) {};
    \node[freed point] at (95:1.10) {};
    \node[freed point] at (45:1.10) {};
    \node[available point,label={[font=\scriptsize,blue]below right:$v_1=e'$}] (pbvone) at (60:0.52) {};
    \node[available point,label={[font=\scriptsize,blue]below left:$v_2$}] (pbvtwo) at (120:0.52) {};
    \draw[moved edge] (pba) -- (135:0.86) arc[start angle=135,end angle=50,radius=0.86] -- (pbvone);
    \draw[moved edge] (pbvone) -- (70:0.66) arc[start angle=70,end angle=110,radius=0.66] -- (pbvtwo);
    \draw[moved edge] (pbvtwo) .. controls (140:0.75) and (165:1.02) .. (pba);
    \node[fixed point,label={[font=\scriptsize]below left:$\bar\tau$}] (pbt) at (225:1.10) {};
    \node[below=1.38cm,blue] at (0,0) {after reconnecting $e,f$};
\end{scope}
\draw[blue!70!black,-{Stealth[length=1.8mm]}] (1.50,-1.75) -- (4.10,-1.75)
    node[midway,above,font=\scriptsize,blue] {reconnect in cyclic order};
\node[anchor=east,font=\small] at (-1.45,-1.75) {\emph{pull-off}};
\end{tikzpicture}
\caption{Single-block examples of the reconnections; every displayed cycle is oriented clockwise. Top: the attachments at $e_1,e_2$ are reassigned to $v_1,v_2$. Bottom: the selected outer positions are $e,f$, with $a,e,f>\bar\tau$. Here $a=\tilde\pi_0^{-1}(e)$ and the chosen inner endpoint is $e'=v_1$, so the new cycle contains the edge $a\to e'$. Position labels stay fixed. Hollow blue and gray points indicate reassigned attachments, not empty vertices; all complementary cycles are omitted. In general, the selected positions may belong to several nested blocks as described in the proof.}
\label{fig:flip-pulloff}
\end{figure}

\emph{Cyclic gluing and inverse recovery.}
For fixed inner and outer choices, choose $e'\in V$ and require
\[
    e'=\begin{cases}
        \tilde\pi\bigl(\tilde\pi_0^{-1}(e)\bigr),&\text{for the outer choice }e,\\
        \tilde\pi\bigl(\tilde\pi_0^{-1}(\tau)\bigr),&\text{for the default outer choice}.
    \end{cases}
\]
The remaining connections are then made in cyclic order.  Thus there are exactly $C(E)$ non-crossing gluings, each preserving every first leg and cycle size.

To recover the choices from $\tilde\pi$, first read its least inner position $u_1$ in a through-cycle. As above, this determines the inner choice and hence $V$. Next compare predecessors in $\tilde\pi$ and $\tilde\pi_0$: the outer choice is
\[
    e=\min\bigl\{j\in\{\bar\tau+1,\ldots,m\}\cap D(E):
       \tilde\pi^{-1}(j)\ne\tilde\pi_0^{-1}(j)\bigr\},
\]
with an empty set indicating the default choice. Indeed, every smaller outer auxiliary position retains its default predecessor, whereas the default predecessor of $e$ now points into $V$. Finally, the displayed formula recovers $e'$. Thus the data $(u_1,e,e')$ recover the permutation, including its cycle order.

To verify that these rules exhaust all annular permutations corresponding to $\pi$, first note two size constraints. Some through-cycle has its first leg in $[\bar\tau,m]$. Also, the least inner through-position is at most $\tau$, since the pure inner blocks beginning in $[m+1,\tau)$ have total size $\tau-m-1$.

First, outer auxiliary positions at or before $\bar\tau$ cannot change their default cycle, so the data $e$ of an annular permutation can be realized from the default permutation. Suppose not, let $j$ be the first position before $\bar\tau$ that changes its predecessor from a first position $b$ to position $a$. If $a$ is an earlier outer first position $a<b$,  the position $b$ still needs an auxiliary position $k$ after $j$ to attain its prescribed size. However, the connection from $a$ to $j$ blocks either the connection of $b$ and $k$ or the cycle beginning in $[\bar\tau,m]$: their first legs lie on opposite outer boundary arcs between $a$ and $j$. If $a$ belongs to the inner circle, the cycle containing $a$ and $j$ would insert a through cycle, which blocks the later through-cycle beginning in $[\bar\tau,m]$. Hence every outer auxiliary position at or before $\bar\tau$ retains its default predecessor.

Second, at the first changed outer auxiliary position $e$, its default predecessor must now point to the inner circle. By minimality of $e$ and the fixed first legs and cycle sizes, suppose the predecessor is connected to an outer position after $e$; this would cross the cycle containing $e$, or violate the outer cyclic order if it is the same cycle. Moreover, this inner image cannot be an auxiliary position $u>\tau$ of a default pure inner cycle with first leg $b$. Indeed, $\tau<b<u$, and there is already a through-connection at an inner position at most $\tau$. To avoid these two through-connections, all pure inner cycles beginning in $[b,u)$ would have to remain within $[b,u)$. Their prescribed total size exceeds $u-b$, however: in the default partition they fill this interval and also occupy $u$. This is impossible. Therefore, the data $e'$ of any annular permutation is in $V_0\cup [\{m+1,...,\tau-1\}\cap D(E)]$. The inner part of the permutation directly gives the reconnection from the default inner circle, so $e'\in V$ and it can be realized from the default permutation.

Once the least inner through-position is fixed, non-crossing forces the inner reconnection along its enclosing chain: a block off that chain stays in an intervening gap, while skipping a participating enclosing block gives alternating endpoints. The same argument, starting at $e$, forces the outer chain. The remaining cyclic gluing is uniquely determined by $e'$ as above. Reversing these steps returns $\tilde\pi_0$, so the choices give a bijection, with $B(E)$ inner choices, $A(E)$ outer choices for each of them, and $C(E)$ gluings for each pair. Their product proves Eqn.~\eqref{eq_equiv1}.

\vspace{0.5cm}

    To prove Eqn.~\eqref{eq_equiv2}, we fix a Lukasiewicz path $E\in L(m,n)$. On the left-hand side, for each choice of $s\in \{\bar{\tau},...,m\}\cap U(E)$ such that $k_1(E,s)\ge 1$ (otherwise the term contributes zero), we obtain a collection of non-crossing annular partitions $\sigma$ that correspond to $(E,s)$, by the following rules:
    \begin{itemize}
        \item The marked block $M_{1}\in\pi_1$ is of size $k_{1}(E,s)\ge1$, and its first leg is at position $s$. The marked block $M_{2}\in\pi_2$ is of size $k_{2}(E)=C(E)$.
        \item As before, the first-leg positions and sizes of the unmarked blocks $B_{1,i}$'s ($B_{2,i}$'s resp.) are determined by the occurrence times and heights of the up or flat steps in $E$ before (after resp.) time $m$, except for the one at time $s$.
    \end{itemize}

Conversely, every $\sigma=(\pi_1,M_1)\times(\pi_2,M_2)$ determines a unique pair $(E,s)$. Set $s=\min M_1$. At the first leg of each unmarked block $B$, give $E$ an increment $|B|-1$; at time $s$, give it the increment $|M_1|+|M_2|-1$; and give it a down step at every remaining position. Before time $m$, this is the Lukasiewicz path of $\pi_1$ with $|M_2|$ additional open legs from time $s$ onward. Hence it stays nonnegative, satisfies $E(m)=|M_2|$, and has $s\in\{\bar\tau,\ldots,m\}\cap U(E)$. After time $m$, comparison with the Lukasiewicz path of $\pi_2$ shows that these $|M_2|$ open legs are closed precisely at the positions of $M_2$; the path again stays nonnegative and ends at zero. Consequently $E\in L(m,n)$,
\[
    k_1(E,s)=|M_1|,\qquad k_2(E)=|M_2|.
\]
This construction is inverse to the two rules above. Thus the pairs $(E,s)$ partition all terms on the right-hand side of Eqn.~\eqref{eq_equiv2}, and it remains only to count the partitions corresponding to one fixed pair.

For fixed $(E,s)$, the outer partition of its corresponding elements in $NC(m)\times NC(n)$ is uniquely determined, so we only need to count the inner partitions. We construct a default partition by scanning the inner positions in increasing order and assigning each auxiliary leg to the most recently opened unmarked block still needing a leg. The remaining $C(E)$ positions form $M_2^{(0)}$, with $\min M_2^{(0)}=\tau$; the unmarked blocks before $\tau$ form a complete non-crossing partition.

Number these unmarked blocks by increasing first-leg position. Choose an auxiliary position $e<\tau$ in the default partition. Select $e$ and all auxiliary legs to its right but before $\tau$ whose block index is no larger than that of the block containing $e$. List the $r$ selected positions together with $M_2^{(0)}$ as
\[
    u_1<\cdots<u_{r+C(E)}.
\]
Move the selected leg at $u_i$ to $u_{C(E)+i}$, retaining its block membership, and assign $u_1,\ldots,u_{C(E)}$ to $M_2$. All other legs stay fixed. Note that overlap is allowed, and $\min M_2=e$.

\begin{example}\label{ex:second-order-walk}
Let $m\ge1$ and take the outer partition to consist of the single marked block $M_1=[m]$. Let $n=9$. Relabel the inner positions $m+j$ by $j$. Consider the default inner partition
\[
 B_{2,1}^{(0)}=\{1,2,5\},\qquad
 B_{2,2}^{(0)}=\{3,4\},\qquad
 M_2^{(0)}=\{6,7,8,9\}.
\]
Then $C(E)=4$, $\tau=m+6$, and $B(E)=4$. The three non-default choices start at the locally labeled legs $2,4,5$, selecting respectively $\{2,5\}$, $\{4,5\}$, and $\{5\}$. In particular, choosing $4$ moves legs from \emph{both} nested unmarked blocks: $4$ moves to $8$ and $5$ to $9$, giving
\[
 B_{2,1}=\{1,2,9\},\qquad B_{2,2}=\{3,8\},\qquad M_2=\{4,5,6,7\}.
\]
Figure~\ref{fig:walk-second-order-configurations} shows all four completions. Their total weight is $4\kappa_{m,4}\kappa_3\kappa_2$, as required.

\end{example}

By viewing such an operation as a map from the default partition to partitions of $n$, we have the following properties of this operation: 
\begin{itemize}
    \item It is injective. 
    \item The first-leg position of every unmarked block is unchanged.
    \item Any moved unmarked block has its last fixed leg to the left of $u_1$ and its first moved leg to the right of $u_{C(E)}$, so $M_2$ lies between two consecutive legs of that block. By the selection rule, every unmoved block and all blocks nested inside it stay fixed. Its span therefore contains neither a selected position nor a leg of $M_2^{(0)}$, and hence no leg of $M_2$. Thus it lies to the left or right of $M_2$, or entirely in a gap between consecutive marked legs. Finally, the moved blocks are nested in the default partition, and the shift preserves the order of their selected legs. Together with the fixed interiors just described, this also prevents crossings between unmarked blocks. Hence the output partition is non-crossing.
    \item Conversely, in any partition corresponding to $(E,s)$, the first leg of $M_2$  cannot exceed $\tau$, since the unmarked blocks starting before $\tau$ have total size $\tau-m-1$ and cannot also accommodate position $\tau$. Once this first leg is fixed, all block sizes and first-leg positions are fixed and they determine a non-crossing partition uniquely. Thus every target partition is either the default partition or the output of the corresponding choice $e$.
\end{itemize}

By definition, the range of the map above has $B(E)-1$ elements. This together with the input partition gives the factor $B(E)$, proving Eqn.~\eqref{eq_equiv2}.
\end{proof}


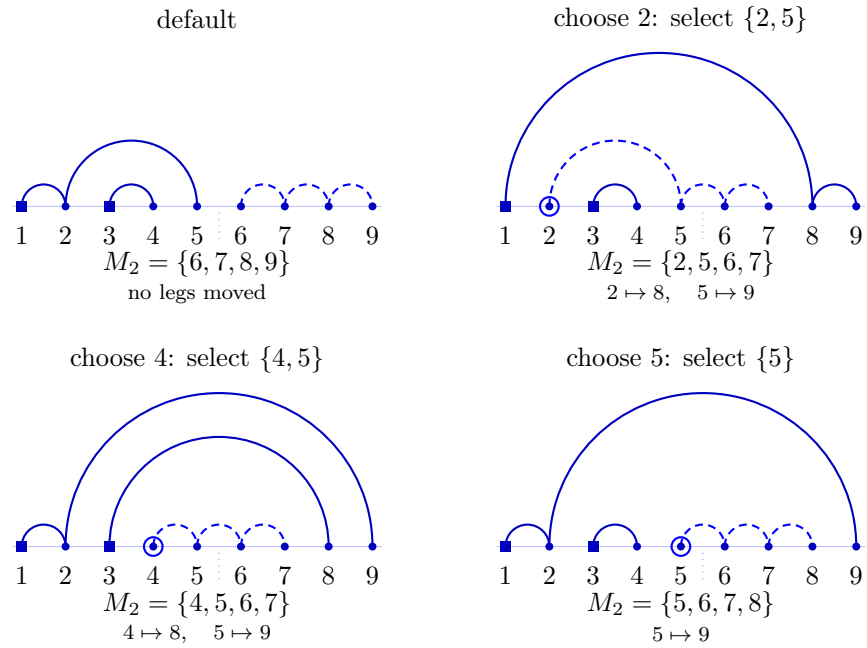
\begin{figure}[htbp]
\centering
\begin{tikzpicture}[
    every node/.style={font=\small},
    unmarked/.style={draw=blue!75!black,thick},
    marked/.style={draw=blue,thick,densely dashed}
]
\newcommand{\SecSevenNinePoints}{%
    \draw[blue!25] (-0.12,0) -- (4.76,0);
    \foreach \j in {1,...,9}{
        \coordinate (v\j) at ({0.58*(\j-1)},0);
        \node[below=4pt] at (v\j) {$\j$};
    }
    \draw[blue!45,dotted] (2.61,-0.08) -- (2.61,-0.47);
}
\newcommand{\SecSevenArc}[3]{%
    \draw[#1] ({0.58*(#2-1)},0)
        arc[start angle=180,end angle=0,radius={0.29*(#3-#2)}];
}
\newcommand{\SecSevenVertices}[1]{%
    \foreach \j in {1,...,9}{\fill[blue!75!black] (v\j) circle (1.5pt);}
    \foreach \j in {1,3}{\fill[blue!75!black] (v\j) ++(-2pt,-2pt) rectangle ++(4pt,4pt);}
    \ifnum#1>0\relax
        \draw[blue,thick] (v#1) circle (3.5pt);
    \fi
}

\begin{scope}[xshift=0cm,yshift=0cm]
    \SecSevenNinePoints
    \SecSevenArc{unmarked}{1}{2}
    \SecSevenArc{unmarked}{2}{5}
    \SecSevenArc{unmarked}{3}{4}
    \SecSevenArc{marked}{6}{7}
    \SecSevenArc{marked}{7}{8}
    \SecSevenArc{marked}{8}{9}
    \SecSevenVertices{0}
    \node at (2.32,2.48) {default};
    \node at (2.32,-0.76) {$M_2=\{6,7,8,9\}$};
    \node[font=\scriptsize] at (2.32,-1.16) {no legs moved};
\end{scope}

\begin{scope}[xshift=6.4cm,yshift=0cm]
    \SecSevenNinePoints
    \SecSevenArc{unmarked}{1}{8}
    \SecSevenArc{unmarked}{8}{9}
    \SecSevenArc{unmarked}{3}{4}
    \SecSevenArc{marked}{2}{5}
    \SecSevenArc{marked}{5}{6}
    \SecSevenArc{marked}{6}{7}
    \SecSevenVertices{2}
    \node at (2.32,2.48) {choose $2$: select $\{2,5\}$};
    \node at (2.32,-0.76) {$M_2=\{2,5,6,7\}$};
    \node[font=\scriptsize] at (2.32,-1.16) {$2\mapsto8,\quad5\mapsto9$};
\end{scope}

\begin{scope}[xshift=0cm,yshift=-4.5cm]
    \SecSevenNinePoints
    \SecSevenArc{unmarked}{1}{2}
    \SecSevenArc{unmarked}{2}{9}
    \SecSevenArc{unmarked}{3}{8}
    \SecSevenArc{marked}{4}{5}
    \SecSevenArc{marked}{5}{6}
    \SecSevenArc{marked}{6}{7}
    \SecSevenVertices{4}
    \node at (2.32,2.48) {choose $4$: select $\{4,5\}$};
    \node at (2.32,-0.76) {$M_2=\{4,5,6,7\}$};
    \node[font=\scriptsize] at (2.32,-1.16) {$4\mapsto8,\quad5\mapsto9$};
\end{scope}

\begin{scope}[xshift=6.4cm,yshift=-4.5cm]
    \SecSevenNinePoints
    \SecSevenArc{unmarked}{1}{2}
    \SecSevenArc{unmarked}{2}{9}
    \SecSevenArc{unmarked}{3}{4}
    \SecSevenArc{marked}{5}{6}
    \SecSevenArc{marked}{6}{7}
    \SecSevenArc{marked}{7}{8}
    \SecSevenVertices{5}
    \node at (2.32,2.48) {choose $5$: select $\{5\}$};
    \node at (2.32,-0.76) {$M_2=\{5,6,7,8\}$};
    \node[font=\scriptsize] at (2.32,-1.16) {$5\mapsto9$};
\end{scope}
\end{tikzpicture}
\caption{The four inner partitions in Example~\ref{ex:second-order-walk}, with local labels increasing from left to right. Arcs join consecutive legs of each block: solid arcs give the unmarked blocks, and dashed arcs give $M_2$. Squares mark the unchanged first legs $1,3$; a ring marks the chosen auxiliary position. The chosen auxiliary positions lie before $6$, the position of $\tau$. The outer partition is unchanged and omitted.}
\label{fig:walk-second-order-configurations}
\end{figure}

\section{Second-order free probability}
\subsection{High-temperature CLTs for Gaussian and Laguerre ensembles}\label{sec:gaussian-laguerre-high-temp}

Let $\theta_N\to0$ with $N\theta_N\to\gamma>0$. The $N$-particle Gaussian $\beta$-ensemble, with $\beta=2\theta_N$, is the probability measure on $a_1\le\cdots\le a_N$ with density
\begin{equation}\label{eq:gaussian-beta-density}
    \frac{1}{Z^{\mathrm G}_{N,\theta_N}}
    \prod_{1\le i<j\le N}(a_j-a_i)^{2\theta_N}
    \prod_{i=1}^N e^{-a_i^2/2}\,d\vec a.
\end{equation}
Its Bessel generating function is
\begin{equation}\label{eq:gaussian-beta-bgf}
    G^{\mathrm G}_{N,\theta_N}(\vec x)
    =\exp\left(\frac12\sum_{i=1}^N x_i^2\right)
    =\exp\left(\frac12p_2(\vec x)\right).
\end{equation}

For the Laguerre ensemble, let $M_N\ge N$ and assume that $M_N/N\to\rho\ge1$. On $0<a_1\le\cdots\le a_N$, consider the density
\begin{equation}\label{eq:laguerre-beta-density}
    \frac{1}{Z^{\mathrm L}_{N,M_N,\theta_N}}
    \prod_{1\le i<j\le N}(a_j-a_i)^{2\theta_N}
    \prod_{i=1}^N
       a_i^{\theta_N(M_N-N+1)-1}e^{-a_i}\,d\vec a.
\end{equation}
The corresponding Bessel generating function, for $\vec x$ in a neighborhood of the origin, is
\begin{equation}\label{eq:laguerre-beta-bgf}
    G^{\mathrm L}_{N,M_N,\theta_N}(\vec x)
    =\prod_{i=1}^N(1-x_i)^{-M_N\theta_N}
    =\exp\left(
       M_N\theta_N\sum_{d\ge1}\frac{p_d(\vec x)}{d}
      \right).
\end{equation}
These Bessel generating functions are standard; see \cite[Section~4.1 and Remark~4.11]{BCG}.

\begin{theorem}\label{thm:gaussian-laguerre-high-temp-clt}
Under the scaling above, both the Gaussian ensemble \eqref{eq:gaussian-beta-density} and the Laguerre ensemble \eqref{eq:laguerre-beta-density} satisfy the high-temperature CLT. Their first-order free $\gamma$-cumulants are
\[
    \kappa_d^{\mathrm G,\gamma}
    =\begin{cases}
       1,&d=2,\\
       0,&d\ne2,
     \end{cases}
    \qquad
    \kappa_d^{\mathrm L,\gamma}=\lambda,
    \qquad \lambda:=\gamma\rho,
\]
and all their second-order free $\gamma$-cumulants vanish:
\[
    \kappa_{r,s}^{\mathrm G,\gamma}
    =\kappa_{r,s}^{\mathrm L,\gamma}=0,
    \qquad r,s\ge1.
\]
Consequently, their limiting covariance matrices are obtained from \eqref{eq_hightempcov} by substituting the respective first-order free cumulants above; in both cases the $Conf_2(m,n)$ contribution is zero.
\end{theorem}

\begin{proof}
Taking logarithms in \eqref{eq:gaussian-beta-bgf} and \eqref{eq:laguerre-beta-bgf} gives
\[
    \log G^{\mathrm G}_{N,\theta_N}=\frac12p_2,
    \qquad
    \log G^{\mathrm L}_{N,M_N,\theta_N}
       =M_N\theta_N\sum_{d\ge1}\frac{p_d}{d}.
\]
Hence $c^{\mathrm G}_{(2)}=1/2$ and $c^{\mathrm G}_{(d)}=0$ for $d\ne2$, while
\[
    c^{\mathrm L}_{(d)}=\frac{M_N\theta_N}{d},
    \qquad d\ge1.
\]
For both ensembles, $c_\nu=0$ whenever $\ell(\nu)\ge2$. Since $M_N\theta_N\to\rho\gamma=\lambda$, we obtain
\[
    d\,c^{\mathrm G}_{(d)}\longrightarrow
       \mathbf 1_{\{d=2\}},
    \qquad
    d\,c^{\mathrm L}_{(d)}\longrightarrow\lambda,
\]
together with
\[
    Nr s\,c^{\mathrm G}_{(r,s)}
    =Nr s\,c^{\mathrm L}_{(r,s)}=0,
    \qquad
    N^{\ell(\nu)/2}c^{\mathrm G}_{\nu}
    =N^{\ell(\nu)/2}c^{\mathrm L}_{\nu}=0
    \quad\text{for }\ell(\nu)\ge3.
\]
Thus all assumptions of the high-temperature ``if'' theorem, Theorem~\ref{thm:hightempclt}, hold. The stated CLTs and cumulants follow immediately, and substituting them into \eqref{eq_hightempcov} gives the covariance assertion.
\end{proof}

\subsection{CLT for fixed- and high-temperature \texorpdfstring{$\beta$}{beta}-additions}

We use the notation $a+_{\theta}b$ for the $\theta$-addition of \cite[Section~6.1]{CX}, so that $\beta=2\theta$. Following \cite[Section~2.4]{BCG}, we restrict its inputs to compactly supported real symmetric generalized functions, with the normalization of Section~\ref{sec:bgf}. The addition is characterized by the multiplicativity of Bessel generating functions: if $c=a+_{\theta}b$, then
\begin{equation}\label{eq:betaadd-bgf}
    G^{c}_{N,\theta}(\vec{x})=G^{a}_{N,\theta}(\vec{x})G^{b}_{N,\theta}(\vec{x}).
\end{equation}

\begin{theorem}\label{thm:betaadditionclt}
Let $a(N)$ and $b(N)$ be two sequences of $N$-particle systems whose laws are compactly supported real symmetric generalized functions on $\mathbb R^N$, each of total mass $N!$ as in Section~\ref{sec:bgf}.

\begin{enumerate}[label=\textup{(\roman*)}]
    \item \textbf{Fixed temperature.} Fix $\theta>0$ and suppose that both $a(N)$ and $b(N)$ satisfy a fixed-temperature CLT. Denote their fixed-temperature first- and second-order free cumulants by
    \[
        \{\kappa^{a}_{d},\kappa^{a}_{r,s}\}_{d,r,s\ge 1}
        \quad\text{and}\quad
        \{\kappa^{b}_{d},\kappa^{b}_{r,s}\}_{d,r,s\ge 1},
    \]
    respectively, and set
    \[
        \kappa^c_d:=\kappa^a_d+\kappa^b_d,
        \qquad
        \kappa^c_{r,s}:=\kappa^a_{r,s}+\kappa^b_{r,s}.
    \]
    Then $c(N):=a(N)+_{\theta}b(N)$ satisfies a fixed-temperature CLT, with first- and second-order free cumulants $\kappa^c_d$ and $\kappa^c_{r,s}$. In particular, its limiting covariance is
    \begin{equation}\label{eq:betaadd-fixed-cov}
    \begin{split}
        \alpha^c_{m,n}
        ={}&\frac{1}{\theta}\sum_{E\in L(m,n)}A(E)B(E)C(E)
        \prod_{B_i\in E}\kappa^c_{|B_i|}\\
        &+\frac{1}{\theta}\sum_{E\in L(m,n)}
        \sum_{s\in\{\bar\tau,\ldots,m\}\cap U(E)}
        B(E)\kappa^c_{k_1(E),k_2(E)}
        \prod_{\substack{B_i\in E\\ t(B_i)\ne s}}
        \kappa^c_{|B_i|},
    \end{split}
    \end{equation}
    where $k_1(E)$ and $k_2(E)$ are as in Theorem~\ref{thm:fixedtempclt}.

    \item \textbf{High temperature.} Let $\theta=\theta_N\to0$ with $N\theta_N\to\gamma>0$, and suppose that both $a(N)$ and $b(N)$ satisfy a high-temperature CLT. Denote their first- and second-order free $\gamma$-cumulants by
    \[
        \{\kappa^{a,\gamma}_{d},\kappa^{a,\gamma}_{r,s}\}_{d,r,s\ge 1}
        \quad\text{and}\quad
        \{\kappa^{b,\gamma}_{d},\kappa^{b,\gamma}_{r,s}\}_{d,r,s\ge 1},
    \]
    and set
    \[
        \kappa^{c,\gamma}_d:=\kappa^{a,\gamma}_d+\kappa^{b,\gamma}_d,
        \qquad
        \kappa^{c,\gamma}_{r,s}:=\kappa^{a,\gamma}_{r,s}+\kappa^{b,\gamma}_{r,s}.
    \]
    Then $c(N):=a(N)+_{\theta_N}b(N)$ satisfies a high-temperature CLT, with first- and second-order free $\gamma$-cumulants $\kappa^{c,\gamma}_d$ and $\kappa^{c,\gamma}_{r,s}$. For $j=1,2$, let $I_k^{(j),c}(E,\vec q;\gamma)$ denote the functional $I_k^{(j)}(E,\vec q;\gamma)$ evaluated at these summed cumulants. Its limiting covariance is therefore
    \begin{equation}\label{eq:betaadd-high-cov}
    \begin{split}
        \alpha^{c,\gamma}_{m,n}
        ={}&\sum_{(E,\vec q)\in Conf_1(m,n)}
        \sum_{k=1}^{|P(E,\vec q)|}
        \gamma^{|D(E)|-k}I_k^{(1),c}(E,\vec q;\gamma)\\
        &+\sum_{(E,\vec q)\in Conf_2(m,n)}
        \sum_{k=1}^{|P(E,\vec q)|}
        \gamma^{|D(E)|-k}I_k^{(2),c}(E,\vec q;\gamma).
    \end{split}
    \end{equation}
\end{enumerate}
\end{theorem}

\begin{proof}
An analogous argument appeared in \cite{BCG} and \cite{CX}, so we only give a sketch.
For $u\in\{a,b,c\}$, let $c^{u,N}_{\lambda}$ be the coefficient of $p_{\lambda}(\vec{x})$ in $\ln G^u_{N,\theta}(\vec{x})$. Taking logarithms in Eqn.~\eqref{eq:betaadd-bgf} gives
\[
    c^{c,N}_{\lambda}=c^{a,N}_{\lambda}+c^{b,N}_{\lambda}
    \qquad\text{for every partition }\lambda.
\]
The only-if statements in Theorems~\ref{thm:intro-fixed-iff} and \ref{thm:intro-high-iff} give the required coefficient asymptotics for the two inputs. The displayed identity shows that the same asymptotics hold for $c(N)$, with the first- and second-order free cumulants replaced by their sums above; the coefficients indexed by partitions of length at least three still vanish under the required scaling. The fixed and high-temperature CLT then follow directly from the if statements in Theorems~\ref{thm:intro-fixed-iff} and \ref{thm:intro-high-iff} respectively. Substituting the summed cumulants into Eqns.~\eqref{eq_fixedtempcov} and \eqref{eq_hightempcov} yields Eqns.~\eqref{eq:betaadd-fixed-cov} and \eqref{eq:betaadd-high-cov}, respectively.
\end{proof}

\subsection{\texorpdfstring{$\gamma\rightarrow\infty$}{gamma to infinity} transition}

We conclude this section by giving a limit transition from the high-temperature covariance formula to its fixed-temperature counterpart.

\begin{theorem}[High-to-fixed-temperature transition]\label{thm:gamma-to-fixed-transition}
Fix $\theta>0$. For each $\gamma>0$, let
\[
    \{\kappa_d^\gamma\}_{d\ge1}
    \qquad\text{and}\qquad
    \{\kappa_{r,s}^\gamma\}_{r,s\ge1}
\]
be first- and second-order free $\gamma$-cumulants, respectively, and let $\alpha_{m,n}^\gamma$ be the covariance defined by the high-temperature formula \eqref{eq_hightempcov}. Similarly, let
\[
    \{\kappa_d\}_{d\ge1}
    \qquad\text{and}\qquad
    \{\kappa_{r,s}\}_{r,s\ge1}
\]
be fixed-temperature free cumulants, and denote by $\alpha_{m,n}^{\mathrm{fix},\theta}$ the covariance in \eqref{eq_fixedtempcov} evaluated at these cumulants.

Assume that, for every $d,r,s\ge1$,
\begin{equation}\label{eq:gamma-to-fixed-cumulants}
    \gamma^{d-1}\kappa_d^\gamma\longrightarrow\kappa_d,
    \qquad
    \gamma^{r+s-1}\kappa_{r,s}^\gamma\longrightarrow\kappa_{r,s}
    \qquad (\gamma\to\infty).
\end{equation}
Then, for every $m,n\ge1$,
\begin{equation}\label{eq:gamma-to-fixed-covariance}
    \lim_{\gamma\to\infty}\gamma\alpha_{m,n}^\gamma
    =\theta\alpha_{m,n}^{\mathrm{fix},\theta}.
\end{equation}
Equivalently,
\[
    \alpha_{m,n}^{\mathrm{fix},\theta}
    =\frac{1}{\theta}\lim_{\gamma\to\infty}
      \gamma\alpha_{m,n}^\gamma.
\]
\end{theorem}

\begin{proof}
The first relation in \eqref{eq:gamma-to-fixed-cumulants} is the first-order transition of \cite[Section~8.2]{BCG}; the second relation is its second-order analogue. We show that these two relations give termwise convergence of the covariance formula.

Fix $m,n$. Recall from Definition~\ref{def:walk-blocks-configuration} that
\[
    D(E)=\{t:E(t)-E(t-)=-1\}
\]
is the set of down-step times of $E$, while $U(E)$ is the set of up- and flat-step times. As in Definitions~\ref{def:functional1} and~\ref{def:functional2}, $B_i$ denotes the block attached to the nonnegative increment at time $t(B_i)\in U(E)$; in particular,
\[
    E(t(B_i))-E(t(B_i)-)=|B_i|-1.
\]
Since a Lukasiewicz path starts and ends at zero, the total height of its nonnegative increments equals the number of its down steps. Hence, for $(E,\vec q)\in Conf_1(m,n)$,
\[
    |D(E)|=\sum_{B_i\in E}(|B_i|-1).
\]
By \eqref{eq:gamma-to-fixed-cumulants}, the cumulant attached to $B_i$ satisfies
\[
    \kappa_{|B_i|}^\gamma
    =\gamma^{-(|B_i|-1)}
      \bigl(\kappa_{|B_i|}+o(1)\bigr).
\]
Thus a nonnegative increment of height $|B_i|-1$ contributes exactly $\gamma^{-(|B_i|-1)}$ through its associated cumulant, and the product over all blocks contributes $\gamma^{-|D(E)|}$.

For $(E,\vec q)\in Conf_2(m,n)$, let $s=s(E,\vec q)$ be the distinguished time. Then
\[
    |D(E)|
    =\sum_{t(B_i)\ne s}(|B_i|-1)
     +k_1(E,\vec q)+k_2(E)-1.
\]
Indeed, $k_1(E,\vec q)+k_2(E)-1=E(s)-E(s-)$ is the height of the distinguished nonnegative increment at $s$. The second-order free cumulant at $s$ contributes $\gamma^{-(k_1+k_2-1)}$, while each remaining block contributes $\gamma^{-(|B_i|-1)}$ as above; together they again contribute $\gamma^{-|D(E)|}$. Thus, for $j=1,2$,
\[
    I_k^{(j)}(E,\vec q;\gamma)
    =\gamma^{-|D(E)|}
      \bigl(I_k^{(j)}(E,\vec q)+o(1)\bigr),
\]
where the functional on the right is evaluated at the limiting fixed-temperature free cumulants.

The explicit factor $\gamma^{|D(E)|-k}$ in \eqref{eq_hightempcov} has the complementary down-step interpretation: all jump times are realized by difference operators, and choosing $k$ derivative times from $P(E,\vec q)=D(E)\setminus J(E,\vec q)$ leaves $|D(E)|-k$ difference-operator down steps, each of which supplies one factor of $\gamma$. Thus the block factors cancel $\gamma^{|D(E)|}$, and multiplication of the covariance by $\gamma$ leaves $\gamma^{1-k}$.

For fixed $m,n$, both configuration sets and all inner sums are finite. Substituting the preceding asymptotics into \eqref{eq_hightempcov} gives
\begin{align*}
    \gamma\alpha_{m,n}^\gamma
    ={}&\sum_{(E,\vec q)\in Conf_1(m,n)}
       \sum_{k=1}^{|P(E,\vec q)|}
       \gamma^{1-k}\bigl(I_k^{(1)}(E,\vec q)+o(1)\bigr)\\
    &+\sum_{(E,\vec q)\in Conf_2(m,n)}
       \sum_{k=1}^{|P(E,\vec q)|}
       \gamma^{1-k}\bigl(I_k^{(2)}(E,\vec q)+o(1)\bigr).
\end{align*}
Every term with $k\ge2$ vanishes, whereas the terms with $k=1$ converge. Consequently,
\[
    \lim_{\gamma\to\infty}\gamma\alpha_{m,n}^\gamma
    =\sum_{(E,\vec q)\in Conf_1(m,n)} I_1^{(1)}(E,\vec q)
     +\sum_{(E,\vec q)\in Conf_2(m,n)} I_1^{(2)}(E,\vec q)
    =\theta\alpha_{m,n}^{\mathrm{fix},\theta},
\]
where the last equality is \eqref{eq_fixedtempcov}. This proves \eqref{eq:gamma-to-fixed-covariance}.

Together with the additivity of first- and second-order free cumulants, this gives the covariance-level transition from second-order $\gamma$-convolution to second-order free convolution. For $\theta=1$ the transition is direct; general fixed $\theta$ differs by the factor $1/\theta$.
\end{proof}

\appendix
\begingroup

\section{Proof of the topological expansion}\label{app:topological-expansion}

For completeness, we reproduce the proof of \cite[Theorem~5.6]{CX} here, following the notation in Section \ref{sec:topexpansion}.
Consider the sequences of variables $\mathbf p=(p_1,p_2,\ldots)$,
$\mathbf q=(q_1,q_2,\ldots)$, and $(g_1,g_2,\ldots)$. Let
$J_\lambda(\mathbf p;\theta)$ be the Jack symmetric function, with the
same monic normalization as in \eqref{eq_BesselJack}, expressed in the
power-sum variables $\mathbf p$. It is a polynomial of degree
$|\lambda|$ when $p_k$ has degree $k$. Define
$J_\lambda(\mathbf q;\theta)$ analogously. Further, consider the
generating series $G(z)=1+\sum_{k\geq1}g_kz^k$. Following \cite{CD},
define the formal power series
\begin{equation}\label{eq:app-cx-tau}
\begin{split}
\widehat F^G(t,\mathbf p,\mathbf q,g_1,g_2,\ldots;\theta)
&:=\sum_{d\geq0}t^d\sum_{|\lambda|=d}
J_\lambda(\mathbf p;\theta)J_\lambda(\mathbf q;\theta)\\
&\quad\times\prod_{(i,j)\in\lambda}
\frac{(\lambda_i-j)+\theta(\lambda'_j-i)+\theta}
{(\lambda_i-j)+\theta(\lambda'_j-i)+1}
\prod_{(i,j)\in\lambda}G(c_\theta(i,j)),
\end{split}
\end{equation}
where $\lambda'_j$ is the number of boxes in column $j$ of $\lambda$
and $c_\theta(i,j):=\theta^{-1}(j-1)-(i-1)$.

The identity of Chapuy--Do\l{}\k{e}ga
\cite[Theorem~6.2]{CD}, recalled in \cite[Theorem~5.5]{CX}, states that
there exists a non-orientability statistic $\nu_\rho(\MM,c)$, taking
nonnegative integer values and vanishing precisely on orientable
constellations, such that
\begin{equation}\label{eq:app-cx-formal}
\begin{split}
\ln\widehat F^G
={}&\theta\sum_{d\geq1}\frac{t^d}{d}
\sum_{\substack{(\MM,c)\in C_{d,\bullet}^{(\infty)}\\
                 \MM\ \mathrm{normal}}}
(\theta^{-1}-1)^{\nu_\rho(\MM,c)}
p_{\pmb{\mu_1}(\MM)}q_{\pmb{\mu_2}(\MM)}
f_{\eta(\MM)}(g_1,g_2,\ldots).
\end{split}
\end{equation}
Here $p_\mu:=\prod_jp_{\mu_j}$ and $q_\mu:=\prod_jq_{\mu_j}$,
and $f_\eta$ is the polynomial specified in
Theorem~\ref{thm:keyexpansion}. Equation~\eqref{eq:app-cx-formal} is an
identity of formal power series; the analytic convergence needed below
is a separate part of the proof.

\begin{proof}[Proof of Theorem~\ref{thm:keyexpansion}]
The multivariate Bessel functions $B_a(\vec x;\theta)$ are analytic
functions of $(a,\vec x)\in\C^N\times\C^N$ and can be expanded in
terms of Jack polynomials as follows (see \cite[Sec.~4]{Ok-Ol}):
\begin{equation}\label{eq:app-cx-jack}
\begin{split}
B_a(\vec x;\theta)
={}&\sum_{\ell(\lambda)\leq N}
\frac{\prod_{(i,j)\in\lambda}
\big((\lambda_i-j)+\theta(\lambda'_j-i)+\theta\big)}
{\prod_{(i,j)\in\lambda}
\big((\lambda_i-j)+\theta(\lambda'_j-i)+1\big)
\big(N\theta+j-1-\theta(i-1)\big)}\\
&\hspace{35mm}\times J_\lambda(\vec x;\theta)J_\lambda(a;\theta).
\end{split}
\end{equation}
On the other hand, consider the power series $\widehat F^G$ obtained
from specializing to the sequence $g_k=(-N)^{-k}$, or equivalently,
from $G(z)=(1+z/N)^{-1}$, and setting $t=(N\theta)^{-1}$.
Furthermore, for $\vec x=(x_1,\ldots,x_N)\in\C^N$ and
$a=(a_1,\ldots,a_N)\in\C^N$, specify
$p_k=p_k(a)=\sum_{i=1}^Na_i^k$ and
$q_k=p_k(\vec x)=\sum_{i=1}^Nx_i^k$.
By comparing \eqref{eq:app-cx-tau} and \eqref{eq:app-cx-jack}, we have
\begin{equation}\label{eq:app-cx-specialization}
\left[
\widehat F^{(1+z/N)^{-1}}\left(
\frac{1}{N\theta},\mathbf p(a),\mathbf q(\vec x),
\{(-N)^{-k}\}_{k\geq1};\theta\right)
\right]_{\leq n}
=\big[B_a(\vec x;\theta)\big]_{\leq n}.
\end{equation}
The brackets denote the terms of total degree at most $n$ in $\vec x$.
This degree-by-degree interpretation suffices: $n<N$, so every partition
occurring here has length at most $N$, and all the specialized factors
are well-defined. Since $B_a(0^N;\theta)=1$, its logarithm has a unique
analytic branch near $\vec x=0$ with value $0$ there.

By \eqref{eq:app-cx-formal}, we have the formal expansion
\begin{equation}\label{eq:app-cx-before-euler}
\begin{split}
\ln B_a(\vec x;\theta)
={}&\sum_{d=1}^n\frac{\alpha^{d-1}}{d}
\sum_{\substack{(\MM,c)\in C_{d,\bullet}^{(\infty)}\\
                 \MM\ \mathrm{normal}}}
(\alpha-1)^{\nu_\rho(\MM,c)}
\prod_{f\in F(\MM)}\frac{p_{\deg(f)}(a/N)}{N}
\prod_{v\in V_0(\MM)}\frac{p_{\deg(v)}(\vec x)}{N}\\
&\quad\times
N^{\ell(\pmb{\mu_1}(\MM))+\ell(\pmb{\mu_2}(\MM))-|\eta(\MM)|}
f_{\eta(\MM)}\big(\{(-1)^{-k}\}_{k\geq1}\big)
+O(\|x\|^{n+1}).
\end{split}
\end{equation}
Here we have used weighted homogeneity:
$f_\eta(\{(-N)^{-k}\}_{k\geq1})
=N^{-|\eta|}f_\eta(\{(-1)^k\}_{k\geq1})$.
By Proposition~\ref{prop:euler}, the exponent of $N$ in the second
line of \eqref{eq:app-cx-before-euler} is $2-2g(\MM)$.
For a fixed $g(\MM)=g$, we have
\[
\begin{split}
\ell(\eta(\MM))\leq|\eta(\MM)|
&=\ell(\pmb{\mu_1}(\MM))+\ell(\pmb{\mu_2}(\MM))+2g-2\\
&<|\pmb{\mu_1}(\MM)|+|\pmb{\mu_2}(\MM)|+2g
=2d+2g\leq2n+2g.
\end{split}
\]
As a result, if we set $k:=2n+2g$, then $\MM$ can be regarded as a
$k$-constellation. In addition, by \cite[Theorem~3.1]{BD22}, the total
number of $k$-constellations of size $d$ is finite. We conclude that
\eqref{eq:app-cx-before-euler} can be rewritten as
\eqref{eq_keyexpansion}, with a finite inner sum for every fixed
$g\in\frac12\Z_{\geq0}$.

It remains to check that when $N>\max(\alpha,1)n$, the sum over
$g\in\frac12\Z_{\geq0}$ in \eqref{eq_keyexpansion} converges absolutely.
Using \eqref{eq:app-cx-jack} and the fact that
$B_a(0^N;\theta)=1$, one can expand $\ln B_a(\vec x;\theta)$ in
terms of products of the form
$\prod_{j=1}^mJ_{\lambda_j}(a;\theta)J_{\lambda_j}(\vec x;\theta)$,
for partitions $\lambda_1,\ldots,\lambda_m$, and subsequently in terms
of products $p_{\mu_1}(a)p_{\mu_2}(\vec x)$, where
$|\mu_1|=|\mu_2|$. Through degree $n$, these are finite expansions.
Their coefficients are rational functions of $N$ that admit a Taylor
expansion at $N=\infty$ when $N>\max(\alpha,1)n$.
Indeed, the $N$-dependent denominator factors, up to the constant
$\theta$, are $N+c_\theta(i,j)$, and for $|\lambda|\leq n$,
$|c_\theta(i,j)|<\max(\alpha,1)n$.
By \eqref{eq:app-cx-specialization}, the expansion of the coefficient of
$p_{\mu_1}(a)p_{\mu_2}(\vec x)$, for $|\mu_1|=|\mu_2|=d$, agrees with
\begin{equation}\label{eq:app-cx-coefficient}
\frac{\alpha^{d-1}}{d}
\sum_{g\in\frac12\Z_{\geq0}}
\sum_{\substack{(\MM,c)\in C_{d,\bullet}^{(\infty)},\ \MM\ \mathrm{normal}\\
g(\MM)=g,\ \pmb{\mu_1}(\MM)=\mu_1,\ \pmb{\mu_2}(\MM)=\mu_2}}
(\alpha-1)^{\nu_\rho(\MM,c)}
N^{2-2g-\ell(\mu_1)-\ell(\mu_2)-d}
f_{\eta(\MM)}\big(\{(-1)^{-k}\}_{k\geq1}\big).
\end{equation}
Hence this genus expansion converges absolutely as well, with the
finite sum at each fixed genus performed first. Since there are only
finitely many pairs of partitions of degrees at most $n$, this proves
the asserted expansion.
\end{proof}
\endgroup

\end{document}